\documentclass[reqno,11pt]{amsart}
\usepackage{amsfonts,amsmath,amssymb}

\usepackage{hyperref}

\usepackage[margin=1.33in]{geometry}
\numberwithin{equation}{section}

\def\e{\varepsilon}

\def\eps{\epsilon}

\def\R{{\bf R}}

\def\R{{\mathbb R}}

\def\C{{\mathbb C}}

\def\f12{{\frac 1 2}}

\def\curl{{\mbox curl\,}}

\def\one{{\widetilde{\mathbf{1}}}}

\def\f{\widetilde{f}}

\def\bar{\overline}
\def\what{\widehat}
\def\R{\mathbb{R}}

\begin{document}

\newtheorem{theo}{Theorem}[section]
\newtheorem{pro}[theo]{Proposition}
\newtheorem{lem}[theo]{Lemma}
\newtheorem{defin}[theo]{Definition}
\newtheorem{rem}[theo]{Remark}
\newtheorem{cor}[theo]{Corollary}

\title[A model for 2D capillary water waves]{Global analysis of a model for capillary\\ water waves in 2D}

\thanks{The first author was partially supported by a Packard Fellowship and NSF Grant DMS 1265818.
The second author was partially supported by a Simons Postdoctoral Fellowship and NSF Grant DMS 1265875.}

\author{Alexandru D. Ionescu}\address{Princeton University}\email{aionescu@math.princeton.edu}

\author{Fabio Pusateri}\address{Princeton University}\email{fabiop@math.princeton.edu}

\begin{abstract}
In this paper we prove a global regularity result for a quadratic quasilinear model associated to the water waves system
with surface tension and no gravity in dimension two (the capillary waves system). The model we consider here retains most of the
difficulties of the full capillary
water waves system, including the delicate time-resonance structure and modified scattering. It is slightly simpler, however,
at the technical level and our goal here is to present our method in this simplified situation.
The analysis of the full system is presented in \cite{IoPuFull}.
\end{abstract}

\maketitle

\setcounter{tocdepth}{1}
\tableofcontents

\section{Introduction}
In this paper we present a global existence result for small solutions
of the Cauchy problem associated to a quasilinear fractional Schr\"{o}dinger equation with a quadratic nonlinearity in one spatial dimension.
Our main motivation comes from the study of the two-dimensional irrotational water waves system under the influence of surface tension and no gravity.

\subsection{Free boundary Euler equations and capillary waves}\label{secWW}

The evolution of an inviscid perfect fluid that occupies a domain $\Omega_t \subset \R^n$, for $n \geq 2$, at time $t \in \R$,
is described by the free boundary incompressible Euler equations.
If $v$ and $p$ denote the velocity and the pressure of the fluid (with constant density equal to $1$)
at time $t$ and position $x \in \Omega_t$, these equations are
\begin{equation}
\label{E}
\left\{
\begin{array}{ll}
v_t + v \cdot \nabla v = - \nabla p - g e_n  &   x \in \Omega_t
\\
\nabla \cdot v = 0    &   x \in \Omega_t
\\
v (x,0) = v_0 (x)     &   x \in \Omega_0 \, ,
\end{array}
\right.
\end{equation}
where $g$ is the gravitational constant.
The free surface $S_t := \partial \Omega_t$ moves with the normal component of the velocity according to the kinematic boundary condition:
\begin{subequations}
\label{BC}
 \begin{equation}
\label{BC1}
\partial_t + v \cdot \nabla  \,\, \mbox{is tangent to} \,\, \bigcup_t S_t \subset \R^{n+1} \, .
\end{equation}
In the presence of surface tension the pressure on the interface is given by
\begin{equation}
\label{BC2}
p (x,t) = \sigma \kappa(x,t)  \,\, , \,\,\, x \in S_t  \, ,
\end{equation}
where $\kappa$ is the mean-curvature of $S_t$ and $\sigma > 0$.
\end{subequations}

In the case of irrotational flows, i.e. $\rm{\curl} v = 0$,
one can reduce \eqref{E}-\eqref{BC} to a system on the boundary.
Such a reduction can be performed identically regardless of the number of spatial dimensions,
but here we only focus on the two dimensional case. Assume also that $\Omega_t \subset \R^2$ is the region below the graph of a function $h : \R_x \times \R_t \rightarrow \R$,
that is $\Omega_t = \{ (x,y) \in \R^2 \, : y \leq h(x,t) \}$ and $S_t = \{ (x,y) : y = h(x,t) \}$.

Let us denote by $\Phi$ the velocity potential: $\nabla \Phi(x,y,t) = v (x,y,t)$, for $(x,y) \in \Omega_t$.
If $\phi(x,t) := \Phi (x, h(x,t),t)$ is the restriction of $\Phi$ to the boundary $S_t$,
the equations of motion reduce to the following system for the unknowns $h, \phi : \R_x \times \R_t \rightarrow \R$:
\begin{equation}
\label{WWE}
\left\{
\begin{array}{l}
\partial_t h = G(h) \phi
\\
\partial_t \phi = - g h + \sigma  \dfrac{\partial_x^2 h}{ (1+h_x^2)^{3/2} }
  - \dfrac{1}{2} {|\phi_x|}^2 + \dfrac{{\left( G(h)\phi + h_x \phi_x \right)}^2}{2(1+{|h_x|}^2)}
\end{array}
\right.
\end{equation}
with
\begin{equation}
\label{defG0}
G(h) := \sqrt{1+{|h_x|}^2} \mathcal{N}(h)
\end{equation}
where $\mathcal{N}(h)$ is the Dirichlet--Neumann map associated to the domain $\Omega_t$.
We refer to  \cite[chap. 11]{SulemBook} or \cite{CraSul} for the derivation of the equations \eqref{WWE}.
This system describes the evolution of an incompressible perfect fluid of infinite depth and infinite extent, with a free moving one-dimensional surface, and a pressure boundary condition given by the Young-Laplace equation.

\subsubsection{Previous results}
The system \eqref{E}-\eqref{BC} has been under very active investigation in recent years.
Without trying to be exhaustive, we mention the early works on
the wellposedness of the Cauchy problem in the irrotational case and with gravity by Nalimov \cite{Nalimov}, Yosihara \cite{Yosi}, and Craig \cite{CraigLim};
the first works on the wellposedness for general data in Sobolev spaces (also for irrotational gravity waves) by Wu \cite{Wu1,Wu2};
and subsequent works on the gravity problem by Christodoulou and Lindblad \cite{CL}, Lannes \cite{Lannes}, Lindblad \cite{Lindblad},
Coutand and Shkoller \cite{CS2}, Shatah and Zeng \cite{ShZ1,ShZ3}, and Alazard, Burq and Zuily \cite{ABZ2,ABZ3}.
Surface tension effects have been considered in the works of Beyer and Gunther \cite{BG}, Ambrose and Masmoudi \cite{AM},
Coutand and Shkoller \cite{CS2}, Shatah and Zeng \cite{ShZ1,ShZ3}, Christianson, Hur and Staffilani \cite{CHS},
and Alazard, Burq and Zuily \cite{ABZ1}.
Recently, some blow-up scenarios have also been investigated \cite{CCFGG,CSSplash,IFL}.

The question of long time regularity of solutions with irrotational, small and localized initial data was also addressed in a few works,
starting with \cite{WuAG}, where Wu showed almost-global regularity for the gravity problem in two dimensions ($1$d interfaces).
Subsequently, Germain, Masmoudi and Shatah \cite{GMS2} and Wu \cite{Wu3DWW}
proved global regularity for gravity waves in three dimensions ($2$d interfaces).
Global regularity in $3$d was also proven in the case of surface tension and no gravity by Germain, Masmoudi and Shatah \cite{GMSC}.
Global regularity for the gravity water waves system in dimension $2$ has been proved by the authors in\footnote{We refer the reader to \cite{IoPu1}
for the analysis of a simplified model (a fractional cubic Schr\"odinger equations),
and to \cite{IoPunote} for some additional details about the asymptotic behavior of the solutions constructed in \cite{IoPu2}.}
\cite{IoPu2} and, independently, by Alazard and Delort \cite{ADa,ADb}.
More recently a different proof of Wu's $2$d almost-global regularity result was given by Hunter, Ifrim and Tataru \cite{HIT},
and later complemented to a proof of global regularity by Ifrim and Tataru in \cite{IT}.

\subsubsection{Paralinearization}\label{secpar}
In this section we describe the paralinearization of the system \eqref{WWE} with $g=0$,
following the work of Alazard, Burq and Zuily \cite{ABZ1} on the local existence theory.
We refer the reader to \cite{AlMet1,ABZ1,ABZ2} for further details about this procedure.

Assume $(h,\phi) \in C(I:H^{N+1}(\R)\times H^{N+1/2}(\R) )$ is a real-valued solution of \eqref{WWE}--\eqref{defG0} with $g=0$.
Let
\begin{equation}\label{par1}
b := \frac{G(h)\phi + h_x\phi_x}{(1+h_x^2)} , \qquad v := \phi_x-bh_x,\qquad\omega := \phi - T_b h,
\end{equation}
where, for any $a, b \in L^2(\R)$, we have denoted by $T_a b$ the operator
\begin{equation}
\label{Tab}
\begin{split}
\mathcal{F}(T_a b)(\xi) & := \frac{1}{2\pi} \int_{\R} \what{a}(\xi-\eta) \what{b}(\eta) \chi_0(\xi-\eta,\eta)\,d\eta,
\\
\chi_0(x,y) & := \sum_{k\geq 0} \varphi_k(y) \varphi_{\leq k-10}(x).
\end{split}
\end{equation}
The function $\omega$ in \eqref{par1} is the so-called ``good-unknown'' of Alinhac.
Its relevance is apparent from the paralinearization formula for the Dirichlet-Neumann operator, which in the one-dimensional case is
\begin{equation}
\label{Gh}
G(h) \phi = |\partial_x|\omega - \partial_x T_v h + \widetilde{R} ,
\end{equation}
where $\widetilde{R}$ denotes semilinear quadratic and higher order terms.
Using \eqref{Gh} and symmetrizing\footnote{The symmetrization
procedure in 2D is simpler than in the general case considered in \cite{ABZ1}.}
the system expressed in the natural energy variables $(|\partial_x|h, |\partial_x|^{1/2}\omega)$,
one can reduce to the equivalent scalar equation
\begin{align}
\label{WWCpar}
\partial_t u + i \Lambda = - T_v \partial_x u + i \Sigma(u) + R , \qquad \Lambda(\xi) = |\xi|^{3/2} ,
\end{align}
for a complex-valued unknown
\begin{align}
\label{par5}
 u \sim i|\partial_x|h + |\partial_x|^{1/2}\omega .
\end{align}
Here $v$ is given in \eqref{par1},
$i\Sigma$ can be thought of as a symmetric differential operator of order $3/2$ which is cubic in $u$,
and $R$ consists of semilinear quadratic and higher order terms.

\subsection{Description of the model}\label{secmodel} It this paper we consider a simplification of the full equation \eqref{WWCpar}.
More precisely, let $\varphi: \R\to[0,1]$ be an even smooth function supported in $[-8/5,8/5]$ and equal to $1$ in $[-5/4,5/4]$.
For any $k \in \mathbb{Z}$ let
\begin{equation}
\label{pr1}
\varphi_k(x) := \varphi(x/2^k) - \varphi(x/2^{k-1}) , \qquad \varphi_{\leq k}(x) := \varphi(x/2^k),
  \qquad\varphi_{\geq k}(x) := 1-\varphi(x/2^{k-1}).
\end{equation}
We define a high-low cutoff function $\chi$, and an interaction symbol $q_0$, by
\begin{align}\label{pr2}
\begin{split}
\chi(x,y) & := \sum_{k\geq-10} \varphi_k(y) \varphi_{\leq k+10}(x) \varphi_{\geq k-10}(x+y),
\\
q_0(\xi,\eta) & := i\chi(\xi-\eta,\eta) |\xi-\eta|^{1/2} \eta=i|\xi-\eta|^{1/2} \eta\sum_{k\geq-10} \varphi_k(\eta) \varphi_{\leq k+10}(\xi-\eta) \varphi_{\geq k-10}(\xi).
\end{split}
\end{align}

We are interested in solutions $u: \R \times [0,T] \to \C$, $T\geq 1$, of the quasilinear equation
\begin{equation}\label{pr3}
\partial_t u + i\Lambda u = \mathcal{N} , \qquad \Lambda(\xi) := |\xi|^{3/2} ,
\end{equation}
where the quadratic nonlinearity $\mathcal{N} = \mathcal{N}(u,\overline{u})$ is defined by
\begin{align}\label{pr4}
 \begin{split}
\what{\mathcal{N}}(\xi,t) & := \frac{1}{2\pi} \int_{\R} q_0(\xi,\eta) \what{V}(\xi-\eta,t) \what{u}(\eta,t) \,d\eta,
\qquad
V := u + \bar{u},
\end{split}
\end{align}
where $\widehat{f}(\xi,t)=\mathcal{F} f(\xi,t)$ denotes the spatial Fourier transform.

Our goal is to prove global regularity for small initial data for this model, see Theorem \ref{maintheo} below. As we explain below, we consider this model for two reasons: (1) it captures the essential difficulties of the global theory of the full capillary water wave system in 2D, and (2) at the technical level it is simpler than the full system and may be of potential interest to people who are interested in normal forms in quasilinear problems, with singularities, but who are not familiar with some of the more complicated features of the water waves system. More precisely:

\setlength{\leftmargini}{1.5em}
\begin{itemize}
\item Comparing \eqref{WWCpar} with \eqref{pr2}--\eqref{pr4} we see that the equations have the same linear part and the same quadratic quasilinear structure, with one derivative on $u$ and half-derivative on
$u+\bar{u}$. Low frequencies are present through the linear flow of the initial data, and the main issue is to understand the interaction of these low frequencies with the very high frequencies. We also allow quadratic interactions between frequencies of comparable size,
in order to include all the possible resonant interactions and observe modified scattering, just like for the full system.

\medskip
\item The full capillary water wave system with no momentum conditions on the Hamiltonian variable
is a long and technical project, which is completed in \cite{IoPuFull}.
Much of the length, however, is due to the complicated nature of the problem, including complex local theory,
higher order terms, and many inessential interactions. The model we discuss here is simpler, but still captures the main new ingredients.
\end{itemize}

\medskip
For simplicity we disregard all cubic interactions and some semilinear quadratic interactions present in \eqref{WWCpar}, for example those with very low frequency output. The very low frequencies evolve essentially linearly in the full model \eqref{WWCpar}, due to suitable null structures at the origin in the frequency space; in our model we keep the very low frequencies coming from the linear flow, but disregard the nonlinear contributions at very low frequencies.

The results of this paper were presented at the Banff conference ``Dynamics in Geometric Dispersive Equations and the Effects of Trapping,
Scattering and Weak Turbulence'' in May 2014. 
Independently, Ifrim--Tataru \cite{ITca} proved global regularity of the capillary water wave system,
in holomorphic coordinates, in the case of data with one momentum condition on the Hamiltonian variables.

\subsection{Main result and ideas of the proof}\label{sectheo}

\subsubsection{The main result}
This is a precise statement of our result concerning the global regularity of \eqref{pr2}--\eqref{pr4}:

\begin{theo}\label{maintheo}
Assume that $N:=10$, $0 < p_1 \leq p_0 \leq 10^{-6}$ are fixed, and the initial data $u_0 \in H^{N}(\R)$ satisfies the assumptions
\begin{equation}\label{mainhyp}
{\|u_0\|}_{H^N} + {\|x\partial_x u_0\|}_{H^2} +
  \sup_{k\leq 0} \big[ 2^{-k(1/2-p_1)} {\|P_ku_0\|}_{L^2} + 2^{-k(1/2-p_1)} {\|P_k (x\partial_x)u_0\|}_{L^2} \big] 
= \eps_0 \leq \overline{\eps},
\end{equation}
for some constant $\overline{\eps}$ sufficiently small.
Then there exists a unique global solution $u\in C([0,\infty):H^{N}(\R))$ of the initial-value problem
\begin{equation}
\partial_t u + i\Lambda u = \mathcal{N}, \qquad u(0)=u_0,
\end{equation}
where $\mathcal{N}$ is as in \eqref{pr4}. In addition, with $S := (3/2)t\partial_t+x\partial_x$, we have the global bounds
\begin{equation}\label{mainconcl1}
\sup_{t\in[0,\infty)} \big[(1+t)^{-p_0} {\|u(t)\|}_{H^N} + (1+t)^{-4p_0} {\|Su(t)\|}_{H^2}+(1+t)^{1/2}\|u(t)\|_{W^{3,\infty}}
  \big] \lesssim \eps_0 .
\end{equation}
Furthermore the solution possesses modified scattering behavior as $t\to\infty$.
\end{theo}

\begin{rem}
The solutions can also be defined on the time interval $(-\infty,0]$ since the equations are time-reversible.
\end{rem}

\begin{rem}
A more precise statement of modified scattering can be found in Lemma \ref{Zcontrol},
(in particular in \eqref{nf50}, \eqref{nf51} and \eqref{nf56}),
where modified scattering is described in terms of an appropriately renormalized profile $f$, see \eqref{nf20}, \eqref{nf3}, \eqref{nf4},
satisfying a cubic equation, see \eqref{nf40}--\eqref{nf42}.
Also, more precise bounds on the solution can be found in section \ref{NormalForm}, see for example Lemma \ref{Lem1},
Remark \ref{extra1}, and Proposition \ref{Sup3}. 
The bounds proven in section \ref{NormalForm} can also be used to obtain asymptotics for the solution $u$ in the physical space, 
as in \cite{IoPunote}.
\end{rem}

\begin{rem}\label{LowFreq}

The space of solutions described in \eqref{mainhyp} is important. We need to assume, of course, enough Sobolev regularity
to construct strong solutions. The more important issue, however, is the assumption on low frequencies.
We regard $u$ as the basic variable, since it corresponds to the energy variables $(Dh,D^{1/2}\omega)$ coming from the main system \eqref{WWE}.
Since the normal form transformation introduces singularities at low frequencies we need a stronger assumption at low frequencies
than the smallness of the $L^2$ norm of $u$. At the same time, we would like to avoid assuming a momentum condition on $u$,
which would correspond to the assumption $\|D^{-1/2}u\|_{L^2}<\infty$ in 1 dimension.

Our choice in \eqref{mainhyp} accomplishes both of these goals. On one hand
it is strong enough to allow us to control the singular terms arising from the normal form transformation. On the other hand, it is
weak enough to avoid making a momentum assumption on $u$.
\end{rem}

The equation \eqref{pr2}--\eqref{pr4}, like the full water waves system, is a time reversible quasilinear equation.
Proving global regularity relies on two main steps:

\setlength{\leftmargini}{1.5em}
\begin{itemize}
  \item Propagate control of high frequencies (high order Sobolev norms);
  \item Prove pointwise decay of the solution over time.
\end{itemize}

In this paper we use a combination of improved energy estimates and asymptotic analysis to achieve these two goals.
We carry out both parts of our analysis in the Fourier space.
More precisely, we construct high order "modified energies" which can be controlled for long times.
Then we prove sharp time decay rates.
We describe these two main aspects of our paper below.

\subsubsection{Energy estimates and the quartic energy identity}
In our model problem \eqref{pr2}--\eqref{pr4} one can use the basic high order energy
functional $E_N(t) = {\| u(t) \|}_{H^N}^2$ to construct smooth local solutions.
In order to go past the local existence time, one needs to rely on the dispersive properties of solutions.
One of the main difficulties in dealing with a one dimensional problem such as \eqref{pr2}--\eqref{pr4} comes from the slow pointwise decay,
which is $t^{-1/2}$ for linear solutions.
To overcome this issue one would like to find an energy functional $\mathcal{E}_N(t)$ such that
$\mathcal{E}_N(t) \approx \|u(t)\|^2_{H^N} + \|Su(t)\|^2_{H^2}$, and which satisfies a {\it quartic energy identity} of the form
\begin{align}
\label{qin}
\partial_t \mathcal{E}_N(t) = \mbox{semilinear quartic terms}.
\end{align}

At least formally, this is related to the classical concept of normal forms of Shatah \cite{shatahKGE}.
The implementation of the method of normal forms however is delicate in quasilinear problems,
due to the potential loss of derivatives.
It can be done in some cases, for example either by using carefully constructed
nonlinear changes of variables (as in Wu \cite{WuAG}), or the ``iterated energy method'' of Germain--Masmoudi \cite{GM}, or
the ``paradifferential normal form method'' of Alazard--Delort \cite{ADb},
or the ``modified energy method'' of Hunter--Ifrim--Tataru \cite{HIT}. 

Quartic energy identities such as \eqref{qin} were proved and played a crucial role
in all the long-term existence results for water waves in 2d in \cite{WuAG,IoPu2,ADb,HIT,IT}. 
The main ingredient for such an identity to hold is, essentially, 
the absence of time-resonant bilinear interactions, 
and the methods described above are largely interchangeable as long as this ingredient 
is present.{\footnote{See also \cite{Delo} and \cite{HITW} for earlier constructions proving quartic energy identities 
like \eqref{qin} in simpler models.}} 
In \cite{IoPu2} we proved a quartic energy identity, suitable for global-in-time analysis, 
by adapting the change of variable of Wu \cite{WuAG}. Here, as a starting point of our analysis, 
we adapt the more elegant and robust approach of Alazard--Delort \cite{ADb} and Hunter--Ifrim--Tataru \cite{HIT}.

For our model \eqref{pr2}--\eqref{pr4} a normal form transformation is formally available
since the only time resonances, i.e. solutions of
\begin{align}
\label{phase1}
\Lambda(\xi) \pm \Lambda(\xi-\eta) \pm \Lambda(\eta) = 0,
\end{align}
occur when one of the three interacting frequencies $(\xi,\xi-\eta,\eta)$ is zero.
However, the superlinear dispersion relation $\Lambda(\xi) = |\xi|^{3/2}$ makes these resonances very strong.
For example, in the case $|\xi| \approx |\eta| \approx 1 \gg |\xi-\eta|$, we see that
\begin{align}
\label{phase2}
| \Lambda(\xi) \pm \Lambda(\xi-\eta) - \Lambda(\eta) | \approx |\xi-\eta| .
\end{align}
Because of this, a standard normal transformation which eliminates the quadratic terms in the equation
will introduce a low frequency singularity. 
For comparison, in the gravity water waves case the dispersion relation is
$\Lambda(\xi)=|\xi|^{1/2}$ and one has the less singular behavior $|\Lambda(\xi) \pm \Lambda(\xi-\eta) - \Lambda(\eta) | \approx |\xi-\eta|^{1/2}$
in the case $|\xi| \approx |\eta| \approx 1 \gg |\xi-\eta|$.

In order to deal with the issue of slow decay and strong time resonances,
we construct here a modified energy functional, in the Fourier space,
 in a way that is similar to the I-method of Colliander--Keel--Staffilani--Takaoka--Tao \cite{CKSTT1,CKSTT2}.
The use of Fourier analysis gives us a lot of flexibility in defining energy functionals and isolating the most singular contributions,
which can be expressed  as multilinear paraproducts with singular multipliers.

We sketch below some of the steps in our proof of the energy estimates.
We begin our analysis  by looking at the basic energy functional expressed in time-frequency space,
\begin{align}
 \label{E2N}
E(t) = \frac{1}{2\pi} \int_{\R} \big( 1 + {|\xi|}^2 \big)^N {|\what{u}(\xi,t)|}^2 \, d\xi +
\frac{1}{2\pi} \int_{\R} \big( 1 + {|\xi|}^{2} \big)^2 {|\what{S u}(\xi,t)|}^2 \, d\xi,
\end{align}
and calculate its evolution using the equation.
After appropriate symmetrizations that avoid losses of derivatives we obtain
\begin{align}
\label{d_tE2}
\partial_t E(t) = \mbox{semilinear cubic terms}.
\end{align}
We then define a cubic energy functional $E^{(3)}$ obtained by dividing the symbols
in the cubic expressions in \eqref{d_tE2} by the appropriate resonant phase function.
This cubic functional is a perturbation of \eqref{E2N} on each fixed time slice, and by construction we have
\begin{align}
\label{d_tE2+E3}
\partial_t \big( E + E^{(3)} \big)(t) = \mbox{singular semilinear quartic terms}.
\end{align}
Notice that this is a singular version of the desired identity \eqref{qin}.
However the singularities 
- which are ultimately due to the lack of symmetries in the equation for $Su$ - are only at zero frequencies,
and we can control them thanks to our low frequencies assumptions \eqref{mainhyp}.

Such a construction, involving quadratic energy functionals and higher order corrections,  was performed earlier in \cite{HIT} in the case of gravity
water waves. In the gravity case, however, the problem is simpler because there are no singularities in the resulting quartic integrals, and the entire construction in \cite{HIT} was performed in the physical space. 
In our case there are strong time resonances
which lead to singularities of the form $(\text{low frequency})^{-1/2}$ in the quartic integrals,
and seem to prevent analysis in the physical space.
Because of this it is important to construct energy functionals in Fourier space (in the spirit of the I-method),
and make careful assumptions on the low frequency structure of solutions, see \eqref{mainhyp} and Remark \ref{LowFreq}.


\subsubsection{Decay and modified scattering}
Having established the $L^2$ bounds described above, we then move on to proving sharp $t^{-1/2}$ pointwise decay of solutions
in sections \ref{sup} and \ref{TechProof}.
We write Duhamel's formula in Fourier space and study the nonlinear oscillations in the spirit of \cite{GMS2,GNT1,IoPu1}.
After a normal form transformation, a stationary phase analysis reveals that a correction to the asymptotic behavior\footnote{Other
examples of dispersive PDEs whose solutions exhibit a behavior which is qualitatively different from the behavior
of a linear solution include the nonlinear Sch\"rodinger and Hartree equations \cite{HN,KP,IoPu1},
the Klein-Gordon equation \cite{DelortKG1d,BosonStar}, and the gravity water waves system \cite{IoPu1,IoPu2,ADa,IT}.
We refer the reader to \cite{BosonStar,IoPu2} for more references on related works on modified scattering
for other 1-dimensional integrable models such as KdV, mKdV and Benjamin-Ono.}
is needed, similarly to our previous works on gravity waves \cite{IoPu1,IoPu2}.
However, the analysis here is more complicated because of the singularities introduced by the quadratic time resonances
.

This is especially evident in this part of the argument where non $L^2$ based norms need to be estimated, and
meaningful symmetrization cannot be performed.
These singularities need to be dealt with at all stages of the argument:
bounding the normal form that recasts the quadratic nonlinearity into a cubic one (subsection \ref{NormalForm});
controlling all cubic terms once the main asymptotic contribution is factored out (subsection \ref{prooftech1});
estimating the quartic terms arising from the renormalization of the cubic equation needed to correct
the asymptotic behavior (subsection \ref{prooftech3}).
Moreover, the convex dispersion relation creates additional cubic resonant interactions (subsection \ref{prooftech2})
which are not present in the case of a concave relation, such as in gravity water waves.
Eventually, we are able to control uniformly, over time and frequencies, an appropriate norm of our solution,
and obtain the necessary decay through an improved linear estimate, as well as modified scattering.

\subsubsection{The bootstrap}
The existence and uniqueness of local-in-time solutions for \eqref{pr2}--\eqref{pr4}
can be proved in a standard fashion.
Our global solutions are then constructed by a bootstrap argument which allows us to continue the local solutions.
More precisely, we assume that $u$ satisfies the bootstrap assumptions
\begin{equation}
\label{pr5}
 \sup_{t\in[0,T]} \big[(1+t)^{-p_0} {\|u(t)\|}_{H^N}+(1+t)^{-4p_0} {\|Su(t)\|}_{H^{2}} \big] \leq \eps_1,
\end{equation}
and
\begin{equation}
\label{pr6}
 \sup_{k \in \mathbb{Z}} \sup_{t\in[0,T]} (1+t)^{1/2} 2^{4\max(k,0)} {\|P_ku(t)\|}_{L^\infty} \leq \eps_1,
\end{equation}
and the initial data assumptions \eqref{mainhyp}, that is
\begin{equation}
\label{pr7}
{\|u_0\|}_{H^N} + {\|x\partial_x u_0\|}_{H^{2}} +
  \sup_{k\leq 0} \big[ 2^{-k(1/2-p_1)} {\|P_ku_0\|}_{L^2} + 2^{-k(1/2-p_1)} {\|P_k Su_0\|}_{L^2} \big] \leq \eps_0 .
\end{equation}
Here, as in the statement of Theorem \ref{maintheo},
\begin{equation}
\label{pr8}
\begin{split}
& S = x\partial_x + (3/2)t\partial_t,\quad N = 10, \quad 0 < p_1 \leq p_0 \leq 10^{-6} \, , \quad 0 < \eps_0 \ll \eps_1 \leq \eps_0^{2/3} \ll 1.
\end{split}
\end{equation}

We then aim to show that the a priori assumptions \eqref{pr5}--\eqref{pr6} can be improved to
\begin{equation}
\label{pr5impr}
 \sup_{t\in[0,T]} \big[ (1+t)^{-p_0} {\|u(t)\|}_{H^N} + (1+t)^{-4p_0} {\|Su(t)\|}_{H^{2}} \big] \lesssim \eps_0 ,
\end{equation}
and
\begin{equation}
\label{pr6impr}
 \sup_{k\in\mathbb{Z}} \sup_{t\in[0,T]} (1+t)^{1/2}2^{4\max(k,0)} {\|P_ku(t)\|}_{L^\infty} \lesssim \eps_0 .
\end{equation}
This is done in the three main steps given by Propositions \ref{EnEst1}, \ref{EnEst2} and \ref{Sup1}.

In section \ref{energy} we improve the control on the $H^N$ norm from \eqref{pr5} to \eqref{pr5impr}
by implementing the energy construction described above.
For this we only need the first a priori assumption in \eqref{pr5} and the decay assumption \eqref{pr6}.
In section \ref{energy2} we improve the control on the weighted norm, using all the a priori assumptions \eqref{pr5}--\eqref{pr7}.
In section \ref{sup} we begin our proof of the decay estimate \eqref{pr6impr}.
We transform the unknown $u$ to a new unknown $v$ satisfying a cubic equation,
and reduce matters to bounding the $Z$-norm, see \eqref{nf29}, of the profile of $v$.
This task is further reduced to the more technical Lemma \ref{Zcontrol}, which is then proved in section \ref{TechProof}.

\section{Energy estimates, I: high Sobolev norms}\label{energy}

In this section we prove the following:

\begin{pro}\label{EnEst1}
If $u$ satisfies \eqref{pr5}--\eqref{pr7} then
\begin{equation}\label{BigOne}
 \sup_{t\in[0,T]}(1+t)^{-p_0}\|u(t)\|_{H^N}\lesssim\eps_0.
\end{equation}
\end{pro}

In other words, we improve the control of the high Sobolev energy norm in the bootstrap assumption \eqref{pr5}. The rest of
the section is concerned with the proof of Proposition \ref{EnEst1}.

Let $W:=D^N u$. Then
\begin{equation}\label{en1}
\begin{split}
&\partial_tW+i\Lambda W=\mathcal{N}_N,\qquad \mathcal{F}(\mathcal{N}_N)(\xi)=\frac{1}{2\pi}\int_{\mathbb{R}}q_N(\xi,\eta)\widehat{V}(\xi-\eta)\widehat{W}(\eta)\,d\eta,\\
&q_N(\xi,\eta)=\frac{|\xi|^N}{|\eta|^N}q_0(\xi,\eta)=i\frac{|\xi|^N}{|\eta|^N}\chi(\xi-\eta,\eta)|\xi-\eta|^{1/2}\eta.
\end{split}
\end{equation}

We define the quadratic energy
\begin{equation}\label{en2}
 E_N^{(2)}(t):=\frac{1}{2\pi}\int_{\mathbb{R}}|\widehat{W}(\xi,t)|^2\,d\xi=\frac{1}{2\pi}\int_{\mathbb{R}}\widehat{W}(\xi,t)\widehat{\overline{W}}(-\xi,t)\,d\xi.
\end{equation}
Using \eqref{en1} and recalling that $V=u+\overline{u}$ we calculate
\begin{equation}\label{en3}
\begin{split}
\frac{d}{dt}E_N^{(2)}&=\frac{1}{2\pi}\int_{\mathbb{R}}\widehat{W}(\xi)\widehat{\overline{\mathcal{N}_N}}(-\xi)+
\widehat{\mathcal{N}_N}(\xi)\widehat{\overline{W}}(-\xi)\,d\xi\\
&=\frac{1}{4\pi^2}\int_{\mathbb{R}\times\mathbb{R}}\widehat{W}(\xi)\widehat{\overline{W}}(-\eta)
\Big[\overline{q_N(\xi,\eta)}\overline{\widehat{V}(\xi-\eta)}+q_N(\eta,\xi)\widehat{V}(\eta-\xi)\Big]\,d\xi d\eta\\
&=\frac{1}{4\pi^2}\int_{\mathbb{R}\times\mathbb{R}}\widehat{W}(\xi)\widehat{\overline{W}}(-\eta)
\big[\overline{q_N(\xi,\eta)}+q_N(\eta,\xi)\big]\big[\widehat{u}(\eta-\xi)+\widehat{\overline{u}}(\eta-\xi)\big]\,d\xi d\eta.
\end{split}
\end{equation}

To eliminate cubic space-time integrals, we define the cubic energies
\begin{equation}\label{en4}
E_{N}^{(3)}(t):=\frac{1}{4\pi^2}\int_{\mathbb{R}\times\mathbb{R}}\widehat{W}(\xi,t)\widehat{\overline{W}(-\eta,t)}m_{N}(\xi,\eta)\widehat{u}(\eta-\xi,t)\,d\xi d\eta
\end{equation}
where
\begin{equation}\label{en5}
 m_{N}(\xi,\eta):=-i\frac{\overline{q_N(\xi,\eta)}+q_N(\eta,\xi)}{|\xi|^{3/2}-|\eta|^{3/2}+|\eta-\xi|^{3/2}}.
\end{equation}
Using \eqref{en1} and \eqref{pr3} we calculate
\begin{equation}\label{en6}
\frac{d}{dt}E_{N}^{(3)}(t)=A_{1}(t)+A_{2}(t)+A_{3}(t),
\end{equation}
where
\begin{equation}\label{en7}
A_{1}:=\frac{1}{4\pi^2}\int_{\mathbb{R}\times\mathbb{R}}\widehat{W}(\xi)\widehat{\overline{W}}(-\eta)
m_{N}(\xi,\eta)\widehat{u}(\eta-\xi)\big(-i|\xi|^{3/2}+i|\eta|^{3/2}-i|\eta-\xi|^{3/2}\big)\,d\xi d\eta,
\end{equation}
\begin{equation}\label{en8}
A_{2}:=\frac{1}{4\pi^2}\int_{\mathbb{R}\times\mathbb{R}}\big[\widehat{\mathcal{N}_N}(\xi)\widehat{\overline{W}}(-\eta)+\widehat{W}(\xi)\widehat{\overline{\mathcal{N}_N}}(-\eta)\big]m_{N}(\xi,\eta)\widehat{u}(\eta-\xi)\,d\xi d\eta,
\end{equation}
and
\begin{equation}\label{en10}
A_{3}:=\frac{1}{4\pi^2}\int_{\mathbb{R}\times\mathbb{R}}\widehat{W}(\xi)\widehat{\overline{W}}(-\eta)
m_{N}(\xi,\eta)\widehat{\mathcal{N}}(\eta-\xi)\,d\xi d\eta.
\end{equation}
Using the definition \eqref{en5}, notice that
\begin{equation*}
m_{N}(\xi,\eta)\big(-i|\xi|^{3/2}+i|\eta|^{3/2}-i|\xi-\eta|^{3/2}\big)=-\big[\overline{q_N(\xi,\eta)}+q_N(\eta,\xi)\big].
\end{equation*}
Therefore, using also \eqref{en3},
\begin{equation}\label{en11}
\frac{d}{dt}[E_N^{(2)}+E_{N}^{(3)}+\overline{E_{N}^{(3)}}](t)=2\Re(A_{2}(t))+2\Re(A_{3}(t)).
\end{equation}

The point of this identity is that the space-time integrals $A_{2},A_{3}$ are quartic expressions (in terms of the variable $u$).
We will estimate these expressions using the bootstrap assumptions \eqref{pr5} and \eqref{pr6}. We need first suitable
symbol-type estimates on the multiplier $q_N,m_N$.

Recall the definition, see \cite{IoPu2},
\begin{equation}\label{Al4}
S^\infty:=\{m:\mathbb{R}^2\to\mathbb{C}:\,m\text{ continuous and }\|m\|_{S^\infty}:=\|\mathcal{F}^{-1}(m)\|_{L^1}<\infty\}.
\end{equation}
Clearly, $S^\infty\hookrightarrow L^\infty(\mathbb{R}\times\mathbb{R})$. The following lemma summarizes some properties of $S^\infty$ symbols.

\begin{lem}\label{touse}
(i) If $m,m'\in S^\infty$ then $m\cdot m'\in S^\infty$ and
\begin{equation}\label{al8}
\|m\cdot m'\|_{S^\infty}\lesssim \|m\|_{S^\infty}\|m'\|_{S^\infty}.
\end{equation}

(ii) Assume $p,q,r\in[1,\infty]$ satisfy $1/p+1/q=1/r$, and $m\in S^\infty$. Then, for any $f,g\in L^2(\mathbb{R})$,
\begin{equation}\label{mk6}
\|M(f,g)\|_{L^r}\lesssim \|m\|_{S^\infty}\|f\|_{L^p}\|g\|_{L^q},
\end{equation}
where the bilinear operator $M$ is defined by
\begin{equation*}
\mathcal{F}\big[M(f,g)\big](\xi)=\frac{1}{2\pi}\int_{\mathbb{R}}m(\xi,\eta)\widehat{f}(\xi-\eta)\widehat{g}(\eta)\,d\eta.
\end{equation*}
In particular, if $1/p+1/q+1/r=1$,
\begin{equation}\label{mk6.01}
\Big|\int_{\mathbb{R}^2}m(\xi,\eta)\widehat{f}(\xi)\widehat{g}(\eta)\widehat{h}(-\xi-\eta)\,d\xi d\eta\Big|\lesssim \|m\|_{S^\infty}\|f\|_{L^p}\|g\|_{L^q}\|h\|_{L^r}.
\end{equation}

(iii) Moreover,  if $p_1,p_2,p_3,p_4\in[1,\infty]$ are exponents that satisfy
\begin{equation*}
\frac{1}{p_1}+\frac{1}{p_2}+\frac{1}{p_3}+\frac{1}{p_4}=1
\end{equation*}
then
\begin{equation}\label{mk6.5}
\begin{split}
\Big|\int_{\mathbb{R}^3}\widehat{f_1}(\xi)\widehat{f_2}(\eta)\widehat{f_3}(\rho-\xi)&\widehat{f_4}(-\rho-\eta)
m(\xi,\eta,\rho)\,d\xi d\rho d\eta\Big|
\\
&\lesssim \|f_1\|_{L^{p_1}}\|f_2\|_{L^{p_2}}\|f_3\|_{L^{p_3}}\|f_4\|_{L^{p_4}}
\|\mathcal{F}^{-1}m\|_{L^1}.
\end{split}
\end{equation}
\end{lem}

See \cite[Lemma 5.2]{IoPu2} for the proof.

Given any multiplier $m:\mathbb{R}^2\to\mathbb{C}$ and any $k,k_1,k_2\in\mathbb{Z}$ we define (recall \eqref{pr1})
\begin{equation}\label{al11}
m^{k,k_1,k_2}(\xi,\eta):=m(\xi,\eta)\cdot\varphi_k(\xi)\varphi_{k_1}(\xi-\eta)\varphi_{k_2}(\eta).
\end{equation}
Moreover, for $k\in\mathbb{Z}$ we denote by $P_k$, $P_{\leq k}$, and $P_{\geq k}$ the operators defined by the Fourier multipliers $\varphi_k$,
$\varphi_{\leq k}$, and $\varphi_{\geq k}$ respectively.
Moreover, we let
\begin{align}
 \label{P'_k}
P'_k := P_{k-1} + P_k + P_{k+1} \qquad \mbox{and} \qquad \varphi'_k := \varphi_{k-1} + \varphi_k + \varphi_{k+1}.
\end{align}

Let
\begin{equation}\label{al11.1}
 \mathcal{X}:=\{(k,k_1,k_2)\in\mathbb{Z}^3:\max(k,k_1,k_2)-\mathrm{med}(k,k_1,k_2)\leq 6\},
\end{equation}
and notice that $m^{k,k_1,k_2}\equiv 0$ unless $(k,k_1,k_2)\in\mathcal{X}$.

\begin{lem}\label{SymbBound}
With $q_0$, $q_N$, and $m_N$ defined as before, for any $k,k_1,k_2\in\mathbb{Z}$ we have
\begin{equation}\label{en21.2}
\|q_0^{k,k_1,k_2}\|_{S^\infty}+\|q_N^{k,k_1,k_2}\|_{S^\infty}\lesssim 2^{k_1/2}2^{k_2}\one_{15}(k,k_1,k_2),
\end{equation}
and
\begin{equation}\label{en21}
\|m_{N}^{k,k_1,k_2}\|_{S^\infty}\lesssim 2^{k_1/2}2^{-k_2/2}\one_{15}(k,k_1,k_2),
\end{equation}
where, for any $d\in\mathbb{Z}_+$,
\begin{equation}\label{en21.21}
\one_d(k,k_1,k_2):=\mathbf{1}_{[-d,\infty)}(k_2)\mathbf{1}_{[-d,\infty)}(k_2-k_1)\mathbf{1}_{[-d,d]}(k_2-k).
\end{equation}
\end{lem}

\begin{proof}[Proof of Lemma \ref{SymbBound}]
The bounds \eqref{en21.2} follow from the definitions and simple integration by parts arguments.
To prove \eqref{en21} we notice first that
\begin{equation}\label{en21.4}
(a+b)^{3/2}-b^{3/2}-a^{3/2}\in[ab^{1/2}/4,4ab^{1/2}]\qquad\text{ if }0\leq a\leq b.
\end{equation}
Therefore, using standard integration by parts,
\begin{equation}\label{en25}
\Big\|\frac{1}{|\xi|^{3/2}-|\eta|^{3/2}\pm|\xi-\eta|^{3/2}}\cdot\varphi_k(\xi)\varphi_{k_1}(\xi-\eta)\varphi_{k_2}(\eta)\Big\|_{S^\infty}
\lesssim \frac{1}{2^{\min(k_1,k_2)}2^{\max(k,k_2)/2}},
\end{equation}
for any $k,k_1,k_2\in\mathbb{Z}$. The bound \eqref{en21} follows from \eqref{en21.2} and \eqref{en25} if $k_1\geq \max(k,k_2)-40$.

On the other hand, if
\begin{equation*}
 k_1\leq \max(k,k_2)-40,
\end{equation*}
then we may assume that $|k-k_2|\leq 4$ and write
\begin{equation*}
\begin{split}
&[q_N(\eta,\xi)+\overline{q_N}(\xi,\eta)]\cdot\varphi_k(\xi)\varphi_{k_1}(\xi-\eta)\varphi_{k_2}(\eta)\\
&=i|\xi-\eta|^{1/2}\varphi_k(\xi)\varphi_{k_1}(\xi-\eta)\varphi_{k_2}(\eta)
\Big[\frac{|\eta|^N\xi}{|\xi|^N}\chi(\eta-\xi,\xi)-\frac{|\xi|^N\eta}{|\eta|^N}\chi(\xi-\eta,\eta)\Big]\\
&=i|\xi-\eta|^{1/2}\frac{\varphi_k(\xi)\varphi_{k_1}(\xi-\eta)\varphi_{k_2}(\eta)\xi\eta}{|\xi|^N|\eta|^N}
\big[\eta^{2N-1}\chi(\eta-\xi,\xi)-\xi^{2N-1}\chi(\xi-\eta,\eta)\big].
\end{split}
\end{equation*}
This shows easily that
\begin{equation*}
\big\|[q_N(\eta,\xi)+\overline{q_N}(\xi,\eta)]\cdot\varphi_k(\xi)\varphi_{k_1}(\xi-\eta)\varphi_{k_2}(\eta)\big\|_{S^\infty}
\lesssim 2^{3k_1/2},
\end{equation*}
and the desired bound \eqref{en21} follows using also \eqref{en25}.
\end{proof}

\begin{rem}\label{asymp}
 In certain estimates we need more precise asymptotics for the symbols $q_N$ and $m_N$. More precisely, if
\begin{equation}\label{en25.1}
k_1+40\leq \max(k,k_2)\qquad\text{ and }\qquad |k-k_2|\leq 4,
\end{equation}
then
\begin{equation}\label{en25.2}
 \Big\|\mathcal{F}^{-1}\big[\big(q_N(\xi,\eta)-i|\xi-\eta|^{1/2}\xi\big)\cdot \varphi_{k_1}(\xi-\eta)
\varphi_{k}(\xi)\varphi_{k_2}(\eta)\big]\Big\|_{L^1}\lesssim 2^{3k_1/2},
\end{equation}
and
\begin{equation}\label{en25.3}
 \Big\|\mathcal{F}^{-1}\Big[\Big(m_N(\xi,\eta)+\frac{4N-2}{3}\frac{|\xi-\eta|^{1/2}\xi}{|\xi|^{3/2}}\Big)\cdot \varphi_{k_1}(\xi-\eta)
\varphi_{k}(\xi)\varphi_{k_2}(\eta)\Big]\Big\|_{L^1}\lesssim 2^{k_1-k_2}.
\end{equation}
The proofs of these estimates are similar to the proofs in Lemma \ref{SymbBound}.
\end{rem}

We control now the cubic energy $E_N^{(3)}$.

\begin{lem}\label{EnergyComp}
Assuming the bounds \eqref{pr5} and \eqref{pr6}, for any $t\in[0,T]$ we have
\begin{equation}\label{en30}
E_{N}^{(3)}(t)\lesssim \eps_1^3(1+t)^{2p_0}.
\end{equation}
\end{lem}

\begin{proof}[Proof of Lemma \ref{EnergyComp}] We use Lemma \ref{touse} (ii) and Lemma \ref{SymbBound}.
More precisely, using also the definition \eqref{en4},
\begin{equation*}
\begin{split}
E_{N,\iota}^{(3)}(t)&\lesssim\sum_{(k,k_1,k_2)\in\mathcal{X}}\|P'_kW(t)\|_{L^2}\|P'_{k_2}W(t)\|_{L^2}\|P'_{k_1}u(t)\|_{L^\infty}
\|m_N^{k,k_1,k_2}\|_{S^\infty}\lesssim\|W(t)\|_{L^2}^2\eps_1,
\end{split}
\end{equation*}
and the desired conclusion follows from \eqref{pr5}.
\end{proof}

We estimate now the functions $A_{2}(t)$ and $A_{3}(t)$ defined in \eqref{en8}--\eqref{en10}.
\begin{lem}\label{EnergyInc}
Assuming the bounds \eqref{pr5} and \eqref{pr6}, for any $t\in[0,T]$ we have
\begin{equation}\label{en31}
|A_{2}(t)|+|A_{3}(t)|\lesssim \eps_1^4(1+t)^{-1+2p_0}.
\end{equation}
\end{lem}

\begin{proof}[Proof of Lemma \ref{EnergyInc}] Using the definitions and simple changes of variables, we write
\begin{equation*}
\begin{split}
A_{2}=\frac{1}{8\pi^3}\int_{\mathbb{R}^3}&q_N(\xi,\rho)\widehat{V}(\xi-\rho)\widehat{W}(\rho)\widehat{\overline{W}}(-\eta)m_{N}(\xi,\eta)\widehat{u}(\eta-\xi)\,d\xi d\eta d\rho\\
+\frac{1}{8\pi^3}\int_{\mathbb{R}^3}&\widehat{W}(\xi)\overline{q_N(\eta,\rho)}\overline{\widehat{V}(\eta-\rho)}\,\overline{\widehat{W}(\rho)}m_{N}(\xi,\eta)\widehat{u}(\eta-\xi)\,d\xi d\eta d\rho\\
=\frac{1}{8\pi^3}\int_{\mathbb{R}^3}&\widehat{W}(\xi)\widehat{\overline{W}}(\eta)\widehat{V}(\rho-\xi)\widehat{u}(-\eta-\rho)\\
&\times\big[q_N(\rho,\xi)m_{N}(\rho,-\eta)-q_N(\xi-\eta-\rho,-\eta)m_{N}(\xi,\xi-\eta-\rho)\big]\,d\xi d\eta d\rho.
\end{split}
\end{equation*}

We can also rewrite
\begin{equation*}
A_{3}=\frac{1}{8\pi^3}\int_{\mathbb{R}^3}\widehat{W}(\xi)\widehat{\overline{W}}(\eta)\widehat{V}(\rho-\xi)\widehat{u}(-\eta-\rho)
m_N(\xi,-\eta)q_0(-\xi-\eta,-\rho-\eta)\,d\xi d\rho d\eta.
\end{equation*}
We use Lemma \ref{touse} (iii). It follows that
\begin{equation}\label{en35}
 |A_l|\lesssim \sum_{k_1,k_2,k_3,k_4\in\mathbb{Z}}\|P'_{k_1}W\|_{L^2}\|P'_{k_2}W\|_{L^2}\|P'_{k_3}u\|_{L^\infty}
\|P'_{k_4}u\|_{L^\infty}\|\mathcal{F}^{-1}(a_l^{k_1,k_2,k_3,k_4})\|_{L^1}
\end{equation}
for $l\in\{2,3\}$, where, with $\widetilde{\varphi}_{k_1,k_2,k_3,k_4}(\xi,\eta,\rho):=\varphi_{k_1}(\xi)\varphi_{k_2}(\eta)\varphi_{k_3}(\rho-\xi)\varphi_{k_4}(-\rho-\eta)$,
\begin{equation}\label{en36}
\begin{split}
a_2^{k_1,k_2,k_3,k_4}(\xi,\eta,\rho):=&\big[q_N(\rho,\xi)m_{N}(\rho,-\eta)-q_N(\xi-\eta-\rho,-\eta)m_{N}(\xi,\xi-\eta-\rho)\big]\\
&\times\widetilde{\varphi}_{k_1,k_2,k_3,k_4}(\xi,\eta,\rho),
 \end{split}
\end{equation}
and
\begin{equation}\label{en36.1}
\begin{split}
a_3^{k_1,k_2,k_3,k_4}(\xi,\eta,\rho):=m_N(\xi,-\eta)q_0(-\xi-\eta,-\rho-\eta)\cdot
\widetilde{\varphi}_{k_1,k_2,k_3,k_4}(\xi,\eta,\rho).
 \end{split}
\end{equation}

We will show that
\begin{equation}\label{en37}
\|\mathcal{F}^{-1}(a_l^{k_1,k_2,k_3,k_4})\|_{L^1}\lesssim 2^{-|k_1-k_2|/4}(2^{k_3/4}+2^{3.9k_3})(2^{k_4/4}+2^{3.9k_4}),
\end{equation}
for any $l\in\{2,3\}$ and for any $k_1,k_2,k_3,k_4\in\mathbb{Z}$. Assuming these bounds, the desired estimates \eqref{en31}
follow from \eqref{en35} and the bootstrap assumptions \eqref{pr6} and \eqref{pr7}.

To prove \eqref{en37} we notice first that if
\begin{equation}\label{en38}
 \text{ if }\quad f(x,y,z)=f_1(x,z)f_2(y,z)\quad\text{ then }\quad \|\mathcal{F}^{-1}f\|_{L^1(\mathbb{R}^3)}\lesssim \|\mathcal{F}^{-1}f_1\|_{L^1(\mathbb{R}^2)}
\|\mathcal{F}^{-1}f_2\|_{L^1(\mathbb{R}^2)}.
\end{equation}
Therefore, using \eqref{en21} and \eqref{en21.2},
\begin{equation*}
 \begin{split}
 \Big\|\mathcal{F}^{-1}\big[&q_N(\rho,\xi)m_N(\rho,-\eta)\cdot
\widetilde{\varphi}_{k_1,k_2,k_3,k_4}(\xi,\eta,\rho)\big]\Big\|_{L^1}\lesssim \sum_{k',k''\in\mathbb{Z},\,|k'-k''|\leq 4}
\|q_N^{k',k_3,k_1}\|_{S^\infty}\|m_N^{k'',k_4,k_2}\|_{S^\infty}\\
&\lesssim \sum_{k',k''\in\mathbb{Z},\,|k'-k''|\leq 4} 2^{k_3/2}2^{\min(k_1,k')}2^{-2|k_1-k'|}\cdot
2^{k_4/2}2^{-\max(k_2,k'')/2}2^{-2|k_2-k''|}\\
&\lesssim 2^{k_3/2}2^{k_4/2}2^{\min(k_1,k_2)/2}2^{-|k_1-k_2|}.
 \end{split}
\end{equation*}
Similarly
\begin{equation*}
\begin{split}
\Big\|\mathcal{F}^{-1}&\big[q_N(\xi-\eta-\rho,-\eta)m_{N}(\xi,\xi-\eta-\rho)\cdot
\widetilde{\varphi}_{k_1,k_2,k_3,k_4}(\xi,\eta,\rho)\big]\Big\|_{L^1}\\
&+\Big\|\mathcal{F}^{-1}\big[m_N(\xi,-\eta)q_0(-\xi-\eta,-\rho-\eta)\cdot
\widetilde{\varphi}_{k_1,k_2,k_3,k_4}(\xi,\eta,\rho)\big]\Big\|_{L^1}\\
&\lesssim 2^{k_3/2}2^{k_4/2}2^{\min(k_1,k_2)/2}2^{-|k_1-k_2|}.
 \end{split}
\end{equation*}
The desired bounds \eqref{en37} follow if $l=3$ or if $l=2$ and $2^{\min(k_1,k_2)}\lesssim 2^{\max(k_3,k_4,0)}$.

It remains to prove the bounds \eqref{en37} in the case $l=2$ and
\begin{equation}\label{en50}
 \max(k_3,k_4,0)+100\leq\min(k_1,k_2),\qquad |k_1-k_2|\leq 4.
\end{equation}
In this case, we use Remark \ref{asymp}, to avoid the loss of $1/2$
derivative. The same argument as before shows that
\begin{equation*}
\begin{split}
\Big\|\mathcal{F}^{-1}\Big[\Big(q_N(\rho,\xi)m_N(\rho,-\eta)+i|\rho-\xi|^{1/2}\rho\frac{4N-2}{3}\frac{|\rho+\eta|^{1/2}\rho}{|\rho|^{3/2}}\Big)\cdot
\widetilde{\varphi}_{k_1,k_2,k_3,k_4}(\xi,\eta,\rho)\Big]\Big\|_{L^1}\\
\lesssim 2^{k_3+k_4/2}+2^{k_3/2+k_4}
\end{split}
\end{equation*}
and
\begin{equation*}
\begin{split}
\Big\|&\mathcal{F}^{-1}\Big[\Big(q_N(\xi-\eta-\rho,-\eta)m_{N}(\xi,\xi-\eta-\rho)\\
&+i|\rho-\xi|^{1/2}(\xi-\eta-\rho)\frac{4N-2}{3}\frac{|\eta+\rho|^{1/2}\xi}{|\xi|^{3/2}}\Big)\cdot
\widetilde{\varphi}_{k_1,k_2,k_3,k_4}(\xi,\eta,\rho)\Big]\Big\|_{L^1}\lesssim 2^{k_3+k_4/2}+2^{k_3/2+k_4}.
 \end{split}
\end{equation*}
provided that \eqref{en50} holds. Moreover
\begin{equation*}
\big\|\mathcal{F}^{-1}\big[\big(|\rho|^{1/2}-(\xi-\eta-\rho)\frac{\xi}{|\xi|^{3/2}}\big)\cdot
\widetilde{\varphi}_{k_1,k_2,k_3,k_4}(\xi,\eta,\rho)\big]\big\|_{L^1}\lesssim 2^{k_4/2}+2^{k_3/2},
\end{equation*}
and the desired conclusion \eqref{en37} follows.
\end{proof}

We can now complete the proof of Proposition \ref{EnEst1}.

\begin{proof}[Proof of Proposition \ref{EnEst1}] We first use the formula \eqref{en11}. It follows that
\begin{equation*}
 \big|E^{(2)}_N(t)\big|\lesssim |E^{(3)}_N(t)\big|+\big|E^{(2)}_N(0)\big|+\big |E^{(3)}_N(0)\big|+\int_0^t |A_2(s)|+|A_3(s)|\,ds
\end{equation*}
for any $t\in[0,T]$. Using now Lemma \ref{EnergyComp} and Lemma \ref{EnergyInc}, we have
\begin{equation*}
 E^{(2)}_N(t)\lesssim E^{(2)}_N(0)+\eps_1^4(1+t)^{2p_0},
\end{equation*}
for any $t\in[0,T]$. The assumption \eqref{pr7} shows that $E^{(2)}_N(0)\lesssim \eps_0^2$. Moreover,
in view of the definition \eqref{pr2}, the low frequencies of solutions are not affected by the nonlinear flow, i.e.
\begin{equation*}
 \|P_{\leq -30}u(t)\|_{L^2}=\|P_{\leq -30}u(0)\|_{L^2}\lesssim\eps_0.
\end{equation*}
Therefore, for any $t\in[0,T]$,
\begin{equation*}
 \|u(t)\|^2_{H^N}\lesssim \|P_{\leq -5}u(t)\|_{L^2}^2+\|W\|_{L^2}^2\lesssim \eps_0^2+\eps_1^4(1+t)^{2p_0}\lesssim \eps_0^2(1+t)^{2p_0},
\end{equation*}
as desired.
\end{proof}

\section{Energy estimates, II: weighted norms}\label{energy2}

In this section we prove the following:

\begin{pro}\label{EnEst2}
If $u$ satisfies \eqref{pr5}--\eqref{pr7} then
\begin{equation}\label{BigTwo}
 \sup_{t\in[0,T]}(1+t)^{-4p_0}\|Su(t)\|_{H^2}\lesssim\eps_0.
\end{equation}
\end{pro}

In other words, we improve the control of the weighted energy norm in the bootstrap assumption \eqref{pr5}. The rest of
the section is concerned with the proof of Proposition \ref{EnEst2}.

It follows from \eqref{pr3} that
\begin{equation*}
(\partial_t+i\Lambda)Su=S\mathcal{N}+\frac{3}{2}(\partial_t+i\Lambda)u=S\mathcal{N}+\frac{3}{2}\mathcal{N}.
\end{equation*}
Let $Z:=m(D)Su$, where $m(r)=1+r^2$. The function $Z$ satisfies the equation
\begin{equation}\label{we2}
(\partial_t+i\Lambda)Z=\mathcal{N}_Z,
\end{equation}
where
\begin{equation}\label{we3}
\begin{split}
\mathcal{F}(\mathcal{N}_Z)(\xi)&=\frac{1}{2\pi}\int_{\mathbb{R}}q(\xi,\eta)\widehat{Z}(\eta)\widehat{V}(\xi-\eta)\,d\eta\\
&+\frac{1}{2\pi}\int_{\mathbb{R}}q_1(\xi,\eta)\widehat{u}(\eta)\widehat{(Z+\overline{Z})}(\xi-\eta)\,d\eta
+\frac{1}{2\pi}\int_{\mathbb{R}}q_2(\xi,\eta)\widehat{u}(\eta)\widehat{(u+\overline{u})}(\xi-\eta)\,d\eta,
\end{split}
\end{equation}
and the symbols $q_1,q_2,q_3$ are given by
\begin{equation}\label{we4}
\begin{split}
&q(\xi,\eta):=q_0(\xi,\eta)\frac{m(\xi)}{m(\eta)},\\
&q_1(\xi,\eta):=q_0(\xi,\eta)\frac{m(\xi)}{m(\xi-\eta)},\\
&q_2(\xi,\eta):=-m(\xi)[\xi\partial_\xi+\eta\partial_\eta-(3/2)]q_0(\xi,\eta).
\end{split}
\end{equation}

As in the previous section, we start with the quadratic energy
\begin{equation}\label{we5}
 E_w^{(2)}(t):=\frac{1}{2\pi}\int_{\mathbb{R}}|\widehat{Z}(\xi,t)|^2\,d\xi=\frac{1}{2\pi}\int_{\mathbb{R}}\widehat{Z}(\xi,t)\widehat{\overline{Z}}(-\xi,t)\,d\xi.
\end{equation}
Using \eqref{we2} and \eqref{we3} we calculate
\begin{equation*}
\frac{d}{dt}E_w^{(2)}(t)=2\Re\frac{1}{2\pi}\int_{\mathbb{R}}\widehat{\mathcal{N}_Z}(\xi,t)\widehat{\overline{Z}}(-\xi,t)\,d\xi=
I_1(t)+I_{2,-}(t)+I_{2,-}(t)+I_{3,+}(t)+I_{3,-}(t),
\end{equation*}
where, with $Z_+:=Z$, $Z_-:=\overline{Z}$, $u_+:=u$, $u_-:=\overline{u}$,
\begin{equation*}
\begin{split}
&I_1:=2\Re\frac{1}{4\pi^2}\int_{\mathbb{R}\times\mathbb{R}}q(\xi,\eta)\widehat{Z}(\eta)\widehat{V}(\xi-\eta)\widehat{\overline{Z}}(-\xi)\,d\xi d\eta,\\
&I_{2,\iota}:=2\Re\frac{1}{4\pi^2}\int_{\mathbb{R}\times\mathbb{R}}q_1(\xi,\eta)\widehat{u}(\eta)\widehat{Z_\iota}(\xi-\eta)\widehat{\overline{Z}}(-\xi)\,d\xi d\eta,\\
&I_{3,\iota}:=2\Re\frac{1}{4\pi^2}\int_{\mathbb{R}\times\mathbb{R}}q_2(\xi,\eta)\widehat{u}(\eta)\widehat{u_\iota}(\xi-\eta)\widehat{\overline{Z}}(-\xi)\,d\xi d\eta.
\end{split}
\end{equation*}

We define now cubic energy corrections. Let
\begin{equation}\label{we6}
\begin{split}
&E_{w,1}^{(3)}(t):=2\Re\frac{1}{4\pi^2}\int_{\mathbb{R}\times\mathbb{R}}m_1(\xi,\eta)\widehat{Z}(\eta,t)\widehat{u}(\xi-\eta,t)\widehat{\overline{Z}}(-\xi,t)\,d\xi d\eta,\\
&E_{w,2,\iota}^{(3)}(t):=2\Re\frac{1}{4\pi^2}\int_{\mathbb{R}\times\mathbb{R}}m_{2,\iota}(\xi,\eta)\widehat{u}(\eta,t)\widehat{Z_\iota}(\xi-\eta,t)\widehat{\overline{Z}}(-\xi,t)\,d\xi d\eta,\\
&E_{w,3,\iota}^{(3)}(t):=2\Re\frac{1}{4\pi^2}\int_{\mathbb{R}\times\mathbb{R}}m_{3,\iota}(\xi,\eta)\widehat{u}(\eta,t)\widehat{u_\iota}(\xi-\eta,t)
\widehat{\overline{Z}}(-\xi,t)\,d\xi d\eta,
\end{split}
\end{equation}
where\footnote{It is important to symmetrize suitably the symbol $m_1$, as in the previous section, to avoid
derivative losses. One can also symmetrize some of the other symbols, but these symmetrizations are meaningless in our situation.}
\begin{equation}\label{we7}
\begin{split}
&m_1(\xi,\eta):=i\frac{q(\xi,\eta)+\overline{q}(\eta,\xi)}{-|\eta|^{3/2}-|\xi-\eta|^{3/2}+|\xi|^{3/2}},\\
&m_{2,\iota}(\xi,\eta):=i\frac{q_1(\xi,\eta)}{-|\eta|^{3/2}-\iota|\xi-\eta|^{3/2}+|\xi|^{3/2}},\\
&m_{3,\iota}(\xi,\eta):=i\frac{q_2(\xi,\eta)}{-|\eta|^{3/2}-\iota|\xi-\eta|^{3/2}+|\xi|^{3/2}}.
\end{split}
\end{equation}
Simple calculations, using also the identities \eqref{pr3} and \eqref{we2} show that
\begin{equation}\label{we8}
 \frac{d}{dt}E_{w,l}^{(3)}(t)=-I_l(t)+J_l(t),\qquad l\in\{1,(2,+),(2,-),(3,+), (3,-)\},
\end{equation}
where, with $\mathcal{N}_{Z,+}:=\mathcal{N}_{Z}$, $\mathcal{N}_{Z,-}:=\overline{\mathcal{N}_{Z}}$, $\mathcal{N}_{+}:=\mathcal{N}$, $\mathcal{N}_{-}:=\overline{\mathcal{N}}$,
\begin{equation}\label{we9}
\begin{split}
4\pi^2J_1=&2\Re\int_{\mathbb{R}\times\mathbb{R}}m_1(\xi,\eta)\widehat{\mathcal{N}_Z}(\eta)\widehat{u}(\xi-\eta)\widehat{\overline{Z}}(-\xi)\,d\xi d\eta\\
+&2\Re\int_{\mathbb{R}\times\mathbb{R}}m_1(\xi,\eta)\widehat{Z}(\eta)\widehat{\mathcal{N}}(\xi-\eta)\widehat{\overline{Z}}(-\xi)\,d\xi d\eta\\
+&2\Re\int_{\mathbb{R}\times\mathbb{R}}m_1(\xi,\eta)\widehat{Z}(\eta)\widehat{u}(\xi-\eta)\widehat{\overline{\mathcal{N}_Z}}(-\xi)\,d\xi d\eta,
\end{split}
\end{equation}
\begin{equation}\label{we10}
\begin{split}
4\pi^2J_{2,\iota}=&2\Re\int_{\mathbb{R}\times\mathbb{R}}m_{2,\iota}(\xi,\eta)\widehat{\mathcal{N}}(\eta)\widehat{Z_\iota}(\xi-\eta)\widehat{\overline{Z}}(-\xi)\,d\xi d\eta\\
+&2\Re\int_{\mathbb{R}\times\mathbb{R}}m_{2,\iota}(\xi,\eta)\widehat{u}(\eta)\widehat{\mathcal{N}_{Z,\iota}}(\xi-\eta)\widehat{\overline{Z}}(-\xi)\,d\xi d\eta\\
+&2\Re\int_{\mathbb{R}\times\mathbb{R}}m_{2,\iota}(\xi,\eta)\widehat{u}(\eta)\widehat{Z_\iota}(\xi-\eta)\widehat{\overline{\mathcal{N}_Z}}(-\xi)\,d\xi d\eta,\\
\end{split}
\end{equation}
\begin{equation}\label{we12}
\begin{split}
4\pi^2J_{3,\iota}=&2\Re\int_{\mathbb{R}\times\mathbb{R}}m_{3,\iota}(\xi,\eta)\widehat{\mathcal{N}}(\eta)\widehat{u_\iota}(\xi-\eta)\widehat{\overline{Z}}(-\xi)\,d\xi d\eta\\
+&2\Re\int_{\mathbb{R}\times\mathbb{R}}m_{3,\iota}(\xi,\eta)\widehat{u}(\eta)\widehat{\mathcal{N}_\iota}(\xi-\eta)\widehat{\overline{Z}}(-\xi)\,d\xi d\eta\\
+&2\Re\int_{\mathbb{R}\times\mathbb{R}}m_{3,\iota}(\xi,\eta)\widehat{u}(\eta)\widehat{u_\iota}(\xi-\eta)\widehat{\overline{\mathcal{N}_Z}}(-\xi)\,d\xi d\eta.
\end{split}
\end{equation}
Therefore
\begin{equation}\label{we14}
\begin{split}
 \frac{d}{dt}\big[E_{w}^{(2)}+&E_{w,1}^{(3)}+E_{w,2,+}^{(3)}+E_{w,2,-}^{(3)}+E_{w,3,+}^{(3)}+E_{w,3,-}^{(3)}\big](t)\\
&=J_1(t)+J_{2,+}(t)+J_{2,-}(t)+J_{3,+}(t)+J_{3,-}(t).
\end{split}
\end{equation}

As in the previous section, we would like to estimate the cubic corrections $E_{w,l}^{(3)}$ and the space-time contributions $J_l$. Recall the definitions \eqref{Al4} and \eqref{al11}, and Lemma \ref{touse}. We record first several symbol-type bounds.

\begin{lem}\label{SymbBound2}
For any $k,k_1,k_2\in\mathbb{Z}$ we have
\begin{equation}\label{we15.1}
\|q^{k,k_1,k_2}\|_{S^\infty}\lesssim 2^{k_1/2}2^{k_2}\one_{15}(k,k_1,k_2),
\end{equation}
\begin{equation}\label{we15.2}
\|q_1^{k,k_1,k_2}\|_{S^\infty}\lesssim 2^{2k}2^{-2\max(k_1,0)}2^{k_1/2}2^{k_2}\one_{15}(k,k_1,k_2),
\end{equation}
and
\begin{equation}\label{we15.3}
\|q_2^{k,k_1,k_2}\|_{S^\infty}\lesssim 2^{2k}2^{k_1/2}2^{k_2}\one_{15}(k,k_1,k_2).
\end{equation}
Moreover
\begin{equation}\label{we15.4}
\|m_1^{k,k_1,k_2}\|_{S^\infty}\lesssim 2^{k_1/2}2^{-k_2/2}\one_{15}(k,k_1,k_2),
\end{equation}
\begin{equation}\label{we15.5}
\|m_{2,\iota}^{k,k_1,k_2}\|_{S^\infty}\lesssim 2^{(k_2-k_1)/2}2^{2k}2^{-2\max(k_1,0)}\one_{15}(k,k_1,k_2),
\end{equation}
and
\begin{equation}\label{we15.7}
\|m_{3,\iota}^{k,k_1,k_2}\|_{S^\infty}\lesssim 2^{(k_2-k_1)/2}2^{2k}\one_{15}(k,k_1,k_2).
\end{equation}
\end{lem}

\begin{proof}[Proof of Lemma \ref{SymbBound2}] The bounds follow from the explicit formulas \eqref{pr2} and \eqref{we4},
and the bound \eqref{en25}, as  in the proof of Lemma \ref{SymbBound}.
\end{proof}

We estimate now the nonlinearities $\mathcal{N}$ and $\mathcal{N}_Z$.

\begin{lem}\label{NonlinEst}
 For any $t\in[0,T]$ and $l\in\mathbb{Z}$, we have
\begin{equation}\label{we16}
 \|P_l\mathcal{N}(t)\|_{L^2}\lesssim\eps_1^2(1+t)^{p_0-1/2}2^{-(N-1)l},
\end{equation}
\begin{equation}\label{we17}
 \|P_l\mathcal{N}(t)\|_{L^\infty}\lesssim\eps_1^2(1+t)^{-1}2^{-3l},
\end{equation}
and
\begin{equation}\label{we18}
 \|P_l\mathcal{N}_Z(t)\|_{L^2}\lesssim\eps_1^2(1+t)^{4p_0-1/2}2^{l}.
\end{equation}
Moreover,
\begin{equation}\label{we19}
 P_l\mathcal{N}=0\quad\text{ and }\quad P_l\mathcal{N}_Z=0\quad\text{ if }\quad l\leq -30.
\end{equation}
\end{lem}

\begin{proof}[Proof of Lemma \ref{NonlinEst}] The identities in \eqref{we19} follow easily from the definitions. To prove the other
estimates we use the bootstrap assumptions \eqref{pr5} and \eqref{pr6},
together with the symbol estimates \eqref{en21.2} and \eqref{we15.1}--\eqref{we15.3}.

To prove \eqref{we16} and \eqref{we17} we recall
the formula
\begin{equation*}
 \mathcal{F}(\mathcal{N})(\xi)=\frac{1}{2\pi}\int_{\mathbb{R}}q_0(\xi,\eta)\widehat{V}(\xi-\eta)\widehat{u}(\eta)\,d\eta.
\end{equation*}
Using Lemma \ref{touse} (ii) and \eqref{en21.2}, for $l\geq -30$,
\begin{equation*}
\begin{split}
\|P_l\mathcal{N}(t)\|_{L^2}&\lesssim \sum_{k_1,k_2\in\mathbb{Z}}\|q_0^{l,k_1,k_2}\|_{S^\infty}
\|P'_{k_1}u(t)\|_{L^\infty}\|P'_{k_2}u(t)\|_{L^2}\\
&\lesssim \eps_1^2\sum_{k_1,k_2\in\mathbb{Z}}\one_{15}(l,k_1,k_2)2^{k_1/2}2^{k_2}\cdot(1+t)^{-1/2}2^{-4\max(k_1,0)}\cdot
(1+t)^{p_0}2^{-Nk_2}\\
&\lesssim \eps_1^2(1+t)^{p_0-1/2}2^{-(N-1)l}
\end{split}
\end{equation*}
and similarly
\begin{equation*}
\begin{split}
\|P_l\mathcal{N}(t)\|_{L^\infty}&\lesssim \sum_{k_1,k_2\in\mathbb{Z}}\|q_0^{l,k_1,k_2}\|_{S^\infty}
\|P'_{k_1}u(t)\|_{L^\infty}\|P'_{k_2}u(t)\|_{L^\infty}\\
&\lesssim \eps_1^2\sum_{k_1,k_2\in\mathbb{Z}}\one_{15}(l,k_1,k_2)2^{k_1/2}2^{k_2}\cdot(1+t)^{-1/2}2^{-4\max(k_1,0)}\cdot
(1+t)^{-1/2}2^{-4k_2}\\
&\lesssim \eps_1^2(1+t)^{-1}2^{-3l}.
\end{split}
\end{equation*}

Also, using \eqref{we3}, Lemma \ref{touse}, and the bounds \eqref{we15.1}--\eqref{we15.3}, for $l\geq -30$,
\begin{equation*}
  \|P_l\mathcal{N}_Z(t)\|_{L^2}\lesssim I+II+III,
\end{equation*}
where
\begin{equation*}
\begin{split}
I:&=\sum_{k_1,k_2\in\mathbb{Z}}\|q^{l,k_1,k_2}\|_{S^\infty}\|P'_{k_1}u(t)\|_{L^\infty}
\|P'_{k_2}Z(t)\|_{L^2}\\
&\lesssim \eps_1^2\sum_{k_1,k_2\in\mathbb{Z}}\one_{15}(l,k_1,k_2)2^{k_1/2}2^{k_2}\cdot (1+t)^{-1/2}2^{-4\max(k_1,0)}
\cdot (1+t)^{4p_0}\\
&\lesssim \eps_1^2(1+t)^{4p_0-1/2}2^{l},
\end{split}
\end{equation*}
\begin{equation*}
\begin{split}
II:&=\sum_{k_1,k_2\in\mathbb{Z}}\|q_1^{l,k_1,k_2}\|_{S^\infty}\|P'_{k_1}Z(t)\|_{L^2}
\|P'_{k_2}u(t)\|_{L^\infty}\\
&\lesssim \eps_1^2\sum_{k_1,k_2\in\mathbb{Z}}\one_{15}(l,k_1,k_2)2^{2l}2^{-2\max(k_1,0)}2^{k_1/2}2^{k_2}
\cdot (1+t)^{4p_0}\cdot(1+t)^{-1/2}2^{-4k_2}\\
&\lesssim \eps_1^2(1+t)^{4p_0-1/2}2^{-l},
\end{split}
\end{equation*}
and
\begin{equation*}
\begin{split}
III:&=\sum_{k_1,k_2\in\mathbb{Z}}\|q_2^{l,k_1,k_2}\|_{S^\infty}\|P'_{k_1}u(t)\|_{L^\infty}
\|P'_{k_2}u(t)\|_{L^2}\\
&\lesssim \eps_1^2\sum_{k_1,k_2\in\mathbb{Z}}\one_{15}(l,k_1,k_2)2^{2l}2^{k_1/2}2^{k_2}
\cdot (1+t)^{-1/2}2^{-4\max(k_1,0)}\cdot(1+t)^{p_0}2^{-Nk_2}\\
&\lesssim \eps_1^2(1+t)^{p_0-1/2}2^{-(N-3)l}.
\end{split}
\end{equation*}
The desired bound \eqref{we18} follows.
\end{proof}

We can estimate now the expressions $J_{2,\pm},J_{3,\pm}$ defined in \eqref{we10}--\eqref{we12}.

\begin{lem}\label{Jest}
For any $t\in[0,T]$ and $\iota\in\{+,-\}$ we have
\begin{equation}\label{we20}
 |J_{2,\iota}(t)|+|J_{3,\iota}(t)|\lesssim \eps_1^4(1+t)^{-1+8p_0}.
\end{equation}
\end{lem}

\begin{proof}[Proof of Lemma \ref{Jest}] In this lemma we use the stronger low-frequency estimates in \eqref{pr7} and the fact
that the nonlinearity $\mathcal{N}$ does not have any low frequencies. Therefore, for any $t\in[0,T]$ and $l\leq -30$
\begin{equation}\label{we21}
\|P'_lu(t)\|_{L^2}+\|P'_lZ(t)\|_{L^2}\lesssim \eps_12^{l(1/2-p_0)}.
\end{equation}

To estimate $|J_{2,+}(t)|$, we decompose first
\begin{equation*}
 \begin{split}
&J_{2,+}:=J^+_{21}+J^+_{22}+J^+_{23},\\
&J^+_{21}:=2\Re\frac{1}{4\pi^2}\int_{\mathbb{R}\times\mathbb{R}}m_{2,+}(\xi,\eta)\widehat{\mathcal{N}}(\eta)\widehat{Z}(\xi-\eta)\widehat{\overline{Z}}(-\xi)\,d\xi d\eta,\\
&J^+_{22}:=2\Re\frac{1}{4\pi^2}\int_{\mathbb{R}\times\mathbb{R}}m_{2,+}(\xi,\eta)\widehat{u}(\eta)\widehat{\mathcal{N}_Z}(\xi-\eta)\widehat{\overline{Z}}(-\xi)\,d\xi d\eta,\\
&J^+_{23}:=2\Re\frac{1}{4\pi^2}\int_{\mathbb{R}\times\mathbb{R}}m_{2,-}(\xi,\eta)\widehat{u}(\eta)\widehat{Z}(\xi-\eta)\widehat{\overline{\mathcal{N}_Z}}(-\xi)\,d\xi d\eta.
 \end{split}
\end{equation*}
To estimate $|J^+_{21}(t)|$ we use first Lemma \ref{touse} (ii),
\begin{equation*}
\begin{split}
&|J^+_{21}(t)| \lesssim\Sigma_{211}+\Sigma_{212}+\Sigma_{213},\\
&\Sigma_{211}:=\sum_{k,k_1,k_2\in\mathbb{Z},\,2^{k_1+40}\geq 1}\|m_{2,+}^{k,k_1,k_2}\|_{S^\infty}
\|P'_{k_2}\mathcal{N}(t)\|_{L^\infty}\|P'_{k_1}Z(t)\|_{L^2}\|P'_{k}Z(t)\|_{L^2},\\
&\Sigma_{212}:=\sum_{k,k_1,k_2\in\mathbb{Z},\,2^{k_1+40}\in[(1+t)^{-2},1]}\|m_{2,+}^{k,k_1,k_2}\|_{S^\infty}
\|P'_{k_2}\mathcal{N}(t)\|_{L^\infty}\|P'_{k_1}Z(t)\|_{L^2}\|P'_{k}Z(t)\|_{L^2},\\
&\Sigma_{213}:=\sum_{k,k_1,k_2\in\mathbb{Z},\,2^{k_1+40}\leq (1+t)^{-2}}\|m_{2,+}^{k,k_1,k_2}\|_{S^\infty}
\|P'_{k_2}\mathcal{N}(t)\|_{L^2}\|P'_{k_1}Z(t)\|_{L^\infty}\|P'_{k}Z(t)\|_{L^2}.
\end{split}
\end{equation*}
Then we use Lemma \ref{NonlinEst}, Lemma \ref{SymbBound2}, the bounds \eqref{we21}, together with the simple bound
$\|P'_lf\|_{L^\infty}\lesssim 2^{l/2}\|P'_lf\|_{L^2}$, to estimate
\begin{equation*}
\begin{split}
 \Sigma_{211}&\lesssim \eps_1^4(1+t)^{8p_0-1}\sum_{k,k_1,k_2\in\mathbb{Z},\,2^{k_1+40}\geq 1}\one_{15}(k,k_1,k_2)
2^{(k_2-k_1)/2}2^{2k-2k_1}2^{-3k_2}\lesssim \eps_1^4(1+t)^{8p_0-1},
\end{split}
\end{equation*}
\begin{equation*}
\begin{split}
 \Sigma_{212}&\lesssim \eps_1^4(1+t)^{4p_0-1}\sum_{k,k_1,k_2\in\mathbb{Z},\,2^{k_1+40}\in[(1+t)^{-2},1]} \one_{15}(k,k_1,k_2)
2^{(k_2-k_1)/2}2^{2k}2^{-3k_2}2^{k_1(1/2-p_0)}\\
&\lesssim \eps_1^4(1+t)^{6p_0-1},
\end{split}
\end{equation*}
and
\begin{equation*}
\begin{split}
 \Sigma_{213}&\lesssim \eps_1^4(1+t)^{5p_0-1/2}\sum_{k,k_1,k_2\in\mathbb{Z},\,2^{k_1+40}\leq (1+t)^{-2}} \one_{15}(k,k_1,k_2)
2^{(k_2-k_1)/2}2^{2k}2^{-(N-1)k_2}2^{k_1(1-p_0)}\\
&\lesssim \eps_1^4(1+t)^{-1}.
\end{split}
\end{equation*}
Therefore
\begin{equation*}
 |J^+_{21}(t)|\lesssim \eps_1^4(1+t)^{8p_0-1}.
\end{equation*}

The estimate for $|J^+_{22}(t)|$ is easier,
\begin{equation*}
 \begin{split}
|J^+_{22}(t)|&\lesssim\sum_{k,k_1,k_2\in\mathbb{Z}}\|m_{2,+}^{k,k_1,k_2}\|_{S^\infty}
\|P'_{k_2}u(t)\|_{L^\infty}\|P'_{k_1}\mathcal{N}_Z(t)\|_{L^2}\|P'_{k}Z(t)\|_{L^2}\\
&\lesssim \eps_1^4(1+t)^{8p_0-1}\sum_{k,k_1,k_2\in\mathbb{Z},\,k_1\geq -30}\one_{15}(k,k_1,k_2)2^{(k_2-k_1)/2}2^{2k-2k_1}
2^{-4k_2}2^{k_1}\\
&\lesssim \eps_1^4(1+t)^{8p_0-1}.
 \end{split}
\end{equation*}

The estimate for $|J^+_{23}(t)|$ is similar to the estimate for $|J_{21}(t)|$. We use first Lemma \ref{touse} (ii), and
decompose the sum over the three cases, $2^{k_1+40}\geq 1$, $2^{k_1+40}\in[(1+t)^{-2},1]$, $2^{k_1+40}\leq(1+t)^{-2}$. Then we use,
as before, the $L^2$ bounds $\|P'_{k_1}Z(t)\|_{L^2}\lesssim \eps_1(1+t)^{4p_0}$ in the first case, and the stronger $L^2$
bounds $\|P'_{k_1}Z(t)\|_{L^2}\lesssim \eps_12^{k_1(1/2-p_0)}$ in the other two cases. We proceed as in the estimate of $|J^+_{21}(t)|$
to conclude that $|J^+_{23}(t)|\lesssim \eps_1^4(1+t)^{8p_0-1}$. Therefore
\begin{equation}\label{we25}
 |J_{2,+}(t)|\lesssim \eps_1^4(1+t)^{8p_0-1}.
\end{equation}
The estimate for $J_{2,-}$ is similar.

The estimate for $|J_{3,+}(t)|$ also proceeds along the same line: start by decomposing
\begin{equation*}
 \begin{split}
&J_{3,+}:=J^+_{31}+J^+_{32}+J^+_{33},\\
&J^+_{31}:=2\Re\frac{1}{4\pi^2}\int_{\mathbb{R}\times\mathbb{R}}m_{3,+}(\xi,\eta)\widehat{\mathcal{N}}(\eta)\widehat{u}(\xi-\eta)\widehat{\overline{Z}}(-\xi)\,d\xi d\eta,\\
&J^+_{32}:=2\Re\frac{1}{4\pi^2}\int_{\mathbb{R}\times\mathbb{R}}m_{3,+}(\xi,\eta)\widehat{u}(\eta)\widehat{\mathcal{N}}(\xi-\eta)\widehat{\overline{Z}}(-\xi)\,d\xi d\eta,\\
&J^+_{33}:=2\Re\frac{1}{4\pi^2}\int_{\mathbb{R}\times\mathbb{R}}m_{3,+}(\xi,\eta)\widehat{u}(\eta)\widehat{u}(\xi-\eta)\widehat{\overline{\mathcal{N}_Z}}(-\xi)\,d\xi d\eta.
 \end{split}
\end{equation*}
Then we estimate, as before,
\begin{equation*}
\begin{split}
&|J^+_{31}(t)| \lesssim\Sigma_{311}+\Sigma_{312},\\
&\Sigma_{311}:=\sum_{k,k_1,k_2\in\mathbb{Z},\,2^{k_1+40}\geq (1+t)^{-2}}\|m_{3,+}^{k,k_1,k_2}\|_{S^\infty}
\|P'_{k_2}\mathcal{N}(t)\|_{L^\infty}\|P'_{k_1}u(t)\|_{L^2}\|P'_{k}Z(t)\|_{L^2},\\
&\Sigma_{312}:=\sum_{k,k_1,k_2\in\mathbb{Z},\,2^{k_1+40}\leq (1+t)^{-2}}\|m_{3,+}^{k,k_1,k_2}\|_{S^\infty}
\|P'_{k_2}\mathcal{N}(t)\|_{L^2}\|P'_{k_1}u(t)\|_{L^\infty}\|P'_{k}Z(t)\|_{L^2}.
\end{split}
\end{equation*}

Then we use Lemma \ref{NonlinEst}, Lemma \ref{SymbBound2}, \eqref{we21}, together with the simple bound
$\|P'_lf\|_{L^\infty}\lesssim 2^{l/2}\|P'_lf\|_{L^2}$, to estimate
\begin{equation*}
\begin{split}
 \Sigma_{311}&\lesssim \eps_1^4(1+t)^{5p_0-1}\sum_{k,k_1,k_2\in\mathbb{Z},\,2^{k_1+40}\geq (1+t)^{-2}} \one_{15}(k,k_1,k_2)2^{2k}2^{(k_2-k_1)/2}2^{-3k_2}2^{k_1(1/2-p_0)}\\
&\lesssim \eps_1^4(1+t)^{7p_0-1},
\end{split}
\end{equation*}
and
\begin{equation*}
\begin{split}
 \Sigma_{312}&\lesssim \eps_1^4(1+t)^{5p_0-1/2}\sum_{k,k_1,k_2\in\mathbb{Z},\,2^{k_1+40}\leq (1+t)^{-2}} \one_{15}(k,k_1,k_2)
2^{2k}2^{(k_2-k_1)/2}2^{-(N-1)k_2}2^{k_1(1-p_0)}\\
&\lesssim \eps_1^4(1+t)^{-1}.
\end{split}
\end{equation*}
Therefore
\begin{equation*}
 |J^+_{31}(t)|\lesssim \eps_1^4(1+t)^{8p_0-1}.
\end{equation*}

The estimate for $|J^+_{32}(t)|$ is easier,
\begin{equation*}
 \begin{split}
|J^+_{32}(t)|&\lesssim\sum_{k,k_1,k_2\in\mathbb{Z}}\|m_{3,+}^{k,k_1,k_2}\|_{S^\infty}
\|P'_{k_2}u(t)\|_{L^\infty}\|P'_{k_1}\mathcal{N}(t)\|_{L^2}\|P'_{k}Z(t)\|_{L^2}\\
&\lesssim \eps_1^4(1+t)^{5p_0-1}\sum_{k,k_1,k_2\in\mathbb{Z},\,k_1\geq -30}\one_{15}(k,k_1,k_2)2^{2k}2^{(k_2-k_1)/2}2^{-4k_2}\\
&\lesssim \eps_1^4(1+t)^{5p_0-1}.
 \end{split}
\end{equation*}

Finally, as before, the estimate for $|J^+_{33}(t)|$ is similar to the estimate for $|J^+_{31}(t)|$. Therefore
\begin{equation}\label{we25.5}
|J^+_{3,+}(t)|\lesssim\eps_1^4(1+t)^{8p_0-1}.
\end{equation}
The estimate for $|J^+_{3,+}(t)|$ is similar, which completes the proof of the lemma.
\end{proof}

We can estimate now the expression $J_1$ defined in \eqref{we9}.

\begin{lem}\label{Jest2}
For any $t\in[0,T]$ we have
\begin{equation}\label{we40}
 |J_1(t)|\lesssim \eps_1^4(1+t)^{-1+8p_0}.
\end{equation}
\end{lem}

\begin{proof}[Proof of Lemma \ref{Jest2}] We start by decomposing
\begin{equation*}
\begin{split}
&\mathcal{N}_Z=\mathcal{N}_Z^1+\mathcal{N}_Z^2,\\
&\mathcal{F}(\mathcal{N}_Z^1)(\xi):=\frac{1}{2\pi}\int_{\mathbb{R}}q(\xi,\eta)\widehat{Z}(\eta)\widehat{V}(\xi-\eta)\,d\eta,\\
&\mathcal{F}(\mathcal{N}_Z^2)(\xi):=\frac{1}{2\pi}\int_{\mathbb{R}}q_1(\xi,\eta)\widehat{u}(\eta)\widehat{(Z+\overline{Z})}(\xi-\eta)\,d\eta
+\frac{1}{2\pi}\int_{\mathbb{R}}q_2(\xi,\eta)\widehat{u}(\eta)\widehat{(u+\overline{u})}(\xi-\eta)\,d\eta.
\end{split}
\end{equation*}
The proof of the bound \eqref{we18} in Lemma \ref{NonlinEst} shows that $\mathcal{N}_Z^2$ satisfies stronger $L^2$ estimates than $\mathcal{N}_Z$,
more precisely
\begin{equation}\label{we41}
\begin{split}
&\|P_l\mathcal{N}_Z^2(t)\|_{L^2}\lesssim\eps_1^2(1+t)^{4p_0-1/2}2^{-l},\qquad t\in[0,T],\,l\in\mathbb{Z},\\
&P_l\mathcal{N}_Z^2(t)=0\qquad\text{ if }\qquad l\leq-30.
\end{split}
\end{equation}

Now we decompose
\begin{equation}\label{we42}
\begin{split}
&J_1=J_{11}+J_{12}+J_{13}+J_{14},\\
&J_{11}:=2\Re\frac{1}{4\pi^2}\int_{\mathbb{R}\times\mathbb{R}}m_1(\xi,\eta)\widehat{\mathcal{N}_Z^2}(\eta)\widehat{u}(\xi-\eta)\widehat{\overline{Z}}(-\xi)\,d\xi d\eta,\\
&J_{12}:=2\Re\frac{1}{4\pi^2}\int_{\mathbb{R}\times\mathbb{R}}m_1(\xi,\eta)\widehat{Z}(\eta)\widehat{\mathcal{N}}(\xi-\eta)\widehat{\overline{Z}}(-\xi)\,d\xi d\eta,\\
&J_{13}:=2\Re\frac{1}{4\pi^2}\int_{\mathbb{R}\times\mathbb{R}}m_1(\xi,\eta)\widehat{Z}(\eta)\widehat{u}(\xi-\eta)\widehat{\overline{\mathcal{N}_Z^2}}(-\xi)\,d\xi d\eta,\\
&J_{14}:=2\Re\frac{1}{4\pi^2}\int_{\mathbb{R}\times\mathbb{R}}\Big[m_1(\xi,\eta)\widehat{\mathcal{N}_Z^1}(\eta)\widehat{u}(\xi-\eta)
\widehat{\overline{Z}}(-\xi)+m_1(\xi,\eta)\widehat{Z}(\eta)
\widehat{u}(\xi-\eta)\widehat{\overline{\mathcal{N}_Z^1}}(-\xi)\Big]\,d\xi d\eta.
\end{split}
\end{equation}

The expressions $J_{11}(t),J_{12}(t),J_{13}(t)$ can be estimated using Lemma \ref{touse} (ii). More precisely,
\begin{equation*}
 \begin{split}
  |J_{11}(t)|&\lesssim\sum_{k,k_1,k_2\in\mathbb{Z}}\|m_1^{k,k_1,k_2}\|_{S^\infty}
\|P'_{k_2}\mathcal{N}_Z^2(t)\|_{L^2}\|P'_{k_1}u(t)\|_{L^\infty}\|P'_{k}Z(t)\|_{L^2}\\
&\lesssim \eps_1^4(1+t)^{8p_0-1}\sum_{k,k_1,k_2\in\mathbb{Z}}\one_{15}(k,k_1,k_2)2^{k_1/2}2^{-k_2/2}\cdot
2^{-k_2}2^{-4\max(k_1,0)}\\
&\lesssim \eps_1^4(1+t)^{8p_0-1},
 \end{split}
\end{equation*}
\begin{equation*}
 \begin{split}
  |J_{12}(t)|&\lesssim\sum_{k,k_1,k_2\in\mathbb{Z}}\|m_1^{k,k_1,k_2}\|_{S^\infty}
\|P'_{k_2}Z(t)\|_{L^2}\|P'_{k_1}\mathcal{N}(t)\|_{L^\infty}\|P'_{k}Z(t)\|_{L^2}\\
&\lesssim \eps_1^4(1+t)^{8p_0-1}\sum_{k,k_1,k_2\in\mathbb{Z},\,k_1\geq -30}\one_{15}(k,k_1,k_2)2^{k_1/2}2^{-k_2/2}\cdot
2^{-3k_1}\\
&\lesssim \eps_1^4(1+t)^{8p_0-1},
 \end{split}
\end{equation*}
and
\begin{equation*}
 \begin{split}
  |J_{13}(t)|&\lesssim\sum_{k,k_1,k_2\in\mathbb{Z}}\|m_1^{k,k_1,k_2}\|_{S^\infty}
\|P'_{k_2}Z(t)\|_{L^2}\|P'_{k_1}u(t)\|_{L^\infty}\|P'_{k}\mathcal{N}_Z^2(t)\|_{L^2}\\
&\lesssim \eps_1^4(1+t)^{8p_0-1}\sum_{k,k_1,k_2\in\mathbb{Z}}\one_{15}(k,k_1,k_2)2^{k_1/2}2^{-k_2/2}\cdot
2^{-4\max(k_1,0)}2^{-k}\\
&\lesssim \eps_1^4(1+t)^{8p_0-1}.
 \end{split}
\end{equation*}

To estimate $J_{14}(t)$ we expand the nonlinearity $\mathcal{N}_Z^1$ according to its definition. The estimate of $|J_{14}(t)|$ is similar to the estimate of $|A_2(t)|$ in the proof of Lemma \ref{EnergyInc}, after replacing $W$ by $Z$ and $q_N,m_N$ by $q,m_1$.
\end{proof}

We show now how to control the cubic corrections $E_{w,l}^{(3)}$. More precisely:

\begin{lem}\label{EnergyComp2}
Assuming the bounds \eqref{pr5}--\eqref{pr7}, for any $t\in[0,T]$ we have
\begin{equation}\label{we50}
|E_{w,l}^{(3)}(t)|\lesssim \eps_1^3(1+t)^{8p_0},\qquad l\in\{1,(2,+),(2,-),(3,+), (3,-)\}.
\end{equation}
\end{lem}

\begin{proof}[Proof of Lemma \ref{EnergyComp2}] We estimate, using Lemma \ref{touse} and Lemma \ref{SymbBound2},
\begin{equation*}
\begin{split}
|E_{w,1}^{(3)}(t)|&\lesssim\sum_{k,k_1,k_2\in\mathbb{Z}}\|m_1^{k,k_1,k_2}\|_{S^\infty}\|P'_kZ(t)\|_{L^2}\|P'_{k_2}Z(t)\|_{L^2}\|P'_{k_1}u(t)\|_{L^\infty}\lesssim \eps_1^3(1+t)^{8p_0}.
\end{split}
\end{equation*}

We use also the bounds \eqref{we21}. We estimate
\begin{equation*}
\begin{split}
&|E_{w,2,\iota}^{(3)}(t)|\lesssim \Sigma_{21}^{(3)}+\Sigma_{22}^{(3)},\\
&\Sigma_{21}^{(3)}:=\sum_{k,k_1,k_2\in\mathbb{Z},\,2^{k_1+40}\geq (1+t)^{-1/2}}\|m_{2,\iota}^{k,k_1,k_2}\|_{S^\infty}\|P'_kZ(t)\|_{L^2}\|P'_{k_2}u(t)\|_{L^\infty}\|P'_{k_1}Z(t)\|_{L^2},\\
&\Sigma_{22}^{(3)}:=\sum_{k,k_1,k_2\in\mathbb{Z},\,2^{k_1+40}\leq (1+t)^{-1/2}}\|m_{2,\iota}^{k,k_1,k_2}\|_{S^\infty}\|P'_kZ(t)\|_{L^2}\|P'_{k_2}u(t)\|_{L^2}\|P'_{k_1}Z(t)\|_{L^\infty}.
\end{split}
\end{equation*}
Then we estimate
\begin{equation*}
\begin{split}
\Sigma_{21}^{(3)}&\lesssim\eps_1^3(1+t)^{8p_0}\sum_{k,k_1,k_2\in\mathbb{Z},\,2^{k_1+40}\geq (1+t)^{-1/2}}\one_{15}(k,k_1,k_2)2^{(k_2-k_1)/2}2^{2k-2\max(k_1,0)}(1+t)^{-1/2}2^{-4k_2}\\
&\lesssim \eps_1^3(1+t)^{8p_0}
\end{split}
\end{equation*}
and
\begin{equation*}
\begin{split}
\Sigma_{22}^{(3)}&\lesssim\eps_1^3(1+t)^{5p_0}\sum_{k,k_1,k_2\in\mathbb{Z},\,2^{k_1+40}\leq (1+t)^{-1/2}}\one_{15}(k,k_1,k_2)2^{(k_2-k_1)/2}2^{2k-2\max(k_1,0)}2^{-Nk_2}2^{k_1(1-p_0)}\\
&\lesssim \eps_1^3(1+t)^{8p_0}.
\end{split}
\end{equation*}
This proves the desired bound \eqref{we50} for $l\in\{(2,+),(2,-)\}$.

On the other hand, for $l\in\{(3,+),(3,-)\}$ we estimate
\begin{equation*}
\begin{split}
|E_{w,3,\iota}^{(3)}(t)|&\lesssim \sum_{k,k_1,k_2\in\mathbb{Z}}\|m_{3,\iota}^{k,k_1,k_2}\|_{S^\infty}\|P'_kZ(t)\|_{L^2}\|P'_{k_2}u(t)\|_{2}\|P'_{k_1}u(t)\|_{L^\infty}\\
&\lesssim \eps_1^3(1+t)^{6p_0}\sum_{k,k_1,k_2\in\mathbb{Z}}\one_{15}(k,k_1,k_2)2^{(k_2-k_1)/2}2^{2k}2^{-Nk_2}2^{k_1(1-p_0)}\\
&\lesssim \eps_1^3(1+t)^{6p_0}.
\end{split}
\end{equation*}
This completes the proof of the lemma.
\end{proof}

\begin{proof}[Proof of Proposition \ref{EnEst2}] We can now complete the proof of Proposition \ref{EnEst2}. The proposition follows from the formula \eqref{we14}, and Lemmas \ref{Jest}, \ref{Jest2}, and \ref{EnergyComp2}, by the same simple argument as in the proof of Proposition \ref{EnEst1} in section \ref{energy}.
\end{proof}

\section{The $L^\infty$ decay estimate}\label{sup}

In this section we prove the following:

\begin{pro}\label{Sup1}
If $u$ satisfies \eqref{pr5}--\eqref{pr7} then
\begin{equation}\label{BigOne2}
 \sup_{k\in\mathbb{Z}}\sup_{t\in[0,T]}(1+t)^{1/2}2^{4\max(k,0)}\|P_ku(t)\|_{L^\infty}\lesssim\eps_0.
\end{equation}
\end{pro}

In other words, we improve the control of the $L^\infty$ norm in the bootstrap assumption \eqref{pr6}. We will prove this proposition
in the next two sections. In this section we show how to reduce the proof of Proposition \ref{Sup1} to the more
technical Lemma \ref{Zcontrol}.

\subsection{Two lemmas}\label{linest}

We prove first two general estimates that are used often in the rest of the paper. The first lemma is our main linear dispersive estimate:

\begin{lem}\label{dispersive}
 For any $t\in\mathbb{R}\setminus\{0\}$, $k\in\mathbb{Z}$, and $f\in L^2(\mathbb{R})$ we have
\begin{equation}\label{disperse}
 \|e^{i t \Lambda}P_kf\|_{L^\infty}\lesssim |t|^{-1/2}2^{k/4}\|\widehat{f}\|_{L^\infty}+|t|^{-3/5}2^{-2k/5}\big[2^k\|\partial \widehat{f}\|_{L^2}+\|\widehat{f}\|_{L^2}\big]
\end{equation}
and
\begin{equation}\label{disperseEa}
 \|e^{i t \Lambda}P_kf\|_{L^\infty}\lesssim |t|^{-1/2}2^{k/4}\|f\|_{L^1}.
\end{equation}
\end{lem}

\begin{proof}[Proof of Lemma \ref{dispersive}] This is similar to the proof of Lemma 2.3 in \cite{IoPu1}. By scale invariance we may assume that $k=0$. The bound \eqref{disperseEa} is a standard dispersive estimate. For \eqref{disperse} it suffices to prove that
\begin{equation}\label{disp1}
\Big|\int_{\mathbb{R}}e^{it|\xi|^{3/2}}e^{ix\xi}\widehat{f}(\xi)\varphi_0(\xi)\,d\xi\Big|\lesssim |t|^{-1/2}\|\widehat{f}\|_{L^\infty}+|t|^{-3/5}\big[\|\partial \widehat{f}\|_{L^2}+\|\widehat{f}\|_{L^2}\big]
\end{equation}
for any $t\in\mathbb{R}\setminus\{0\}$ and $x\in\mathbb{R}$. Clearly,
\begin{equation*}
 \Big|\int_{\mathbb{R}}e^{it|\xi|^{3/2}}e^{ix\xi}\widehat{f}(\xi)\varphi_0(\xi)\,d\xi\Big|\lesssim \|\widehat{f}\|_{L^2}.
\end{equation*}
Therefore in proving \eqref{disp1} we may assume $|t|\geq 2^{10}$.

Let $\Psi(\xi):=t|\xi|^{3/2}+x\xi$ and notice that
\begin{equation}\label{disp2.02}
\Psi'(\xi)=\frac{3}{2}t\xi|\xi|^{-1/2}+x\qquad\text{ and }\qquad |\Psi''(\xi)|\approx |t||\xi|^{-1/2}.
\end{equation}
If $|x|\notin [|t|2^{-10},|t|2^{10}]$ then we integrate by parts in $\xi$ to estimate the left-hand side of \eqref{disp1} by
\begin{equation*}
C\int_{\mathbb{R}}\frac{1}{|t|}\big[|\partial\widehat{f}(\xi)|+|\widehat{f}(\xi)|\big]\varphi'_0(\xi)\,d\xi\lesssim |t|^{-1}\big[\|\partial \widehat{f}\|_{L^2}+\|\widehat{f}\|_{L^2}\big],
\end{equation*}
which suffices to prove \eqref{disp1}, in view of the assumption $|t|\geq 1$.

It remains to prove the bound \eqref{disp1} when $|x|\in [|t|2^{-10},|t|2^{10}]$. Let $\xi_0\in\mathbb{R}$ denote the unique
solution of the equation $\Psi'(\xi)=0$, i.e.
\begin{equation*}
 \xi_0:=\mathrm{sign}(-x/t)\frac{4x^2}{9t^2}.
\end{equation*}
and notice that $|\xi_0|\approx 1$. Let $l_0$ denote the smallest integer with the property that $2^{l_0}\geq |t|^{-1/2}$
and estimate
\begin{equation}\label{disp6}
\Big|\int_{\mathbb{R}}e^{it|\xi|^{3/2}}e^{ix\xi}\widehat{f}(\xi)\varphi_0(\xi)\,d\xi\Big|\leq\sum_{l=l_0}^{100}|J_l|,
\end{equation}
where, for any $l\geq l_0$,
\begin{equation*}
J_{l}:=\int_{\mathbb{R}}e^{i\Psi(\xi)}\cdot \widehat{P_0f}(\xi)\varphi_l^{(l_0)}(\xi-\xi_0)\,d\xi.
\end{equation*}
Clearly
\begin{equation*}
 |J_{l_0}|\lesssim 2^{l_0}\|\widehat{P_0f}\|_{L^\infty}\lesssim |t|^{-1/2}\|\widehat{f}\|_{L^\infty}.
\end{equation*}
Moreover, since $|\Psi'(\xi)|\gtrsim |t|2^l$ whenever $|\xi|\approx 1$ and $|\xi-\xi_0|\approx 2^l$, for $l\geq l_0+1$
we can integrate by parts to estimate
\begin{equation*}
\begin{split}
|J_l|&\lesssim \frac{1}{|t|2^l}\big[2^{-l}\|\widehat{P_0f}(\xi)\cdot\mathbf{1}_{[0,2^{l+4}]}(|\xi-\xi_0|)\|_{L^1_\xi}+
\|\partial(\widehat{P_0f})(\xi)\cdot\mathbf{1}_{[0,2^{l+4}]}(|\xi-\xi_0|)\|_{L^1_\xi}\big]\\
&\lesssim |t|^{-1}2^{-l}\big[\|\widehat{P_0f}\|_{L^\infty_\xi}+
2^{l/2}\|\partial(\widehat{P_0f})\|_{L^2}\big].
\end{split}
\end{equation*}
The desired bound \eqref{disp1} follows from \eqref{disp6} and the last two estimates. This completes the proof of the lemma.
\end{proof}

We also need the following interpolation lemma:

\begin{lem}\label{interpolation}
 For any $k\in\mathbb{Z}$, and $f\in L^2(\mathbb{R})$ we have
\begin{equation}\label{interp1}
 \big\|\widehat{P_kf}\big\|^2_{L^\infty}\lesssim \big\|P_kf\big\|^2_{L^1}\lesssim
2^{-k}\|\widehat{f}\|_{L^2}\big[2^k\|\partial \widehat{f}\|_{L^2}+\|\widehat{f}\|_{L^2}\big].
\end{equation}
\end{lem}

\begin{proof}[Proof of Lemma \ref{interpolation}] By scale invariance we may assume that $k=0$. It suffices to prove that
\begin{equation}\label{interp2}
 \big\|P_0f\big\|^2_{L^1}\lesssim \|\widehat{f}\|_{L^2}\big[\|\partial \widehat{f}\|_{L^2}+\|\widehat{f}\|_{L^2}\big].
\end{equation}
For $R\geq 1$ we estimate
\begin{equation*}
\begin{split}
\big\|P_0f\big\|_{L^1}&\lesssim \int_{|x|\leq R}|P_0f(x)|\,dx+\int_{|x|\geq R}|xP_0f(x)|\cdot \frac{1}{|x|}\,dx\\
&\lesssim R^{1/2}\|P_0f\|_{L^2}+R^{-1/2}\|xP_0f(x)\|_{L^2_x}\\
&\lesssim R^{1/2}\|\widehat{f}\|_{L^2}+R^{-1/2}\big[\|\partial \widehat{f}\|_{L^2}+\|\widehat{f}\|_{L^2}\big].
\end{split}
\end{equation*}
The desired estimate \eqref{interp2} follows by choosing $R$ suitably.
\end{proof}

\subsection{The normal form transformation}\label{NormalForm}
Recall that for any suitable multiplier $m:\mathbb{R}\to\mathbb{C}$ we define the associated bilinear operator $M$ by the formula
\begin{equation*}
\mathcal{F}\big[M(f,g)\big](\xi)=\frac{1}{2\pi}\int_{\mathbb{R}}m(\xi,\eta)\widehat{f}(\xi-\eta)\widehat{g}(\eta)\,d\eta.
\end{equation*}
With this notation, our basic equation is
\begin{equation}\label{nf1}
\partial_tu+i\Lambda u=Q_0(V,u),\qquad V=u+\overline{u},
\end{equation}
where $Q_0$ is the bilinear operator defined by the Fourier multiplier $q_0$ in \eqref{pr2}.

We would like to define a modified variable $v$ that satisfies a cubic equation. Let
\begin{equation}\label{nf3}
v:=u+A(u,u)+B(\overline{u},u),
\end{equation}
where $A$ and $B$ are the bilinear operators defined by the multipliers
\begin{equation}\label{nf4}
\begin{split}
&a(\xi,\eta):=\frac{iq_0(\xi,\eta)}{|\xi|^{3/2}-|\xi-\eta|^{3/2}-|\eta|^{3/2}},\\
&b(\xi,\eta):=\frac{iq_0(\xi,\eta)}{|\xi|^{3/2}+|\xi-\eta|^{3/2}-|\eta|^{3/2}}.
\end{split}
\end{equation}
It is easy to see that $v$ verifies the equation
\begin{equation}\label{nf5}
\partial_tv+i\Lambda v=A(Q_0(V,u),u)+A(u,Q_0(V,u))+B(\overline{Q_0(V,u)},u)+B(\overline{u},Q_0(V,u)).
\end{equation}

We prove first several bounds on the modified variable $v$.

\begin{lem}\label{Lem1}
For any $t\in[0,T]$ and $k\geq -30$ we have
\begin{equation}\label{nf10}
\|P_k(u(t)-v(t))\|_{L^\infty}\lesssim\eps_1^22^{-7k/2}(1+t)^{-3/4+2p_0},
\end{equation}
\begin{equation}\label{nf11}
\|P_k(u(t)-v(t))\|_{L^2}\lesssim\eps_1^22^{-(N-1/2)k}(1+t)^{-1/4+3p_0},
\end{equation}
and
\begin{equation}\label{nf11.1}
\|P_kS(u(t)-v(t))\|_{L^2}\lesssim\eps_1^22^{-3k/2}(1+t)^{-1/4+6p_0}.
\end{equation}
Moreover
\begin{equation}\label{nf11.2}
P_k(u(t)-v(t))=0\qquad\text{ if }\qquad k \leq -30.
\end{equation}
\end{lem}

\begin{proof}[Proof of Lemma \ref{Lem1}] The identities \eqref{nf11.2} follow from the definitions.
To prove \eqref{nf10} and \eqref{nf11} we notice first that
\begin{equation}\label{nf12}
\|a^{k,k_1,k_2}\|_{S^\infty}+\|b^{k,k_1,k_2}\|_{S^\infty}\lesssim 2^{(k_2-k_1)/2}\one_{15}(k,k_1,k_2),
\end{equation}
as a consequence of \eqref{en21.2} and \eqref{en25}. Recall the bounds \eqref{pr5}--\eqref{pr7},
\begin{equation}\label{nf13}
\begin{split}
&\|P_lu(t)\|_{L^2}\lesssim\eps_1(1+t)^{p_0}\min\big[2^{-Nl},2^{l(1/2-p_0)}\big],\\
&\|P_lu(t)\|_{L^\infty}\lesssim \eps_1(1+t)^{-1/2}2^{-4\max(l,0)},\\
&\|P_lSu(t)\|_{L^2}\lesssim\eps_1(1+t)^{4p_0}\min\big[2^{-2l},2^{l(1/2-p_0)}\big],
\end{split}
\end{equation}
for any $l\in\mathbb{Z}$ and $t\in[0,T]$. Therefore, using Lemma \ref{touse} (ii), for any $k\in\mathbb{Z}$, $k\geq -30$,
\begin{equation*}
\begin{split}
\|P_k&A(u(t),u(t))\|_{L^2}\lesssim \sum_{k_1,k_2\in\mathbb{Z}}\|a^{k,k_1,k_2}\|_{S^\infty}\|P'_{k_1}u(t)\|_{L^\infty}\|P'_{k_2}u(t)\|_{L^2}\\
&\lesssim \eps_1^2(1+t)^{-1/2+p_0}\sum_{k_1,k_2\in\mathbb{Z},\,2^{k_1}\geq (1+t)^{-1/2}}\one_{15}(k,k_1,k_2)2^{(k_2-k_1)/2}2^{-Nk_2}2^{-4\max(k_1,0)}\\
&+\eps_1^2(1+t)^{2p_0}\sum_{k_1,k_2\in\mathbb{Z},\,2^{k_1}\leq (1+t)^{-1/2}}\one_{15}(k,k_1,k_2)2^{(k_2-k_1)/2}2^{-Nk_2}2^{k_1/2}2^{k_1(1/2-p_0)}\\
&\lesssim \eps_1^22^{-(N-1/2)k}(1+t)^{-1/4+3p_0},
\end{split}
\end{equation*}
and
\begin{equation*}
\begin{split}
\|P_k&A(u(t),u(t))\|_{L^\infty}\lesssim \sum_{k_1,k_2\in\mathbb{Z}}\|a^{k,k_1,k_2}\|_{S^\infty}\|P'_{k_1}u(t)\|_{L^\infty}\|P'_{k_2}u(t)\|_{L^\infty}\\
&\lesssim \eps_1^2(1+t)^{-1}\sum_{k_1,k_2\in\mathbb{Z},\,2^{k_1}\geq (1+t)^{-1/2}}\one_{15}(k,k_1,k_2)2^{(k_2-k_1)/2}2^{-4k_2}2^{-4\max(k_1,0)}\\
&+\eps_1^2(1+t)^{-1/2+p_0}\sum_{k_1,k_2\in\mathbb{Z},\,2^{k_1}\leq (1+t)^{-1/2}}\one_{15}(k,k_1,k_2)2^{(k_2-k_1)/2}2^{-4k_2}2^{k_1/2}2^{k_1(1/2-p_0)}\\
&\lesssim \eps_1^22^{-3.5k}(1+t)^{-3/4+2p_0}.
\end{split}
\end{equation*}
The bounds for $\|P_kB(\overline{u}(t),u(t))\|_{L^2}$ and  $\|P_kB(\overline{u}(t),u(t))\|_{L^\infty}$ are similar, since the bounds for the symbols $a$ and $b$ in \eqref{nf12} are similar. Therefore
\begin{equation}\label{nf14}
\begin{split}
&\|P_kA(u(t),u(t))\|_{L^2}+\|P_kB(\overline{u}(t),u(t))\|_{L^2}\lesssim \eps_1^22^{-(N-1/2)k}(1+t)^{-1/4+3p_0},\\
&\|P_kA(u(t),u(t))\|_{L^\infty}+\|P_kB(\overline{u}(t),u(t))\|_{L^\infty}\lesssim \eps_1^22^{-3.5k}(1+t)^{-3/4+2p_0},
\end{split}
\end{equation}
for any $k\geq -30$.
The bounds \eqref{nf10} and \eqref{nf11} follow.

To prove \eqref{nf11.1} we have to calculate $S(A(u,u))$ and $S(B(\overline{u},u))$. Using the definitions, we calculate for any suitable functions $f$ and $g$,
\begin{equation}\label{nf15.9}
\begin{split}
2\pi &\mathcal{F}[SA(f(t),g(t))](\xi)=\Big[\frac{3}{2}t\partial_t-\xi\partial_\xi-I\Big]\int_{\mathbb{R}}a(\xi,\eta)\widehat{f}(\xi-\eta,t)\widehat{g}(\eta,t)\,d\eta\\
&=\int_{\mathbb{R}}a(\xi,\eta)\big[\widehat{Sf}(\xi-\eta,t)-\eta\partial\widehat{f}(\xi-\eta,t)\big]\widehat{g}(\eta,t)\,d\eta\\
&+\int_{\mathbb{R}}a(\xi,\eta)\widehat{f}(\xi-\eta,t)\frac{3}{2}t\partial_t\widehat{g}(\eta,t)\,d\eta-\int_{\mathbb{R}}\xi\partial_\xi a(\xi,\eta)\widehat{f}(\xi-\eta,t)\widehat{g}(\eta,t)\,d\eta\\
&=\int_{\mathbb{R}}a(\xi,\eta)\widehat{Sf}(\xi-\eta,t)\widehat{g}(\eta,t)\,d\eta+\int_{\mathbb{R}}a(\xi,\eta)\widehat{f}(\xi-\eta,t)\widehat{Sg}(\eta,t)\,d\eta\\
&-\int_{\mathbb{R}}(\xi\partial_\xi+\eta\partial_\eta)a(\xi,\eta)\widehat{f}(\xi-\eta,t)\widehat{g}(\eta,t)\,d\eta.
\end{split}
\end{equation}
A similar calculation holds for $S(B(\overline{u},u))$. Therefore
\begin{equation}\label{nf16}
\begin{split}
&SA(u,u)=A(Su,u)+A(u,Su)-\widetilde{A}(u,u),\\
&SB(\overline{u},u)=B(S\overline{u},u)+B(\overline{u},Su)-\widetilde{B}(\overline{u},u),
\end{split}
\end{equation}
where $\widetilde{A},\widetilde{B}$ are the bilinear operators associated to the multipliers
\begin{equation}\label{nf17}
\widetilde{a}(\xi,\eta):=(\xi\partial_\xi+\eta\partial_\eta)a(\xi,\eta),\qquad \widetilde{b}(\xi,\eta):=(\xi\partial_\xi+\eta\partial_\eta)b(\xi,\eta).
\end{equation}

We examine now the symbols $\widetilde{a}$ and $\widetilde{b}$. Recalling the definition \eqref{nf4}, we have
\begin{equation*}
a(\xi,\eta)=-\chi(\xi-\eta,\eta)a_1(\xi,\eta),\qquad b(\xi,\eta)=-\chi(\xi-\eta,\eta)b_1(\xi,\eta),
\end{equation*}
where
\begin{equation*}
a_1(\xi,\eta):=\frac{|\xi-\eta|^{1/2}\eta}{|\xi|^{3/2}-|\xi-\eta|^{3/2}-|\eta|^{3/2}},\qquad b_1(\xi,\eta):=\frac{|\xi-\eta|^{1/2}\eta}{|\xi|^{3/2}+|\xi-\eta|^{3/2}-|\eta|^{3/2}}.
\end{equation*}
The symbols $a_1$ and $b_1$ are homogeneous of degree $0$, i. e. $a_1(\lambda\xi,\lambda\eta)=a_1(\xi,\eta)$ and $b_1(\lambda\xi,\lambda\eta)=b_1(\xi,\eta)$ for any $\lambda\in(0,\infty)$, therefore
\begin{equation*}
(\xi\partial_\xi+\eta\partial_\eta)a_1=(\xi\partial_\xi+\eta\partial_\eta)b_1\equiv 0.
\end{equation*}
Therefore, the symbols $\widetilde{a}$ and $\widetilde{b}$ satisfy the same bounds as the symbols $a$ and $b$, i.e.
\begin{equation*}
\|\widetilde{a}^{k,k_1,k_2}\|_{S^\infty} + \| \widetilde{b}^{k,k_1,k_2}\|_{S^\infty} \lesssim 2^{(k_2-k_1)/2}\one_{15}(k,k_1,k_2),
\end{equation*}
for any $k,k_1,k_2\in\mathbb{Z}$.

Therefore, as in the proof of \eqref{nf14},
\begin{equation}\label{nf18}
\|P_k\widetilde{A}(u(t),u(t))\|_{L^2}+\|P_k\widetilde{B}(\overline{u}(t),u(t))\|_{L^2}\lesssim \eps_1^22^{-(N-1/2)k}(1+t)^{-1/4+3p_0}
\end{equation}
for any $k\geq -30$ and $t\in[0,T]$. Moreover, using \eqref{nf12}, \eqref{nf13}, and Lemma \ref{touse} (ii), we estimate as before, for $k\geq -30$,
\begin{equation*}
\begin{split}
\|P_k&A(u(t),Su(t))\|_{L^2}\lesssim \sum_{k_1,k_2\in\mathbb{Z}}\|a^{k,k_1,k_2}\|_{S^\infty}\|P'_{k_1}u(t)\|_{L^\infty}\|P'_{k_2}Su(t)\|_{L^2}\\
&\lesssim \eps_1^2(1+t)^{-1/2+4p_0}\sum_{k_1,k_2\in\mathbb{Z},\,2^{k_1}\geq (1+t)^{-1/2}}\one_{15}(k,k_1,k_2)2^{(k_2-k_1)/2}2^{-2k_2}\\
&+\eps_1^2(1+t)^{5p_0}\sum_{k_1,k_2\in\mathbb{Z},\,2^{k_1}\leq (1+t)^{-1/2}}\one_{15}(k,k_1,k_2)2^{(k_2-k_1)/2}2^{-2k_2}2^{k_1/2}2^{k_1(1/2-p_0)}\\
&\lesssim \eps_1^22^{-3k/2}(1+t)^{-1/4+6p_0},
\end{split}
\end{equation*}
and
\begin{equation*}
\begin{split}
\|P_k&A(Su(t),u(t))\|_{L^2}\\
&\lesssim \sum_{k_1,k_2\in\mathbb{Z}}\|a^{k,k_1,k_2}\|_{S^\infty}\min\big[\|P'_{k_1}Su(t)\|_{L^2}\|P'_{k_2}u(t)\|_{L^\infty},\|P'_{k_1}Su(t)\|_{L^\infty}\|P'_{k_2}u(t)\|_{L^2}\big]\\
&\lesssim \eps_1^2(1+t)^{-1/2+4p_0}\sum_{k_1,k_2\in\mathbb{Z},\,2^{k_1}\geq (1+t)^{-1/2}}\one_{15}(k,k_1,k_2)2^{(k_2-k_1)/2}2^{-4k_2}\\
&+\eps_1^2(1+t)^{5p_0}\sum_{k_1,k_2\in\mathbb{Z},\,2^{k_1}\leq (1+t)^{-1/2}}\one_{15}(k,k_1,k_2)2^{(k_2-k_1)/2}2^{-Nk_2}2^{k_1/2}2^{k_1(1/2-p_0)}\\
&\lesssim \eps_1^22^{-3k}(1+t)^{-1/4+6p_0}.
\end{split}
\end{equation*}
The bounds for $\|P_kB(S\overline{u}(t),u(t))\|_{L^2}$ and  $\|P_kB(\overline{u}(t),Su(t))\|_{L^2}$ are similar, and the bound \eqref{nf11.1} follows.
\end{proof}

\begin{rem}\label{extra1}
(i) For later use, we notice that the bounds \eqref{nf10}--\eqref{nf11} and Sobolev embedding also show that
\begin{equation}\label{nf11.6}
\|P_k(u(t)-v(t))\|_{L^\infty}\lesssim\eps_1^22^{-4k}(1+t)^{-1/2}.
\end{equation}

(ii) We record also an additional bound on $\|P_k(u(t)-v(t))\|_{L^2}$,
\begin{equation}\label{extra2}
\|P_k(u(t)-v(t))\|_{L^2}\lesssim \eps_1^22^{-4k}(1+t)^{-1/2+6p_0}.
\end{equation}
This alternative bound, which provides faster decay in time but slower decay at high frequencies, is used in subsection \ref{prooftech3}. To prove it, we estimate as before, for $k\geq -30$
\begin{equation*}
\|P_kA(u(t),u(t))\|_{L^2}\lesssim I+II,
\end{equation*}
where
\begin{equation*}
\begin{split}
I:=&\sum_{k_1,k_2\in\mathbb{Z},\,2^{k_1}\geq (1+t)^{-4}}\|a^{k,k_1,k_2}\|_{S^\infty}\|P'_{k_1}u(t)\|_{L^2}\|P'_{k_2}u(t)\|_{L^\infty}\\
&\lesssim \eps_1^2(1+t)^{-1/2+p_0}\sum_{k_1,k_2\in\mathbb{Z},\,2^{k_1}\geq (1+t)^{-4}}\one_{15}(k,k_1,k_2)2^{(k_2-k_1)/2}2^{-4k_2}2^{k_1(1/2-p_0)}\\
&\lesssim \eps_1^22^{-4k}(1+t)^{-1/2+6p_0},
\end{split}
\end{equation*}
and
\begin{equation*}
\begin{split}
II:=&\sum_{k_1,k_2\in\mathbb{Z},\,2^{k_1}\leq (1+t)^{-4}}\|a^{k,k_1,k_2}\|_{S^\infty}\|P'_{k_1}u(t)\|_{L^\infty}\|P'_{k_2}u(t)\|_{L^2}\\
&\lesssim \eps_1^2(1+t)^{p_0}\sum_{k_1,k_2\in\mathbb{Z},\,2^{k_1}\leq (1+t)^{-4}}\one_{15}(k,k_1,k_2)2^{(k_2-k_1)/2}2^{-Nk_2}2^{k_1(1-p_0)}\\
&\lesssim \eps_1^22^{-(N-1)k}(1+t)^{-1},
\end{split}
\end{equation*}
The bounds on $P_kB(\overline{u}(t),u(t))$ are similar, and the desired bounds \eqref{extra2} follow.
\end{rem}

\subsection{The profile $f$}\label{Profile}
For $t\in[0,T]$ we define the profile
\begin{equation}\label{nf20}
f(t):=e^{it\Lambda}v(t).
\end{equation}
Our proposition below summarizes the main properties of the function $f$.

\begin{pro}\label{Sup3}
If $u$ satisfies \eqref{pr5}--\eqref{pr7} and $v,f$ are defined as before then
\begin{equation}\label{nf22}
\begin{split}
&e^{-it\Lambda}\partial_tf=\partial_tv+i\Lambda v=\mathcal{N}',\\
&\mathcal{N}':=A(Q_0(V,u),u)+A(u,Q_0(V,u))+B(\overline{Q_0(V,u)},u)+B(\overline{u},Q_0(V,u)).
\end{split}
\end{equation}
Moreover, for any $t\in[0,T]$ and $k\in\mathbb{Z}$,
\begin{equation}\label{nf23}
\big\|P_k[e^{-it\Lambda}f(t)]\big\|_{L^\infty}\lesssim\eps_1(1+t)^{-1/2}2^{-4\max(k,0)},
\end{equation}
\begin{equation}\label{nf24}
\|P_kf(t)\|_{L^2}\lesssim\eps_0(1+t)^{6p_0}\min\big(2^{-(N-1)k},2^{k(1/2-p_0)}\big),
\end{equation}
and
\begin{equation}\label{nf25}
\|P_k(x\partial_xf(t))\|_{L^2}\lesssim\eps_0(1+t)^{6p_0}\min\big(2^{-k},2^{k(1/2-p_0)}\big).
\end{equation}
\end{pro}

\begin{proof}[Proof of Proposition \ref{Sup3}] The identity \eqref{nf22} follows from definitions and \eqref{nf5}.
The bound \eqref{nf24} follows
from the assumption \eqref{pr7}, the bound \eqref{nf11}, the identities \eqref{nf11.2}, and Proposition \ref{EnEst1}. The bound \eqref{nf23} follows from the assumption \eqref{pr6} and the bounds \eqref{nf11.6}.

To prove \eqref{nf25} we start from the identity
\begin{equation*}
Sv=e^{-it\Lambda}(x\partial_xf)+(3/2)te^{-it\Lambda}(\partial_tf),
\end{equation*}
which is a consequence of the commutation identity $[S,e^{-it\Lambda}]=0$. Therefore, for any $t\in[0,T]$ and $k\in\mathbb{Z}$,
\begin{equation}\label{nf26}
\|P_k(x\partial_xf(t))\|_{L^2}\lesssim \|P_k(Sv(t))\|_{L^2}+(1+t)\|P_k\mathcal{N}'(t)\|_{L^2}.
\end{equation}
In view of the support properties of the symbols $a$ and $b$,
\begin{equation*}
\|P_k\mathcal{N}'(t)\|_{L^2}=0\qquad\text{ and }\qquad\|P_k(Sv(t))\|_{L^2}=\|P_k(Su(t))\|_{L^2}=\|P_k(Su_0)\|_{L^2}
\end{equation*}
if $k\leq -30$. The desired bound \eqref{nf25} follows from \eqref{pr7} if $k\leq -30$. On the other hand, if $k\geq -30$,
we use \eqref{nf26} again, together with Proposition \ref{EnEst2} and \eqref{nf11}, to estimate
\begin{equation*}
\begin{split}
(1+t)^{-6p_0}2^{k}\|P_k(x\partial_xf(t))\|_{L^2}&\lesssim (1+t)^{-6p_0}2^{k}\big[\|P_k(Sv(t))\|_{L^2}+(1+t)\|P_k\mathcal{N}'(t)\|_{L^2}\big]\\
&\lesssim (1+t)^{-6p_0}2^{k}(1+t)\|P_k\mathcal{N}'(t)\|_{L^2}+(\eps_0+\eps_1^{3/2}).
\end{split}
\end{equation*}
Therefore, for \eqref{nf25} it remains to prove that, for any $t\in[0,T]$ and $k\geq -30$
\begin{equation}\label{nf27}
(1+t)^{1-6p_0}2^{k}\|P_k\mathcal{N}'(t)\|_{L^2}\lesssim \eps_1^2.
\end{equation}

Recall that $\mathcal{N}=Q_0(V,u)$, and the bounds in Lemma \ref{NonlinEst},
\begin{equation}\label{nf28}
\begin{split}
&\|P_l\mathcal{N}(t)\|_{L^2}\lesssim\eps_1^2(1+t)^{p_0-1/2}2^{-(N-1)l},\\
&\|P_l\mathcal{N}(t)\|_{L^\infty}\lesssim\eps_1^2(1+t)^{-1}2^{-3l},\\
&P_l\mathcal{N}=0\qquad\text{ if }\quad l\leq -30.
\end{split}
\end{equation}
Recall also the bounds \eqref{nf13} and the symbol bounds \eqref{nf12}. Using Lemma \ref{touse} (ii), we estimate, for any $k\geq -30$,
\begin{equation*}
\begin{split}
\|P_k&A(\mathcal{N}(t),u(t))\|_{L^2}\lesssim \sum_{k_1,k_2\in\mathbb{Z}}\|a^{k,k_1,k_2}\|_{S^\infty}\|P'_{k_1}\mathcal{N}(t)\|_{L^\infty}\|P'_{k_2}u(t)\|_{L^2}\\
&\lesssim \eps_1^2(1+t)^{-1+p_0}\sum_{k_1,k_2\in\mathbb{Z},\,k_1\geq -10}\one_{15}(k,k_1,k_2)2^{(k_2-k_1)/2}2^{-Nk_2}\\
&\lesssim \eps_1^2(1+t)^{-1+p_0}2^{-k},
\end{split}
\end{equation*}
and
\begin{equation*}
\begin{split}
\|P_k&A(u(t),\mathcal{N}(t))\|_{L^2}\\
&\lesssim \sum_{k_1,k_2\in\mathbb{Z}}\|a^{k,k_1,k_2}\|_{S^\infty}\min\big[\|P'_{k_1}u(t)\|_{L^2}\|P'_{k_2}\mathcal{N}(t)\|_{L^\infty},\|P'_{k_1}u(t)\|_{L^\infty}\|P'_{k_2}\mathcal{N}(t)\|_{L^2}\big]\\
&\lesssim \eps_1^2(1+t)^{-1+p_0}\sum_{k_1,k_2\in\mathbb{Z},\,2^{k_1}\geq (1+t)^{-2}}\one_{15}(k,k_1,k_2)2^{(k_2-k_1)/2}2^{-3k_2}2^{k_1(1/2-p_0)}\\
&+\eps_1^2(1+t)^{-1/2+2p_0}\sum_{k_1,k_2\in\mathbb{Z},\,2^{k_1}\leq (1+t)^{-2}}\one_{15}(k,k_1,k_2)2^{(k_2-k_1)/2}2^{-(N-1)k_2}2^{k_1/2}2^{k_1(1/2-p_0)}\\
&\lesssim \eps_1^2(1+t)^{-1+4p_0}2^{-k}.
\end{split}
\end{equation*}
The bounds for $\|P_kB(\overline{\mathcal{N}}(t),u(t))\|_{L^2}$ and  $\|P_kB(\overline{u}(t),\mathcal{N}(t))\|_{L^2}$ are similar, since the bounds on the symbols $a$ and $b$ are similar, and the desired bound \eqref{nf27} follows.
\end{proof}

The following proposition on the uniform control of the $Z$-norm of the function $f$ represents the main step in the proof of
Proposition \ref{Sup1}. For any $h\in L^2(\mathbb{R})$ let
\begin{equation}\label{nf29}
 \|h\|_Z:=\big\|\big(|\xi|^{p_1}+|\xi|^{5}\big)\widehat{h}(\xi)\big\|_{L^\infty_\xi},
\end{equation}
where $p_1>0$ is the same small exponent as in \eqref{pr7}.

\begin{pro}\label{Znorm}
Assume that $T'\in[0,T]$ and
\begin{equation}\label{nf30}
\sup_{t\in[0,T']}\|f(t)\|_Z\leq \eps_1.
\end{equation}
Then
\begin{equation}\label{nf31}
\sup_{t\in[0,T']}\|f(t)\|_Z\lesssim \eps_0.
\end{equation}
\end{pro}

We show now how to use Proposition \ref{Znorm} and the estimates in Lemma \ref{dispersive} and Lemma \ref{interpolation} to complete the proof of the main
Proposition \ref{Sup1}.

\begin{proof}[Proof of Proposition \ref{Sup1}] Let
\begin{equation*}
z(t):=\|f(t)\|_Z,
\end{equation*}
and notice that $z:[0,T]\to\mathbb{R}_+$ is a continuous function.

We show first that
\begin{equation}\label{nf32}
 z(0)\lesssim \eps_0.
\end{equation}
Indeed, using the definitions and Lemma \ref{interpolation}
\begin{equation*}
\begin{split}
 z(0)&\lesssim \sup_{k\in\mathbb{Z}}\,(2^{p_1k}+2^{5k})\|\widehat{P_kv(0)}\|_{L^\infty}\\
&\lesssim \sup_{k\in\mathbb{Z}}\,(2^{p_1k}+2^{5k})2^{-k/2}\|\widehat{P'_kv(0)}\|_{L^2}^{1/2}\big[2^k\|\partial \widehat{P'_kv(0)}\|_{L^2}+\|\widehat{P'_kv(0)}\|_{L^2}\big]^{1/2}.
\end{split}
\end{equation*}
In addition, using \eqref{pr7}, \eqref{nf11}--\eqref{nf11.2} with $t=0$,
\begin{equation*}
\begin{split}
&\|P'_kv(0)\|_{L^2}\lesssim \eps_0\min\big[2^{-(N-1)k},2^{k(1/2-p_1)}\big],\\
&2^k\|\partial \widehat{P'_kv(0)}\|_{L^2}+\|\widehat{P'_kv(0)}\|_{L^2}\lesssim \eps_0\min\big[2^{-k},2^{k(1/2-p_1)}\big],
\end{split}
\end{equation*}
for any $k\in\mathbb{Z}$. The desired bound \eqref{nf32} follows by combining these inequalities.

We apply now Proposition \ref{Znorm}. By continuity, $z(t)\lesssim  \eps_0$ for any $t\in[0,T]$, provided that $\eps_0$ is
sufficiently small and  $\eps_0\ll\eps_1\leq \eps_0^{2/3}\ll 1$ (see \eqref{pr8}). Therefore, for any $k\in\mathbb{Z}$ and $t\in[0,T]$,
\begin{equation}\label{nf33}
(2^{p_1k}+2^{5k})\|\widehat{P_kf}(t)\|_{L^\infty}\lesssim \eps_0.
\end{equation}

Since $v(t)=e^{-it\Lambda}f(t)$,
it follows from Lemma \ref{dispersive} that
\begin{equation}\label{cle1}
 \|P_kv(t)\|_{L^\infty}\lesssim |t|^{-1/2}2^{k/4}\|\widehat{P'_kf(t)}\|_{L^\infty}+|t|^{-3/5}2^{-2k/5}
\big[2^k\|\partial \widehat{P'_kf(t)}\|_{L^2}+\|\widehat{P'_kf(t)}\|_{L^2}\big]
\end{equation}
and
\begin{equation}\label{cle2}
\|P_kv(t)\|_{L^\infty}\lesssim |t|^{-1/2}2^{k/4}\|P'_kf(t)\|_{L^1},
\end{equation}
for any $k\in\mathbb{Z}$ and $t\in[0,T]$.

It follows from Lemma \ref{interpolation} and \eqref{nf24}--\eqref{nf25} that
\begin{equation*}
\|P'_kf(t)\|_{L^1}\lesssim \eps_0(1+t)^{6p_0}\min\big(2^{-(N+1)k/2},2^{-kp_0}\big).
\end{equation*}
Therefore, using \eqref{cle2},
\begin{equation}\label{cle3}
\|P_kv(t)\|_{L^\infty}\lesssim \eps_0|t|^{-1/2}2^{k/4}(1+t)^{6p_0}\min\big(2^{-(N+1)k/2},2^{-kp_0}\big).
\end{equation}

Recall that we are looking to prove the main decay bound \eqref{BigOne2},
\begin{equation}\label{cle4}
(1+t)^{1/2}2^{4\max(k,0)}\|P_ku(t)\|_{L^\infty}\lesssim\eps_0,
\end{equation}
for any $k\in\mathbb{Z}$ and $t\in[1,T]$. Clearly, this follows from \eqref{cle3} and \eqref{nf11.6} if $2^k\geq (1+t)^{10p_0}$.

On the other hand, if $2^k\leq (1+t)^{10p_0}$ then we use the bound \eqref{cle1}. Using \eqref{nf24}--\eqref{nf25},
\begin{equation*}
2^k\|\partial \widehat{P'_kf(t)}\|_{L^2}+\|\widehat{P'_kf(t)}\|_{L^2}\lesssim\eps_0(1+t)^{6p_0}2^{k(1/2-p_0)}.
\end{equation*}
Therefore, using also \eqref{nf33} and \eqref{cle1},
\begin{equation*}
\|P_kv(t)\|_{L^\infty}\lesssim |t|^{-1/2}2^{k/4}\frac{\eps_0}{2^{p_1k}+2^{5k}}+|t|^{-3/5}2^{-2k/5}\eps_0(1+t)^{6p_0}2^{k(1/2-p_0)}.
\end{equation*}
The desired bound \eqref{cle4} follows when $2^k\leq (1+t)^{10p_0}$.
\end{proof}

\subsection{Proof of Proposition \ref{Znorm}}\label{reduc1}

In this subsection we show that Proposition \ref{Znorm} is a consequence of the more technical Lemma \ref{Zcontrol} below. We start
from the equation \eqref{nf22}, which we write in the form
\begin{equation}\label{nf40}
 \partial_tf=e^{it\Lambda}\big(\mathcal{N}''+\mathcal{R}_{\geq 4}\big),
\end{equation}
where
\begin{equation}\label{nf41}
\mathcal{N}'':=A(Q_0(v+\overline{v},v),v)+A(v,Q_0(v+\overline{v},v))+B(\overline{Q_0(v+\overline{v},v)},v)
+B(\overline{v},Q_0(v+\overline{v},v)),
\end{equation}
and
\begin{equation}\label{nf42}
\begin{split}
\mathcal{R}_{\geq 4}:=&[A(Q_0(u+\overline{u},u),u)-A(Q_0(v+\overline{v},v),v)]\\
+&[A(u,Q_0(u+\overline{u},u))-A(v,Q_0(v+\overline{v},v))]\\
+&[B(\overline{Q_0(u+\overline{u},u)},u)-B(\overline{Q_0(v+\overline{v},v)},v)]\\
+&[B(\overline{u},Q_0(u+\overline{u},u))-B(\overline{v},Q_0(v+\overline{v},v))].
\end{split}
\end{equation}
The point of this decomposition is to identify $\mathcal{N}''$ as the main ``cubic'' part of the nonlinearity, which can be
expressed only in terms of $v(t)=e^{-it\Lambda}f(t)$. The remainder $\mathcal{R}_{\geq 4}$ can be thought of as a ``quartic''
contribution, due to the quadratic nature of $u-v$, see Lemma \ref{Lem1}.

To analyze the equation \eqref{nf40} and identify the asymptotic phase logarithmic correction, we need to further decompose the
nonlinearity $\mathcal{N}''$. Notice that $\overline{Q_0(g_1,g_2)}=Q_0(\overline{g_1},\overline{g_2})$. We write
\begin{equation}\label{nf43}
\begin{split}
&\mathcal{N}''=\mathcal{C}^{++-}+\mathcal{C}^{+++}+\mathcal{C}^{-++},\\
&\mathcal{C}^{++-}=\mathcal{C}^{++-}(v,v,\overline{v}):=A(Q_0(\overline{v},v),v)+A(v,Q_0(\overline{v},v))+B(Q_0(v,\overline{v}),v)
+B(\overline{v},Q_0(v,v)),\\
&\mathcal{C}^{+++}=\mathcal{C}^{+++}(v,v,v):=A(Q_0(v,v),v)+A(v,Q_0(v,v)),\\
&\mathcal{C}^{--+}=\mathcal{C}^{--+}(\overline{v},\overline{v},v):=B(Q_0(\overline{v},\overline{v}),v)+B(\overline{v},Q_0(\overline{v},v)).
\end{split}
\end{equation}
Letting $v^+=v,v^-=\overline{v}$, we expand
\begin{equation}\label{nf44}
 \widehat{\mathcal{C}^{\iota_1\iota_2\iota_3}}(\xi)=\frac{i}{4\pi^2}
\int_{\mathbb{R}\times\mathbb{R}}c^{\iota_1\iota_2\iota_3}(\xi,\eta,\sigma)\widehat{v^{\iota_1}}(\xi-\eta)
\widehat{v^{\iota_2}}(\eta-\sigma)\widehat{v^{\iota_3}}(\sigma)\,d\eta d\sigma
\end{equation}
for $(\iota_1,\iota_2,\iota_3)\in\{(++-),(+++),(--+)\}$, where
\begin{equation}\label{nf45}
\begin{split}
ic^{++-}(\xi,\eta,\sigma):&=a(\xi,\xi-\eta)q_0(\eta,\eta-\sigma)+a(\xi,\eta)q_0(\eta,\eta-\sigma)\\
&+b(\xi,\xi-\eta)q_0(\eta,\sigma)+b(\xi,\xi-\sigma)q_0(\xi-\sigma,\eta-\sigma),\\
ic^{+++}(\xi,\eta,\sigma):&=a(\xi,\xi-\eta)q_0(\eta,\sigma)+a(\xi,\eta)q_0(\eta,\sigma),\\
ic^{--+}(\xi,\eta,\sigma):&=b(\xi,\sigma)q_0(\xi-\sigma,\eta-\sigma)+b(\xi,\eta)q_0(\eta,\sigma).
\end{split}
\end{equation}

In view of the definitions \eqref{pr2} and \eqref{pr4}, the symbols $c^{\iota_1\iota_2\iota_3}$ are real-valued. Recall the formulas $\widehat{v^+}(\xi,t)=\widehat{f}(\xi,t)e^{-it|\xi|^{3/2}}$ and $\widehat{v^-}(\xi,t)=\widehat{\overline{f}}(\xi,t)e^{it|\xi|^{3/2}}$.
We can rewrite
\begin{equation}\label{nf46}
\mathcal{F}\big(e^{it\Lambda}\mathcal{N}''(t)\big)(\xi)=\frac{i}{4\pi^2}\big[I^{++-}(\xi,t)+I^{+++}(\xi,t)+I^{--+}(\xi,t)\big],
\end{equation}
where
\begin{equation}\label{nf47}
\begin{split}
I^{\iota_1\iota_2\iota_3}(\xi,t):=\int_{\mathbb{R}\times\mathbb{R}}&e^{it(|\xi|^{3/2}-\iota_1|\xi-\eta|^{3/2}-\iota_2|\eta-\sigma|^{3/2}-\iota_3|\sigma|^{3/2})}\\
&\times c^{\iota_1\iota_2\iota_3}(\xi,\eta,\sigma)\widehat{f^{\iota_1}}(\xi-\eta)
\widehat{f^{\iota_2}}(\eta-\sigma)\widehat{f^{\iota_3}}(\sigma)\,d\eta d\sigma
\end{split}
\end{equation}
for $(\iota_1,\iota_2,\iota_3)\in\{(++-),(+++),(--+)\}$. The formula \eqref{nf40} becomes
\begin{equation}\label{nf48}
(\partial_t\widehat{f})(\xi,t)=\frac{i}{4\pi^2}\big[I^{++-}(\xi,t)+I^{+++}(\xi,t)+I^{--+}(\xi,t)\big]+e^{it|\xi|^{3/2}}\widehat{\mathcal{R}_{\geq 4}}(\xi,t).
\end{equation}

In analyzing the formula \eqref{nf48}, the main contribution comes from the stationary points of the phase functions $(t,\eta,\sigma)\to t\Phi^{\iota_1\iota_2\iota_3}(\xi,\eta,\sigma)$, where
\begin{equation}\label{nf49}
\Phi^{\iota_1\iota_2\iota_3}(\xi,\eta,\sigma):=|\xi|^{3/2}-\iota_1|\xi-\eta|^{3/2}-\iota_2|\eta-\sigma|^{3/2}-\iota_3|\sigma|^{3/2}.
\end{equation}
More precisely, one needs to understand the contribution of the {\it{spacetime resonances}}, i.e. the points where $\Phi^{\iota_1\iota_2\iota_3}(\xi,\eta,\sigma)=(\partial_\eta\Phi^{\iota_1\iota_2\iota_3})(\xi,\eta,\sigma)=(\partial_\sigma\Phi^{\iota_1\iota_2\iota_3})(\xi,\eta,\sigma)=0$. In our case, it turns out that the only spacetime resonances correspond to $(\iota_1\iota_2\iota_3)=(++-)$ and the points $(\xi,\eta,\sigma)=(\xi,0,-\xi)$. Moreover, the contribution of these points is not absolutely integrable, and we have to identify and eliminate this contribution using a suitable logarithmic phase correction.
More precisely, let
\begin{equation}\label{nf50}
\begin{split}
&\widetilde{c}(\xi):=\frac{-8\pi|\xi|^{1/2}}{3}c^{++-}(\xi,0,-\xi)=\frac{8\pi|\xi|^{1/2}}{3}ib(\xi,2\xi)q_0(2\xi,\xi),\\
&L(\xi,t):=\frac{\widetilde{c}(\xi)}{4\pi^2}\int_0^t|\widehat{f}(\xi,s)|^2\frac{1}{s+1}\,ds,\\
&g(\xi,t):=e^{iL(\xi,t)}\widehat{f}(\xi,t).
\end{split}
\end{equation}
The formula \eqref{nf48} becomes
\begin{equation}\label{nf51}
\begin{split}
(\partial_tg)(\xi,t)&=\frac{i}{4\pi^2}e^{iL(\xi,t)}\Big[I^{++-}(\xi,t)+\widetilde{c}(\xi)\frac{|\widehat{f}(\xi,t)|^2}{t+1}\widehat{f}(\xi,t)\Big]\\
&+\frac{i}{4\pi^2}e^{iL(\xi,t)}\big[I^{+++}(\xi,t)+I^{--+}(\xi,t)\big]+e^{it|\xi|^{3/2}}e^{iL(\xi,t)}\widehat{\mathcal{R}_{\geq 4}}(\xi,t).
\end{split}
\end{equation}
Notice that the phase $L$ is real-valued. Therefore, to complete the proof of Proposition \ref{Znorm}, it suffices to prove the following lemma:

\begin{lem}\label{Zcontrol}
Recall that, for any $t\in[0,T']$ and $k\in\mathbb{Z}$,
\begin{equation}\label{nf55}
\begin{split}
\|\widehat{P_kf}(t)\|_{L^\infty}&\lesssim\eps_1\min\big(2^{-p_1k},2^{-5k}\big),\\
\|\widehat{P_kf}(t)\|_{L^2}&\lesssim\eps_0(1+t)^{6p_0}\min\big(2^{-(N-1)k},2^{k(1/2-p_0)}\big),\\
\|\widehat{P_kf}(t)\|_{L^2}+2^k\|\partial\widehat{P_kf}(t)\|_{L^2}&\lesssim\eps_0(1+t)^{6p_0}\min\big(2^{-k},2^{k(1/2-p_0)}\big),
\end{split}
\end{equation}
see \eqref{nf24}--\eqref{nf25} and \eqref{nf30}. Then for any $m\in\{1,2,\ldots\}$ and any $t_1\leq t_2\in[2^{m}-2,2^{m+1}]\cap[0,T']$ we have
\begin{equation}\label{nf56}
\big\|(|\xi|^{p_1}+|\xi|^{5})[g(\xi,t_2)-g(\xi,t_1)]\big\|_{L^\infty_\xi}\lesssim \eps_02^{-p_0m}.
\end{equation}
\end{lem}

\section{Proof of Lemma \ref{Zcontrol}}\label{TechProof}

In this section we provide the proof of Lemma \ref{Zcontrol}, which is the analogue of Lemma 6.1 in \cite{IoPu2}.
Notice first that the desired conclusion is a simple consequence of Lemma \ref{interpolation} and the bounds \eqref{nf55} if $|\xi|\geq (1+t)^{20p_0}$. Indeed, in this case for any $t\in[2^{m}-2,2^{m+1}]\cap[0,T']$ and any $|\xi|\approx 2^k\geq (1+t)^{20p_0}$
\begin{equation*}
\begin{split}
|\xi|^{5}|g(\xi,t)|&\lesssim 2^{5k}\|\widehat{P_kf}(t)\|_{L^\infty}\lesssim 2^{5k}\big[2^{-k}\|\widehat{P_kf}\|_{L^2}\big(2^k\|\partial \widehat{P_kf}\|_{L^2}+\|\widehat{P_kf}\|_{L^2}\big)\big]^{1/2}\\
&\lesssim \eps_0(1+t)^{6p_0}2^{-k/2}\lesssim\eps_0,
\end{split}
\end{equation*}
as desired. Notice also that the bound \eqref{nf56} is trivial if $|\xi|\leq 2^{-30}$.

It remains to prove \eqref{nf56} in the intermediate range $|\xi|\in[2^{-30},(1+t)^{20p_0}]$.
For $k\in\mathbb{Z}$ let $f_k^+:=P_kf$ and $f_k^-:=P_k\overline{f}$. For any $k_1,k_2,k_3\in\mathbb{Z}$ let
\begin{equation}\label{rak0}
\begin{split}
I^{\iota_1\iota_2\iota_3}_{k_1,k_2,k_3}(\xi,t):=\int_{\mathbb{R}\times\mathbb{R}}&e^{it(|\xi|^{3/2}-\iota_1|\xi-\eta|^{3/2}-\iota_2|\eta-\sigma|^{3/2}-\iota_3|\sigma|^{3/2})}\\
&\times c^{\iota_1\iota_2\iota_3}(\xi,\eta,\sigma)\widehat{f^{\iota_1}_{k_1}}(\xi-\eta)
\widehat{f^{\iota_2}_{k_2}}(\eta-\sigma)\widehat{f^{\iota_3}_{k_3}}(\sigma)\,d\eta d\sigma.
\end{split}
\end{equation}
It follows from \eqref{nf55} and Lemma \ref{dispersive} that for any $t\in[0,T]$ and $l\leq 0$
\begin{equation}\label{rak1}
\begin{split}
\|\widehat{f_l^{\pm}(t)}\|_{L^2}+2^l\|\partial\widehat{f_l^{\pm}(t)}\|_{L^2}&\lesssim \eps_12^{-p_0l}2^{l/2}(1+t)^{6p_0},\\
\|e^{\mp it\Lambda}f_l^{\pm}(t)\|_{L^\infty}&\lesssim \eps_12^{-p_0l}2^{l/10}(1+t)^{-1/2},\\
\|\widehat{f_l^{\pm}(t)}\|_{L^\infty}&\lesssim \eps_12^{-p_0l}.
\end{split}
\end{equation}
Moreover, if $l\geq 0$ then
\begin{equation}\label{rak1.5}
\begin{split}
\|\widehat{f_l^{\pm}(t)}\|_{L^2}&\lesssim \eps_12^{-9l}(1+t)^{6p_0},\\
\|\widehat{f_l^{\pm}(t)}\|_{L^2}+2^l\|\partial\widehat{f_l^{\pm}(t)}\|_{L^2}&\lesssim \eps_12^{-l}(1+t)^{6p_0},\\
\|e^{\mp it\Lambda}f_l^{\pm}(t)\|_{L^\infty}&\lesssim \eps_12^{-7l/5}(1+t)^{-1/2},\\
\|\widehat{f_l^{\pm}(t)}\|_{L^\infty}&\lesssim \eps_12^{-5l}.
\end{split}
\end{equation}

It is easy to see using the definitions and \eqref{nf12} that
\begin{equation}\label{rak1.6}
\begin{split}
\big|c^{\iota_1\iota_2\iota_3}&(\xi,\eta,\sigma)\cdot \varphi'_{k_1}(\xi-\eta)\varphi'_{k_2}(\eta-\sigma)\varphi'_{k_3}(\sigma)\big|\\
&\lesssim 2^{4\max(k_1,k_2,k_3)}2^{[\mathrm{med}(k_1,k_2,k_3)-\min(k_1,k_2,k_3)]/2}\mathbf{1}_{[2^{-25},\infty)}(|\xi|).
\end{split}
\end{equation}
Also $\widehat{\mathcal{R}_{\geq 4}}(\xi,t)=0$ if $|\xi|\leq 2^{-25}$. One can decompose the integrals $I^{\iota_1\iota_2\iota_3}$ into sums of the integrals $I_{k_1,k_2,k_3}^{\iota_1\iota_2\iota_3}$, and then estimate the terms corresponding to large frequencies and the terms corresponding to small frequencies (relative to $m$) using only the bounds \eqref{rak1}--\eqref{rak1.6}. As in \cite[Section 5]{IoPu1} and \cite[Section 6]{IoPu2}, we can reduce matters to proving the following:

\begin{lem}\label{Zcontrol2}
Assume that $k\in[-30,20p_0m]$, $|\xi|\in [2^k,2^{k+1}]$, $m\geq 1$, $t_1\leq t_2\in[2^{m}-2,2^{m+1}]\cap[0,T']$, and $k_1,k_2,k_3$ are integers satisfying
\begin{equation}\label{rak3}
\begin{split}
&k_1,k_2,k_3\in [-3m,3m/N-1000],\\
&\min(k_1,k_2,k_3)/2+3\mathrm{med}(k_1,k_2,k_3)/2\geq -m(1+10p_0).
\end{split}
\end{equation}
Then
\begin{equation}\label{rak4}
\Big|\int_{t_1}^{t_2}e^{iL(\xi,s)}\Big[I^{++-}_{k_1,k_2,k_3}(\xi,s)+\widetilde{c}(\xi)\frac{\widehat{f^+_{k_1}}(\xi,s)\widehat{f^+_{k_2}}(\xi,s)\widehat{f^-_{k_3}}(-\xi,s)}{s+1}\Big]\,ds\Big|\lesssim \eps_1^32^{-200p_0m},
\end{equation}
and, for $(\iota_1,\iota_2,\iota_3)\in\{(+++),(--+)\}$,
\begin{equation}\label{rak5}
\Big|\int_{t_1}^{t_2}e^{iL(\xi,s)}I^{\iota_1\iota_2\iota_3}_{k_1,k_2,k_3}(\xi,s)\,ds\Big|\lesssim \eps_1^32^{-200p_0m}.
\end{equation}
Moreover
\begin{equation}\label{rak6}
\Big|\int_{t_1}^{t_2}e^{iL(\xi,s)}e^{is|\xi|^{3/2}}\widehat{\mathcal{R}_{\geq 4}}(\xi,s)\,ds\Big|\lesssim \eps_1^32^{-200p_0m}.
\end{equation}
\end{lem}

The rest of this section is concerned with the proof of this lemma. We will often use the alternative formulas
\begin{equation}\label{rak7}
I^{\iota_1\iota_2\iota_3}_{k_1,k_2,k_3}(\xi,t)=\int_{\mathbb{R}\times\mathbb{R}}e^{it\Phi^{\iota_1\iota_2\iota_3}(\xi,\eta,\sigma)}c^{\ast,\iota_1\iota_2\iota_3}_{\xi;k_1,k_2,k_3}(\eta,\sigma)\widehat{f^{\iota_1}_{k_1}}(\xi+\eta)
\widehat{f^{\iota_2}_{k_2}}(\xi+\sigma)\widehat{f^{\iota_3}_{k_3}}(-\xi-\eta-\sigma)\,d\eta d\sigma,
\end{equation}
where
\begin{equation}\label{rak8}
\begin{split}
&\Phi^{\iota_1\iota_2\iota_3}(\xi,x,y):=\Lambda(\xi)-\iota_1\Lambda(\xi+x)-\iota_2\Lambda(\xi+y)-\iota_3\Lambda(\xi+x+y),\\
&c^{\ast,\iota_1\iota_2\iota_3}_{\xi;k_1,k_2,k_3}(x,y):=c^{\iota_1\iota_2\iota_3}(\xi,-x,-\xi-x-y)\cdot \varphi'_{k_1}(\xi+x)
\varphi'_{k_2}(\xi+y)\varphi'_{k_3}(\xi+x+y).
\end{split}
\end{equation}
These formulas follow easily from \eqref{rak0} by changes of variables.

Using the explicit formulas \eqref{nf45} it is easy to verify that the symbols $c^{\ast,\iota_1\iota_2\iota_3}_{\xi;k_1,k_2,k_3}$ satisfy the $S^\infty$ estimates
\begin{equation}\label{rak9}
\big\|\mathcal{F}^{-1}\big(c^{\ast,\iota_1\iota_2\iota_3}_{\xi;l_1,l_2,l_3}\big)\big\|_{L^1(\mathbb{R}^2)}\lesssim 2^{4l_{\max}}2^{(l_{\mathrm{med}}-l_{\min})/2},
\end{equation}
and
\begin{equation}\label{rak9.5}
\begin{split}
&\big\|\mathcal{F}^{-1}\big[(\partial_\eta c^{\ast,\iota_1\iota_2\iota_3}_{\xi;l_1,l_2,l_3})(\eta,\sigma)\big]\big\|_{L^1(\mathbb{R}^2)}\lesssim 2^{4l_{\max}}2^{(l_{\mathrm{med}}-l_{\min})/2}2^{-\min(l_1,l_3)},\\
&\big\|\mathcal{F}^{-1}\big[(\partial_\sigma c^{\ast,\iota_1\iota_2\iota_3}_{\xi;l_1,l_2,l_3})(\eta,\sigma)\big]\big\|_{L^1(\mathbb{R}^2)}\lesssim 2^{4l_{\max}}2^{(l_{\mathrm{med}}-l_{\min})/2}2^{-\min(l_2,l_3)},
\end{split}
\end{equation}
for any $\xi\in\mathbb{R}$ with $|\xi|\gtrsim 1$ and any $l_1,l_2,l_3\in\mathbb{Z}$, where $l_{\max}=\max(l_1,l_2,l_3)$, $l_{\mathrm{med}}=\mathrm{med}(l_1,l_2,l_3)$, $l_{\min}=\min(l_1,l_2,l_3)$.

\subsection{Proof of \eqref{rak4}}\label{prooftech1}

We divide the proof of the bound \eqref{rak4} into several lemmas.
For simplicity of notation, in this subsection let $\Phi:=\Phi^{++-}$ and $c^{\ast}_{\mathbf{k}}:=c^{\ast,++-}_{\xi;k_1,k_2,k_3}$.

\begin{lem}\label{lemma1}
The bound \eqref{rak4} holds provided that \eqref{rak3} holds and, in addition,
\begin{equation}\label{rak10}
\max\big(|k-k_1|,|k-k_2|,|k-k_3|\big)\leq 20.
\end{equation}
\end{lem}

\begin{proof}[Proof of Lemma \ref{lemma1}] This is the main case, which gives the precise form of the correction. However, the proof is similar (in fact easier, due the extra assumption $|\xi|\gtrsim 1$ in our situation) to the proof of \cite[Lemma 6.4]{IoPu2}.
The only difference is that, in our case $\Lambda(\xi)=|\xi|^{3/2}$; therefore one has the expansion
\begin{equation*}
\Big|\Lambda(\xi)-\Lambda(\xi+\eta)-\Lambda(\xi+\sigma)+\Lambda(\xi+\eta+\sigma)-\frac{3\eta\sigma}{4|\xi|^{1/2}}\Big|\lesssim 2^{-3k/2}(|\eta|^3+|\sigma|^3).
\end{equation*}
The same argument as in \cite[Lemma 6.4]{IoPu2} then leads to desired bound.
\end{proof}

\begin{lem}\label{lemmainbound20}
The bound \eqref{rak4} holds provided that \eqref{rak3} holds and, in addition,
\begin{align}
\label{hyp20}
\begin{split}
& \max (|k-k_1|, |k- k_2|, |k-k_3|) \geq 21,
\\
& \max (|k_1 - k_3| , |k_2-k_3|) \geq 5 \quad \mbox{and} \quad \min (k_1,k_2,k_3) \geq -\frac{49}{100}m .
\end{split}
\end{align}
\end{lem}

\begin{proof}[Proof of Lemma \ref{lemmainbound20}] In this case we will show that
\begin{equation}\label{hyp20.1}
|I^{++-}_{k_1,k_2,k_3}(\xi,s)|\lesssim \eps_1^32^{-m}2^{-200p_0m}.
\end{equation}
Without loss of generality, by symmetry we can assume that $|k_1-k_3| \geq 5$ and $k_2\leq \max(k_1,k_3)+5$.
Under the assumptions \eqref{hyp20} we have
\begin{align*}
 |(\partial_\eta\Phi)(\xi,\eta,\sigma)| = |-\Lambda'(\xi+\eta)+\Lambda'(\xi+\eta+\sigma)| \gtrsim 2^{k_{\max}/2}.
\end{align*}
Therefore we can integrate by parts in $\eta$ in the expression \eqref{rak7} for $I_{k_1,k_2,k_3}^{++-}$.
This gives
\begin{align*}
|I_{k_1,k_2,k_3}^{++-}(\xi,s)| & \lesssim  |K_1(\xi,s)| + |K_2(\xi,s)| + |K_3(\xi,s)| + |K_4(\xi,s)|,
\end{align*}
where
\begin{align}
\label{IBPeta}
\begin{split}
K_1(\xi) & := \int_{\mathbb{R}\times\mathbb{R}} e^{is\Phi(\xi,\eta,\sigma)}
   m_1(\eta,\sigma) c^\ast_{\mathbf{k}}(\eta,\sigma)
  (\partial\what{f_{k_1}^{+}})(\xi+\eta) \what{f_{k_2}^{+}}(\xi+\sigma) \what{f_{k_3}^{-}}(-\xi-\eta-\sigma) \,d\eta d\sigma,
\\
K_2(\xi) & := \int_{\mathbb{R}\times\mathbb{R}} e^{is\Phi(\xi,\eta,\sigma)}
   m_1(\eta,\sigma) c^\ast_{\mathbf{k}}(\eta,\sigma)
  \what{f_{k_1}^{+}}(\xi+\eta) \what{f_{k_2}^{+}}(\xi+\sigma) (\partial\what{f_{k_3}^{-}})(-\xi-\eta-\sigma) \,d\eta d\sigma,
\\
K_3(\xi) & := \int_{\mathbb{R}\times\mathbb{R}} e^{is\Phi(\xi,\eta,\sigma)}
  (\partial_\eta  m_1)(\eta,\sigma) c^\ast_{\mathbf{k}}(\eta,\sigma)
  \what{f_{k_1}^{+}}(\xi+\eta) \what{f_{k_2}^{+}}(\xi+\sigma) \what{f_{k_3}^{-}}(-\xi-\eta-\sigma) \,d\eta d\sigma,
\\
K_4(\xi) & := \int_{\mathbb{R}\times\mathbb{R}} e^{is\Phi(\xi,\eta,\sigma)}
   m_1(\eta,\sigma) (\partial_\eta c^\ast_{\mathbf{k}})(\eta,\sigma)
  \what{f_{k_1}^{+}}(\xi+\eta) \what{f_{k_2}^{+}}(\xi+\sigma) \what{f_{k_3}^{-}}(-\xi-\eta-\sigma) \,d\eta d\sigma,
\end{split}
\end{align}
having denoted
\begin{align}
\label{symIBPeta}
 m_1(\eta,\sigma) := \frac{1}{s \partial_\eta \Phi(\xi,\eta,\sigma)} \varphi'_{k_1}(\xi+\eta)\varphi'_{k_3}(\xi+\eta+\sigma).
\end{align}

Observe that, in our situation,
\begin{align}
\label{symIBPetaest1}
{\| m_1(\eta,\sigma) \|}_{S^\infty} \lesssim 2^{-m} 2^{-k_{\max}/2}.
\end{align}
We can then estimate $K_1$ using Lemma \ref{touse} (ii), the estimate on $c^\ast_{\mathbf{k}}$ in \eqref{rak9},
the bounds \eqref{rak1}--\eqref{rak1.5}, and the last constraint in \eqref{hyp20}. For example, if $k_1\leq k_3$ (so $2^{k_3}\approx 2^{k_{\max}}$) then we estimate
\begin{align*}
|K_1(\xi,s)| &\lesssim {\| m_1(\eta,\sigma) \|}_{S^\infty}  {\| c^\ast_{\mathbf{k}}(\eta,\sigma) \|}_{S^\infty}
  {\| \partial \what{f_{k_1}^{+}}(s) \|}_{L^2} {\| e^{-is\Lambda} f_{k_2}^{+}(s) \|}_{L^\infty} {\| \what{f_{k_3}^{-}}(s) \|}_{L^2}
\\
& \lesssim 2^{-m} 2^{-k_{\max}/2} \cdot 2^{5k_{\max}} 2^{-k_{\min}/2}\cdot \eps_1 2^{-k_1(1/2+p_0)} 2^{6mp_0} \cdot \eps_1 2^{-m/2} \cdot \eps_1 2^{6mp_0}2^{-9k_3}
\\
& \lesssim \e_1^3 2^{-3m/2} 2^{12p_0m} 2^{-4k_{\max}}2^{-k_{\min}/2}2^{-k_1(1/2+p_0)}\\
&\lesssim \e_1^3 2^{15p_0m}2^{-3m/2}2^{-k_{\min}}.
\end{align*}
This suffices to prove \eqref{hyp20.1}. On the other hand, if $k_1\geq k_3$ (so $2^{k_1}\approx 2^{k_{\max}}$) then we estimate
\begin{align*}
|K_1(\xi,s)| &\lesssim {\| m_1(\eta,\sigma) \|}_{S^\infty}  {\| c^\ast_{\mathbf{k}}(\eta,\sigma) \|}_{S^\infty}
  {\| \partial \what{f_{k_1}^{+}}(s) \|}_{L^2} \\
  &\times\min\big[{\| e^{-is\Lambda} f_{k_2}^{+}(s) \|}_{L^\infty} {\| \what{f_{k_3}^{-}}(s) \|}_{L^2},{\| \widehat{f_{k_2}^{+}}(s) \|}_{L^2} {\| e^{is\Lambda}f_{k_3}^{-}(s) \|}_{L^\infty}\big]
\\
& \lesssim 2^{-m} 2^{-k_{\max}/2} \cdot 2^{4.5k_{\max}} 2^{-k_{\min}/2}\cdot \eps_1 2^{-2k_1} 2^{6mp_0} \cdot \eps_1^2 2^{-m/2} 2^{-6\max(k_2,k_3,0)}2^{6p_0m}
\\
& \lesssim \e_1^3 2^{-3m/2} 2^{14p_0m} 2^{-k_{\min}/2}\cdot 2^{2k_{\max}}2^{-6\max(k_2,k_3,0)}.
\end{align*}
This also suffices to prove \eqref{hyp20.1}, since $2^{2k_{\max}}2^{-6\max(k_2,k_3,0)}\lesssim 2^{50p_0m}$ as a consequence of the assumption $k\leq 20p_0m$. The estimates for $|K_2(s,\xi)|$, $|K_3(s,\xi)|$, and $|K_4(s,\xi)|$ are similar, which completes the proof of the lemma.
\end{proof}

\begin{lem}\label{lemmainbound22}
  The bound \eqref{rak4} holds provided that \eqref{rak3} holds and, in addition,
\begin{align}
\label{hyp42}
\begin{split}
& \max (|k-k_1|, |k- k_2|, |k-k_3|) \geq 21,
\\
& \max (|k_1 - k_3| , |k_2-k_3|) \geq 5 \quad \mbox{and} \quad \min (k_1,k_2) \leq -\frac{48m}{100}.
\end{split}
\end{align}
\end{lem}

\begin{proof}[Proof of Lemma \ref{lemmainbound22}]
By symmetry we may assume that $k_2=\min(k_1,k_2)$. The main observation is that we still have the strong lower bound
\begin{align*}
 |(\partial_\eta \Phi)(\xi,\eta,\sigma)| = |-\Lambda'(\xi+\eta)+\Lambda'(\xi+\eta+\sigma)| \gtrsim 2^{-k_{\max}/2}.
\end{align*}
This is easy to see since $|\xi|\geq 2^{-31}$, $|\xi+\sigma|\leq 2^{k_2+1}$, and $|\xi+\eta|\approx 2^{k_1}$. Integrating by parts in $\eta$ and estimating the resulting integrals as in Lemma \ref{lemmainbound20} (placing $\widehat{f_{k_2}^+}$ in $L^2$ and recalling also the restriction $k_{\min}+3k_{\mathrm{med}}\geq -2m(1+10p_0)$, see \eqref{rak3}) gives the desired conclusion.
\end{proof}

\begin{lem}\label{lemmainbound23}
  The bound \eqref{rak4} holds provided that \eqref{rak3} holds and, in addition,
\begin{align}
\label{hyp43}
\begin{split}
& \max (|k-k_1|, |k- k_2|, |k-k_3|) \geq 21,
\\
& \max (|k_1 - k_3| , |k_2-k_3|) \geq 5 \quad \mbox{and} \quad \min (k_1,k_2) \geq -\frac{48m}{100}\quad \mbox{and} \quad k_3 \leq -\frac{49m}{100}.
\end{split}
\end{align}
\end{lem}

\begin{proof}[Proof of Lemma \ref{lemmainbound23}]
In this case we need to integrate by parts in time. Without loss of generality, we may again assume $k_2 = \min (k_1,k_2)$, therefore
\begin{align}
\label{conf10}
k_3 \leq -49m/100\leq -48m/100\leq k_2\leq k_1,\qquad k_1 \geq -30.
\end{align}

Recall that
\begin{align*}
\Phi(\xi,\eta,\sigma) = \Lambda(\xi)-\Lambda(\xi+\eta)-\Lambda(\xi+\sigma)+\Lambda(\xi+\eta+\sigma).
\end{align*}
For $|\xi+\eta| \approx 2^{k_1}$, $|\xi+\sigma| \approx |\eta| \approx 2^{k_2}$, $|\xi+\eta+\sigma| \approx 2^{k_3}$,
with $k_1,k_2,k_3$ satisfying \eqref{conf10}, one can use \eqref{en21.4} to show that
\begin{align*}
| \Phi(\xi,\eta,\sigma) | & \geq |\Lambda(\xi)-\Lambda(\xi+\eta)-\Lambda(\eta) | - 2^{10}|\xi+\eta+\sigma| 2^{k_2/2}\gtrsim 2^{k_2}.
\end{align*}

Because of the above lower bound we can integrate by parts in $s$ to obtain
\begin{align}
\label{IBPs}
\begin{split}
\Big| \int_{t_1}^{t_2} &e^{iL(\xi,s)} I_{k_1,k_2,k_3}^{++-}(\xi,s) \, ds \Big|\\
&\lesssim  | N_1 (\xi,t_1) | + | N_1 (\xi,t_2) | + \int_{t_1}^{t_2} | N_2(\xi,s)|+ | N_3(\xi,s) | +  |(\partial_sL)(\xi,s)|| N_1(\xi,s) | \, ds,
\end{split}
\end{align}
where
\begin{align}
\label{IBPs2}
\begin{split}
& N_1(\xi) := \int_{\mathbb{R}\times\mathbb{R}} e^{is\Phi(\xi,\eta,\sigma)} \frac{c^\ast_{\mathbf{k}} (\eta,\sigma)}{i \Phi(\xi,\eta,\sigma)}
  \what{f_{k_1}^{+}}(\xi+\eta) \what{f_{k_2}^{+}}(\xi+\sigma) \what{f_{k_3}^{-}}(-\xi-\eta-\sigma) \,d\eta d\sigma ,
\\
& N_2(\xi) := \int_{\mathbb{R}\times\mathbb{R}} e^{is\Phi(\xi,\eta,\sigma)} \frac{c^\ast_{\mathbf{k}}(\eta,\sigma)}{i \Phi(\xi,\eta,\sigma)}
 \, (\partial_s\what{f_{k_1}^{+}})(\xi+\eta) \what{f_{k_2}^{+}}(\xi+\sigma) \what{f_{k_3}^{-}}(-\xi-\eta-\sigma) \,d\eta d\sigma ,
\\
& N_3(\xi) := \int_{\mathbb{R}\times\mathbb{R}} e^{is\Phi(\xi,\eta,\sigma)} \frac{c^\ast_{\mathbf{k}}(\eta,\sigma)}{i \Phi(\xi,\eta,\sigma)}
  \what{f_{k_1}^{+}}(\xi+\eta)(\partial_s\what{f_{k_2}^{+}})(\xi+\sigma) \what{f_{k_3}^{-}}(-\xi-\eta-\sigma) \,d\eta d\sigma.
\end{split}
\end{align}
Notice that we do not get a term containing the time derivative of $\what{f_{k_3}^{+}}$ because $k_3 \ll 0$.

To estimate the first term in \eqref{IBPs} it suffices to use the pointwise bound
\begin{align}
\label{IBPssym}
\Big| \frac{c^\ast_{\mathbf{k}} (\eta,\sigma)}{ \Phi(\xi,\eta,\sigma)} \Big| \lesssim 2^{4k_1} 2^{-k_2/2} 2^{-k_3/2},
\end{align}
see \eqref{rak9}. Using also \eqref{rak1}--\eqref{rak1.5} we obtain, for any $s\in[t_1,t_2]$,
\begin{align*}
|N_1(\xi,s)| & \lesssim 2^{4k_1} 2^{-k_2/2} 2^{-k_3/2}
    {\| \what{f_{k_1}^{+}}(s) \|}_{L^\infty} {\| \what{f_{k_2}^{+}}(s) \|}_{L^\infty} {\| \what{f_{k_3}^{-}}(s) \|}_{L^\infty} 2^{k_2} 2^{k_3}
\\
& \lesssim \e_1^3 2^{k_3(1/2-p_0)}\\
&\lesssim \e_1^3 2^{-m/10}.
\end{align*}
Moreover, the definition of $L$ in \eqref{nf50} and the apriori assumptions \eqref{rak1}--\eqref{rak1.5} show that
\begin{align*}
\big|(\partial_sL)(\xi,s) \big| \lesssim \e_1^2 2^{-m}.
\end{align*}
Therefore
\begin{align}
\label{N1bound}
| N_1 (\xi,t_1) | + | N_1 (\xi,t_2) | + \int_{t_1}^{t_2} |(\partial_sL)(\xi,s)|| N_1(\xi,s) | \, ds\lesssim \e_1^3 2^{-m/10}.
\end{align}

Using the same pointwise bound \eqref{IBPssym} on the symbol, and the $L^2$ bound
\begin{equation}\label{prs}
\big\|(\partial_s\what{f_{l}^{\pm}})(s)\big\|_{L^2}\lesssim \eps_1^22^{-l}2^{-m+6p_0m}\mathbf{1}_{[-30,\infty)}(l),
\end{equation}
see \eqref{nf27} and \eqref{nf22}, the term $N_2$ can be estimated in the following way,
 \begin{align*}
\begin{split}
|N_2(\xi,s)| & \lesssim 2^{4k_1} 2^{-k_2/2} 2^{-k_3/2}
    {\big\|(\partial_s\what{f_{k_1}^{+}})(s) \big\|}_{L^2}
    {\| \what{f_{k_2}^{+}}(s) \|}_{L^2} {\| \what{f_{k_3}^{-}}(s) \|}_{L^\infty} 2^{k_3}
\\
& \lesssim \eps_1^3 2^{3k_1}2^{-m+6p_0m}\min(2^{-p_0k_2},2^{-8k_2}) 2^{k_3(1/2-p_0)} \\
&\lesssim \eps_1^3 2^{-m}2^{-m/10}.
\end{split}
\end{align*}
The integral $|N_3(\xi,s)|$ can be estimated in a similar way. Therefore
\begin{align*}
\int_{t_1}^{t_2} | N_2(\xi,s) |+|N_3(\xi,s)| \, ds\lesssim \e_1^3 2^{-m/10}.
\end{align*}
The lemma follows using also \eqref{IBPs} and \eqref{N1bound}.
\end{proof}

\begin{lem}\label{lemmainbound21}
  The bound \eqref{rak4} holds provided that \eqref{rak3} holds and, in addition,
\begin{equation}
\label{hyp40}
\max (|k-k_1|, |k- k_2|, |k-k_3|) \geq 21 \qquad \mbox{and} \qquad \max (|k_1 - k_3|, |k_2-k_3|) \leq 4.
\end{equation}
\end{lem}

\begin{proof}[Proof of Lemma \ref{lemmainbound21}]
 In this case we have $\min (k_1,k_2,k_3) \geq  k + 10 \geq -20$.
In particular, $|\sigma|\approx 2^{k_2}$ and we can integrate by parts in $\eta$ similarly to what was done before
in Lemma \ref{lemmainbound20}.
More precisely, under the assumptions \eqref{hyp40} we have
\begin{align*}
 |(\partial_\eta \Phi)(\xi,\eta,\sigma)| = |-\Lambda'(\xi+\eta)+\Lambda'(\xi+\eta+\sigma)| \gtrsim 2^{k_2/2} \gtrsim 2^{k_{\max}/2}.
\end{align*}
Integrating by parts in $\eta$ as in the previous lemma gives
\begin{align*}
|I_{k_1,k_2,k_3}^{++-}(\xi,s)| & \lesssim  |K_1(\xi,s)| + |K_2(\xi,s)| + |K_3(\xi,s)| + |K_4(\xi,s)|
\end{align*}
where the term $K_j$, $j=1,\dots 4$ are defined in \eqref{IBPeta}-\eqref{symIBPeta}. The same estimates as before show that $|I^{++-}_{k_1,k_2,k_3}(\xi,s)|\lesssim \eps_1^32^{-m}2^{-200p_0m}$, which suffices to prove the lemma.
\end{proof}

\subsection{Proof of \eqref{rak5}}\label{prooftech2}

As with the proof of \eqref{rak4}, we divide the proof of the bound \eqref{rak5} into several lemmas. We only consider in detail the case $(\iota_1\iota_2\iota_3)=(--+)$ since the case $(\iota_1\iota_2\iota_3)=(+++)$ is very similar. For simplicity of notation, in this subsection let $\Phi:=\Phi^{--+}$ and $c^{\ast}_{\mathbf{k}}:=c^{\ast,--+}_{\xi;k_1,k_2,k_3}$.

\begin{lem}\label{lemmainBO20}
The bound \eqref{rak5} holds provided that \eqref{rak3} holds and, in addition,
\begin{align}
\label{hyr20}
\max (|k_1 - k_3| , |k_2-k_3|) \geq 5 \quad \mbox{and} \quad \min (k_1,k_2,k_3) \geq -\frac{49}{100}m.
\end{align}
\end{lem}

\begin{proof}[Proof of Lemma \ref{lemmainBO20}] This is similar to the proof of Lemma \ref{lemmainbound20}. Without loss of generality, by symmetry we can assume that $|k_1-k_3| \geq 5$ and $k_2\leq \max(k_1,k_3)+5$. Under the assumptions \eqref{hyr20} we still have the strong lower bound
\begin{align*}
 |(\partial_\eta\Phi)(\xi,\eta,\sigma)| = |\Lambda'(\xi+\eta)-\Lambda'(\xi+\eta+\sigma)| \gtrsim 2^{k_{\max}/2},
\end{align*}
and the proof proceeds exactly as in Lemma \ref{lemmainbound20}, using integration by parts in $\eta$.
\end{proof}

\begin{lem}\label{lemmainBO21}
The bound \eqref{rak5} holds provided that \eqref{rak3} holds and, in addition,
\begin{align}
\label{hyr21}
\max (|k_1 - k_3| , |k_2-k_3|) \geq 5 \quad \mbox{and} \quad \mathrm{med} (k_1,k_2,k_3) \leq -48m/100.
\end{align}
\end{lem}

\begin{proof}[Proof of Lemma \ref{lemmainBO21}] This is similar to the situation in Lemma \ref{lemmainbound22}. By symmetry we may assume that $k_2=\min(k_1,k_2)$. The main observation is that we still have the strong lower bound
\begin{align*}
 |(\partial_\eta \Phi)(\xi,\eta,\sigma)| = |-\Lambda'(\xi+\eta)+\Lambda'(\xi+\eta+\sigma)| \gtrsim 2^{k_{\max}/2}.
\end{align*}
Then we integrate by parts in $\eta$ and estimate the resulting integrals as in Lemma \ref{lemmainbound20} (placing $\widehat{f_{k_2}^+}$ in $L^2$ and recalling also the restriction $k_{\min}+3k_{\mathrm{med}}\geq -2m(1+10p_0)$, see \eqref{rak3}).
\end{proof}

\begin{lem}\label{lemmainBO23}
  The bound \eqref{rak5} holds provided that \eqref{rak3} holds and, in addition,
\begin{align}
\label{hyr43}
\max (|k_1 - k_3| , |k_2-k_3|) \geq 5 \quad \mbox{and} \quad k_{\mathrm{med}}\geq -\frac{48m}{100}\quad \mbox{and} \quad k_{\min}\leq -\frac{49m}{100}.
\end{align}
\end{lem}

\begin{proof}[Proof of Lemma \ref{lemmainBO23}] This is similar to the situation in Lemma \ref{lemmainbound23}. The main observation is that we have the lower bound
\begin{align*}
| \Phi(\xi,\eta,\sigma) | \gtrsim 2^{k_{\mathrm{med}}},
\end{align*}
so we can integrate by parts in time and estimate the resulting integrals, as in the proof of Lemma \ref{lemmainbound23}.
\end{proof}

\begin{lem}\label{lemmainBO24}
  The bound \eqref{rak5} holds provided that \eqref{rak3} holds and, in addition,
\begin{equation}
\label{hyr44}
\max (|k-k_1|, |k- k_2|, |k-k_3|) \geq 21 \qquad \mbox{and} \qquad \max (|k_1 - k_3|, |k_2-k_3|) \leq 4.
\end{equation}
\end{lem}

\begin{proof}[Proof of Lemma \ref{lemmainBO24}]
This is similar to the proof of Lemma \ref{lemmainbound21}, using the observation that $k+10\leq \min(k_1,k_2,k_3)$, the strong lower bound
\begin{align*}
 |(\partial_\eta \Phi)(\xi,\eta,\sigma)| = |\Lambda'(\xi+\eta)-\Lambda'(\xi+\eta+\sigma)|\gtrsim 2^{k_{\max}/2},
\end{align*}
and integration by parts in $\eta$.
\end{proof}

\begin{lem}\label{lemmainBO25}
  The bound \eqref{rak5} holds provided that \eqref{rak3} holds and, in addition,
\begin{align}
\label{hyr45}
\max (|k-k_1|, |k- k_2|, |k-k_3|) \leq 20.
\end{align}
\end{lem}

\begin{proof}[Proof of Lemma \ref{lemmainBO25}]
This is the main case where there is a substantial difference between the integrals $I_{k_1,k_2,k_3}^{++-}$ and $I_{k_1,k_2,k_3}^{--+}$. The main point is that the phase function $\Phi^{--+}$ does not have any spacetime resonances, i.e. there are no $(\eta,\sigma)$ solutions of the equations
\begin{equation*}
\Phi^{--+}(\xi,\eta,\sigma)=(\partial_\eta\Phi^{--+})(\xi,\eta,\sigma)=(\partial_\sigma\Phi^{--+})(\xi,\eta,\sigma)=0.
\end{equation*}

For any $l,j\in\mathbb{Z}$ satisfying $l\leq j$ we define
\begin{equation*}
\varphi_j^{(l)}:=
\begin{cases}
\varphi_j\qquad&\text{ if }j\geq l+1,\\
\varphi_{\leq l}\qquad&\text{ if }j=l.
\end{cases}
\end{equation*}
Let $\overline{l}:=k-20$ and decompose
\begin{equation*}
I^{--+}_{k_1,k_2,k_3}=\sum_{l_1,l_2\in[k-20,k+40]}J_{l_1,l_2},
\end{equation*}
where
\begin{equation*}
J_{l_1,l_2}(\xi,t):=\int_{\mathbb{R}\times\mathbb{R}}e^{it\Phi(\xi,\eta,\sigma)}c^{\ast}_{\mathbf{k}}(\eta,\sigma)\varphi_{l_1}^{(\overline{l})}(\eta)\varphi_{l_2}^{(\overline{l})}(\sigma)\widehat{f^{-}_{k_1}}(\xi+\eta)
\widehat{f^{-}_{k_2}}(\xi+\sigma)\widehat{f^{+}_{k_3}}(-\xi-\eta-\sigma)\,d\eta d\sigma.
\end{equation*}

The contributions of the integrals $J_{l_1,l_2}$ for $(l_1,l_2)\neq(\overline{l},\overline{l})$ can be estimated by integration by parts either in $\eta$ or in $\sigma$ (depending on the relative sizes of $l_1$ and $l_2$), since the $(\eta,\sigma)$ gradient of the phase function $\Phi$ is bounded from below by $c2^{k/2}$ in the support of these integrals.

On the other hand, to estimate the contribution of the integral $J_{\overline{l},\overline{l}}$ we notice that $|\Phi(\xi,\eta,\sigma)|\gtrsim 2^{3k/2}$ in the support of the integral and integrate by parts in $s$. The result is
\begin{align}
\begin{split}
\label{sum1}
\Big| \int_{t_1}^{t_2} &e^{iL(\xi,s)} J_{\overline{l},\overline{l}}(\xi,s) \, ds \Big|\lesssim  | L_4 (\xi,t_1) | + | L_4 (\xi,t_2) |\\
&+ \int_{t_1}^{t_2} | L_1(\xi,s)|+ | L_2(\xi,s) | + |L_3(\xi,s)| + |(\partial_sL)(\xi,s)|| L_4(\xi,s) | \, ds,
\end{split}
\end{align}
where
\begin{align*}
\begin{split}
& L_1(\xi) := \int_{\mathbb{R}\times\mathbb{R}} e^{is\Phi(\xi,\eta,\sigma)} \frac{c^\ast_{\mathbf{k}}(\eta,\sigma)\varphi_{\leq \overline{l}}(\eta)\varphi_{\leq \overline{l}}(\sigma)}{i \Phi(\xi,\eta,\sigma)}
 \, (\partial_s\what{f_{k_1}^{-}})(\xi+\eta) \what{f_{k_2}^{-}}(\xi+\sigma) \what{f_{k_3}^{+}}(-\xi-\eta-\sigma) \,d\eta d\sigma ,
\\
& L_2(\xi) := \int_{\mathbb{R}\times\mathbb{R}} e^{is\Phi(\xi,\eta,\sigma)} \frac{c^\ast_{\mathbf{k}}(\eta,\sigma)\varphi_{\leq \overline{l}}(\eta)\varphi_{\leq \overline{l}}(\sigma)}{i \Phi(\xi,\eta,\sigma)}
  \what{f_{k_1}^{-}}(\xi+\eta)(\partial_s\what{f_{k_2}^{-}})(\xi+\sigma) \what{f_{k_3}^{+}}(-\xi-\eta-\sigma) \,d\eta d\sigma,
  \\
  & L_3(\xi) := \int_{\mathbb{R}\times\mathbb{R}} e^{is\Phi(\xi,\eta,\sigma)} \frac{c^\ast_{\mathbf{k}}(\eta,\sigma)\varphi_{\leq \overline{l}}(\eta)\varphi_{\leq \overline{l}}(\sigma)}{i \Phi(\xi,\eta,\sigma)}
  \what{f_{k_1}^{-}}(\xi+\eta)\what{f_{k_2}^{-}}(\xi+\sigma) (\partial_s\what{f_{k_3}^{+}})(-\xi-\eta-\sigma) \,d\eta d\sigma,
  \\
  & L_4(\xi) := \int_{\mathbb{R}\times\mathbb{R}} e^{is\Phi(\xi,\eta,\sigma)} \frac{c^\ast_{\mathbf{k}} (\eta,\sigma)\varphi_{\leq \overline{l}}(\eta)\varphi_{\leq \overline{l}}(\sigma)}{i \Phi(\xi,\eta,\sigma)}
  \what{f_{k_1}^{-}}(\xi+\eta) \what{f_{k_2}^{-}}(\xi+\sigma) \what{f_{k_3}^{+}}(-\xi-\eta-\sigma) \,d\eta d\sigma.
\end{split}
\end{align*}

To estimate the integrals $L_1, L_2, L_3, L_4$ we notice that
\begin{equation*}
\Big\|\frac{c^\ast_{\mathbf{k}} (\eta,\sigma)\varphi_{\leq \overline{l}}(\eta)\varphi_{\leq \overline{l}}(\sigma)}{i \Phi(\xi,\eta,\sigma)}\Big\|_{S^\infty}\lesssim 2^{3k}.
\end{equation*}
Therefore, using Lemma \ref{touse} (ii) and the bounds \eqref{rak1}--\eqref{rak1.5} and \eqref{prs},
\begin{align*}
|L_4(\xi,s)|\lesssim \Big\|\frac{c^\ast_{\mathbf{k}} (\eta,\sigma)\varphi_{\leq \overline{l}}(\eta)\varphi_{\leq \overline{l}}(\sigma)}{i\Phi(\xi,\eta,\sigma)}\Big\|_{S^\infty}
  {\| \what{f_{k_1}^{-}}(s) \|}_{L^2} {\| e^{is\Lambda} f_{k_2}^{-}(s) \|}_{L^\infty} {\| \what{f_{k_3}^{+}}(s) \|}_{L^2}\lesssim \eps_1^3 2^{-m/4},
\end{align*}
and
\begin{align*}
|L_1(\xi,s)|\lesssim \Big\|\frac{c^\ast_{\mathbf{k}} (\eta,\sigma)\varphi_{\leq \overline{l}}(\eta)\varphi_{\leq \overline{l}}(\sigma)}{i\Phi(\xi,\eta,\sigma)}\Big\|_{S^\infty}
  {\| (\partial_s\what{f_{k_1}^{-}})(s) \|}_{L^2} {\| e^{is\Lambda} f_{k_2}^{-}(s) \|}_{L^\infty} {\| \what{f_{k_3}^{+}}(s) \|}_{L^2}\lesssim \eps_1^3 2^{-5m/4}.
\end{align*}
The bounds on $|L_2(\xi,s)|$ and $L_3(\xi,s)|$ are similar to the bound on $|L_1(\xi,s)|$. Recalling also the bound $\big|(\partial_sL)(\xi,s) \big| \lesssim \e_1^2 2^{-m}$, see the definition \eqref{nf50}, it follows that the right-hand side of \eqref{sum1} is dominated by $C\eps_1^32^{-m/10}$. This completes the proof of the lemma.
\end{proof}

\subsection{Proof of \eqref{rak6}}\label{prooftech3}

We show now how to bound the quartic contributions $\mathcal{R}_{\geq 4}$. We rely mostly on elliptic estimates. Let
\begin{equation*}
\mathcal{N}_u:=Q_0(u+\overline{u},u),\qquad \mathcal{N}_v:=Q_0(v+\overline{v},v)
\end{equation*}
and recall the definition \eqref{nf42}
\begin{equation}\label{lmj0}
\begin{split}
\mathcal{R}_{\geq 4}=&[A(\mathcal{N}_u,u)-A(\mathcal{N}_v,v)]+[A(u,\mathcal{N}_u)-A(v,\mathcal{N}_v)]\\
+&[B(\overline{\mathcal{N}_u},u)-B(\overline{\mathcal{N}_v},v)]+[B(\overline{u},\mathcal{N}_u)-B(\overline{v},\mathcal{N}_v)].
\end{split}
\end{equation}

Recall the bounds, which hold for any $l\in\mathbb{Z}$ and $t\in[0,T]$,
\begin{equation}\label{lmj1}
\begin{split}
\|P_lu(t)\|_{L^2}+\|P_lv(t)\|_{L^2}&\lesssim\eps_1(1+t)^{p_0}\min\big[2^{-(N-1/2)l},2^{l(1/2-p_0)}\big],\\
\|P_lu(t)\|_{L^\infty}+\|P_lv(t)\|_{L^\infty}&\lesssim \eps_1(1+t)^{-1/2}2^{-4\max(l,0)},\\
\|P_lSu(t)\|_{L^2}+\|P_lSv(t)\|_{L^2}&\lesssim\eps_1(1+t)^{4p_0}\min\big[2^{-3l/2},2^{l(1/2-p_0)}\big],
\end{split}
\end{equation}
and
\begin{equation}\label{lmj2}
\begin{split}
\|P_l(u(t)-v(t))\|_{L^2}&\lesssim\eps_1^2(1+t)^{-3/8+6p_0}2^{-(N+3)k/2},\\
\|P_l(u(t)-v(t))\|_{L^\infty}&\lesssim\eps_1^2(1+t)^{-3/4+2p_0}2^{-7l/2},\\
\|P_lS(u(t)-v(t))\|_{L^2}&\lesssim\eps_1^2(1+t)^{-1/4+6p_0}2^{-3l/2},\\
P_l(u(t)-v(t))&=0\qquad\text{ if }l\leq -30.
\end{split}
\end{equation}
see \eqref{pr5}--\eqref{pr7}, Lemma \ref{Lem1}, and Remark \ref{extra1}. We prove first similar bounds on the functions $\mathcal{N}_u,\mathcal{N}_v,\mathcal{N}_u-\mathcal{N}_v$.

\begin{lem}\label{LemmaLMJ}
For any $t\in[0,T]$ and $l\in\mathbb{Z}$ we have
\begin{equation}\label{lmj3}
\begin{split}
\|P_l\mathcal{N}_u(t)\|_{L^2}+\|P_l\mathcal{N}_v(t)\|_{L^2}&\lesssim \eps_1^2(1+t)^{p_0-1/2}2^{-(N-3/2)l},\\
\|P_l\mathcal{N}_u(t)\|_{L^\infty}+\|P_l\mathcal{N}_v(t)\|_{L^\infty}&\lesssim \eps_1^2(1+t)^{-1}2^{-3l},\\
\|P_lS\mathcal{N}_u(t)\|_{L^2}+\|P_lS\mathcal{N}_v(t)\|_{L^2}&\lesssim\eps_1^2(1+t)^{4p_0-1/2}2^{-l/2},\\
P_l\mathcal{N}_u(t)=P_l\mathcal{N}_v(t)&=0\qquad\text{ if }l\leq -30.
\end{split}
\end{equation}
and
\begin{equation}\label{lmj4}
\begin{split}
\|P_l(\mathcal{N}_u(t)-\mathcal{N}_v(t))\|_{L^2}&\lesssim \eps_1^3(1+t)^{-7/8+6p_0}2^{-3l},\\
\|P_l(\mathcal{N}_u(t)-\mathcal{N}_v(t))\|_{L^\infty}&\lesssim\eps_1^3(1+t)^{-5/4+2p_0}2^{-5l/2},\\
\|P_lS(\mathcal{N}_u(t)-\mathcal{N}_v(t))\|_{L^2}&\lesssim\eps_1^3(1+t)^{-3/4+8p_0}2^{-l/2}.
\end{split}
\end{equation}
\end{lem}

\begin{proof}[Proof of Lemma \ref{LemmaLMJ}] We use Lemma \ref{touse} (ii) and the $S^\infty$ bound \eqref{en21.2},
\begin{equation}\label{lmj5}
\|q_0^{k,k_1,k_2}\|_{S^\infty}\lesssim 2^{k_1/2}2^{k_2}\one_{15}(k,k_1,k_2).
\end{equation}
We also use the formulas (compare with \eqref{nf15.9})
\begin{equation}\label{lmj6}
\begin{split}
&S[Q_0(u+\overline{u},u)]=Q_0(Su+S\overline{u},u)+Q_0(u+\overline{u},Su)-\widetilde{Q}_0(u+\overline{u},u),\\
&S[Q_0(v+\overline{v},v)]=Q_0(Sv+S\overline{v},v)+Q_0(v+\overline{v},Sv)-\widetilde{Q}_0(v+\overline{v},v),
\end{split}
\end{equation}
where $\widetilde{Q}_0$ is the bilinear operator associated to the multiplier
\begin{equation}\label{lmj7}
\widetilde{q}_0(\xi,\eta):=(\xi\partial_\xi+\eta\partial_\eta)q_0(\xi,\eta).
\end{equation}

We estimate, as in Lemma \ref{NonlinEst}, for $l\geq -30$,
\begin{equation*}
\begin{split}
\|P_l\mathcal{N}_u(t)\|_{L^2}&\lesssim \sum_{k_1,k_2\in\mathbb{Z}}\|q_0^{l,k_1,k_2}\|_{S^\infty}
\|P'_{k_1}u(t)\|_{L^\infty}\|P'_{k_2}u(t)\|_{L^2}\lesssim \eps_1^2(1+t)^{p_0-1/2}2^{-(N-3/2)l},
\end{split}
\end{equation*}
and similarly
\begin{equation*}
\begin{split}
\|P_l\mathcal{N}_u(t)\|_{L^\infty}&\lesssim \sum_{k_1,k_2\in\mathbb{Z}}\|q_0^{l,k_1,k_2}\|_{S^\infty}
\|P'_{k_1}u(t)\|_{L^\infty}\|P'_{k_2}u(t)\|_{L^\infty}\lesssim \eps_1^2(1+t)^{-1}2^{-3l}.
\end{split}
\end{equation*}

Similarly we estimate
\begin{equation*}
\begin{split}
\|P_lQ_0(u+\overline{u},Su)(t)\|_{L^2}\lesssim \sum_{k_1,k_2\in\mathbb{Z}}\|q_0^{l,k_1,k_2}\|_{S^\infty}
\|P'_{k_1}u(t)\|_{L^\infty}\|P'_{k_2}Su(t)\|_{L^2}\lesssim \eps_1^2(1+t)^{4p_0-1/2}2^{-l/2},
\end{split}
\end{equation*}
and
\begin{equation*}
\begin{split}
\|P_lQ_0(Su+S\overline{u},u)(t)\|_{L^2}&\lesssim \sum_{k_1,k_2\in\mathbb{Z}}\|q_0^{l,k_1,k_2}\|_{S^\infty}
\|P'_{k_1}Su(t)\|_{L^2}\|P'_{k_2}u(t)\|_{L^\infty}\\
&\lesssim \eps_1^2(1+t)^{4p_0-1/2}2^{-3l}.
\end{split}
\end{equation*}
By homogeneity, the symbol $\widetilde{q}_0$ satisfies the same $S^\infty$ bounds as the symbol $q_0$, see \eqref{lmj5}. Therefore, as before
\begin{equation*}
\begin{split}
\|P_l\widetilde{Q}_0(u+\overline{u},u)(t)\|_{L^2}\lesssim\eps_1^2(1+t)^{p_0-1/2}2^{-(N-3/2)l}.
\end{split}
\end{equation*}
The last three inequalities and the formulas \eqref{lmj6} show that
\begin{equation*}
\|P_lS\mathcal{N}_u(t)\|_{L^2}\lesssim \eps_1^2(1+t)^{4p_0-1/2}2^{-l/2}.
\end{equation*}
Similar estimates hold for the function $\mathcal{N}_v$ (since the bounds for $u$ and $v$ in \eqref{lmj1} are identical), so the bounds \eqref{lmj3} follow.

We prove now the bounds \eqref{lmj4}. In view of the definition,
\begin{equation}\label{lmj8}
\mathcal{N}_u-\mathcal{N}_v=Q_0(u+\overline{u},u-v)+Q_0(u+\overline{u}-v-\overline{v},v).
\end{equation}
Then we estimate, using \eqref{lmj1}--\eqref{lmj2},
\begin{equation*}
\begin{split}
\|P_lQ_0(u+\overline{u},u-v)(t)\|_{L^2}&\lesssim \sum_{k_1,k_2\in\mathbb{Z}}\|q_0^{l,k_1,k_2}\|_{S^\infty}
\|P'_{k_1}u(t)\|_{L^\infty}\|P'_{k_2}(u-v)(t)\|_{L^2}\\
&\lesssim \eps_1^3(1+t)^{-7/8+6p_0}2^{-Nl/2},
\end{split}
\end{equation*}
and
\begin{equation*}
\begin{split}
\|P_lQ_0(u+\overline{u}-v-\overline{v},v)(t)\|_{L^2}&\lesssim \sum_{k_1,k_2\in\mathbb{Z}}\|q_0^{l,k_1,k_2}\|_{S^\infty}
\|P'_{k_1}(u-v)(t)\|_{L^2}\|P'_{k_2}v(t)\|_{L^\infty}\\
&\lesssim \eps_1^3(1+t)^{-7/8+6p_0}2^{-3l}.
\end{split}
\end{equation*}
The bound in the first line of \eqref{lmj4} follows.

The proof of the bound in the second line of \eqref{lmj4} is similar, using also the identity \eqref{lmj8}.
The proof of the last bound in \eqref{lmj4} is also similar, using both identities \eqref{lmj6} and \eqref{lmj8}.
\end{proof}

We prove now similar bounds on the function $\mathcal{R}_{\geq 4}$.

\begin{lem}\label{LemmaLMJ2}
For any $t\in[0,T]$ and $l\in\mathbb{Z}$ we have
\begin{equation}\label{lmj10}
\begin{split}
\|P_l\mathcal{R}_{\geq 4}(t)\|_{L^2}&\lesssim\eps_1^4(1+t)^{-9/8+10p_0}2^{-2l},\\
\|P_lS\mathcal{R}_{\geq 4}(t)\|_{L^2}&\lesssim\eps_1^4(1+t)^{-1+20p_0},\\
P_l\mathcal{R}_{\geq 4}(t)&=0\qquad\text{ if }l\leq -30.
\end{split}
\end{equation}
\end{lem}

\begin{proof}[Proof of Lemma \ref{LemmaLMJ2}] We examine the formula \eqref{lmj0} 
and concentrated on the term $[A(u,\mathcal{N}_u)-A(v,\mathcal{N}_v)]$. We rewrite
\begin{equation*}
A(u,\mathcal{N}_u)-A(v,\mathcal{N}_v)=A(u-v,\mathcal{N}_u)+A(v,\mathcal{N}_u-\mathcal{N}_v).
\end{equation*}
Then we estimate, using \eqref{nf12} and \eqref{lmj1}--\eqref{lmj4},
\begin{equation*}
\begin{split}
\big\|P_lA(u-v,\mathcal{N}_u)(t)\big\|_{L^2}&\lesssim \sum_{k_1,k_2\in\mathbb{Z}}\|a^{l,k_1,k_2}\|_{S^\infty}
\|P'_{k_1}(u-v)(t)\|_{L^\infty}\|P'_{k_2}\mathcal{N}_u(t)\|_{L^2}\\
&\lesssim \eps_1^4(1+t)^{-5/4+6p_0}2^{-4l}
\end{split}
\end{equation*}
and
\begin{equation*}
\begin{split}
\big\|P_lA&(v,\mathcal{N}_u-\mathcal{N}_v)(t)\big\|_{L^2}\lesssim \sum_{k_1,k_2\in\mathbb{Z}}\|a^{l,k_1,k_2}\|_{S^\infty}
\|P'_{k_1}v(t)\|_{L^\infty}\|P'_{k_2}(\mathcal{N}_u-\mathcal{N}_v)(t)\|_{L^2}\\
&\lesssim \eps_1^4(1+t)^{-7/8+6p_0}\sum_{k_1,k_2\in\mathbb{Z}}\one_{15}(k,k_1,k_2)2^{(k_2-k_1)/2}
2^{-3k_2}\min[(1+t)^{-1/2},(1+t)^{p_0}2^{k_1(1-p_0)}]\\
&\lesssim \eps_1^4(1+t)^{-9/8+10p_0}2^{-2l}.
\end{split}
\end{equation*}
Therefore
\begin{equation}\label{lmj12}
\big\|P_l\big[A(u,\mathcal{N}_u)-A(v,\mathcal{N}_v)\big](t)\big\|_{L^2}\lesssim \eps_1^4(1+t)^{-9/8+10p_0}2^{-2l}.
\end{equation}

Similarly, using also \eqref{nf15.9},
\begin{equation*}
\begin{split}
S[A(u,\mathcal{N}_u)-A(v,\mathcal{N}_v)]&=A(S(u-v),\mathcal{N}_u)+A(u-v,S\mathcal{N}_u)-\widetilde{A}(u-v,\mathcal{N}_u)\\
&+A(Sv,\mathcal{N}_u-\mathcal{N}_v)+A(v,S(\mathcal{N}_u-\mathcal{N}_v))-\widetilde{A}(v,\mathcal{N}_u-\mathcal{N}_v).\\
\end{split}
\end{equation*}
Then we estimate, using \eqref{nf12} and \eqref{lmj1}--\eqref{lmj4},
\begin{equation*}
\begin{split}
\big\|P_lA(S(u-v),\mathcal{N}_u)(t)\big\|_{L^2}&\lesssim \sum_{k_1,k_2\in\mathbb{Z}}\|a^{l,k_1,k_2}\|_{S^\infty}
\|P'_{k_1}S(u-v)(t)\|_{L^2}\|P'_{k_2}\mathcal{N}_u(t)\|_{L^\infty}\\
&\lesssim \eps_1^4(1+t)^{-5/4+10p_0}2^{-2l},
\end{split}
\end{equation*}
\begin{equation*}
\begin{split}
\big\|P_lA(u-v,S\mathcal{N}_u)(t)\big\|_{L^2}&\lesssim \sum_{k_1,k_2\in\mathbb{Z}}\|a^{l,k_1,k_2}\|_{S^\infty}
\|P'_{k_1}(u-v)(t)\|_{L^\infty}\|P'_{k_2}S\mathcal{N}_u(t)\|_{L^2}\\
&\lesssim \eps_1^4(1+t)^{-5/4+10p_0}2^{-2l},
\end{split}
\end{equation*}
\begin{equation*}
\begin{split}
\big\|P_lA(Sv,\mathcal{N}_u-&\mathcal{N}_v)(t)\big\|_{L^2}\lesssim \sum_{k_1,k_2\in\mathbb{Z},\,2^{k_1}\geq (1+t)^{-4}}\|a^{l,k_1,k_2}\|_{S^\infty}
\|P'_{k_1}Sv(t)\|_{L^2}\|P'_{k_2}(\mathcal{N}_u-\mathcal{N}_v)(t)\|_{L^\infty}\\
&+\sum_{k_1,k_2\in\mathbb{Z},\,2^{k_1}\leq (1+t)^{-4}}\|a^{l,k_1,k_2}\|_{S^\infty}
2^{k_1/2}\|P'_{k_1}Sv(t)\|_{L^2}\|P'_{k_2}(\mathcal{N}_u-\mathcal{N}_v)(t)\|_{L^2}\\
&\lesssim \eps_1^4(1+t)^{-5/4+10p_0},
\end{split}
\end{equation*}
and
\begin{equation*}
\begin{split}
\big\|&P_lA(v,S(\mathcal{N}_u-\mathcal{N}_v))(t)\big\|_{L^2}\lesssim \sum_{k_1,k_2\in\mathbb{Z}}\|a^{l,k_1,k_2}\|_{S^\infty}
\|P'_{k_1}v(t)\|_{L^\infty}\|P'_{k_2}S(\mathcal{N}_u-\mathcal{N}_v)(t)\|_{L^2}\\
&\lesssim \eps_1^4(1+t)^{-3/4+8p_0}\sum_{k_1,k_2\in\mathbb{Z}}\one_{15}(k,k_1,k_2)2^{(k_2-k_1)/2}2^{-k_2/2}\min[(1+t)^{-1/2},(1+t)^{p_0}2^{k_1(1-p_0)}]\\
&\lesssim \eps_1^4(1+t)^{-1+20p_0}.
\end{split}
\end{equation*}
Moreover, as in \eqref{lmj12},
\begin{equation*}
\big\|P_l\big[\widetilde{A}(u,\mathcal{N}_u)-\widetilde{A}(v,\mathcal{N}_v)\big](t)\big\|_{L^2}\lesssim \eps_1^4(1+t)^{-9/8+10p_0}2^{-2l},
\end{equation*}
since $a$ and $\widetilde{a}$ satisfy identical symbol-type estimates. Therefore,
\begin{equation}\label{lmj13}
\big\|P_lS\big[A(u,\mathcal{N}_u)-A(v,\mathcal{N}_v)\big](t)\big\|_{L^2}\lesssim \eps_1^4(1+t)^{-1+20p_0}.
\end{equation}

The term $B(\overline{u},\mathcal{N}_u)-B(\overline{v},\mathcal{N}_v)$ in \eqref{lmj0} can be estimated in a similar way, since $a$ and $b$ satisfy identical symbol-type estimates. The terms $A(\mathcal{N}_u,u)-A(\mathcal{N}_v,v)$ and $B(\overline{\mathcal{N}_u},u)-B(\overline{\mathcal{N}_v},v)$ are easier to estimate because there are no low frequency contributions. The desired conclusion of the lemma follows.
\end{proof}

We can now complete the proof of the main estimate \eqref{rak6}.
\begin{lem}\label{LemmaLMJ3}
Assume that $k\in[-30,20p_0m]$, $|\xi_0|\in [2^k,2^{k+1}]$, $m\geq 1$, $t_1\leq t_2\in[2^{m}-2,2^{m+1}]\cap[0,T']$. Then
\begin{equation}\label{lmj20}
\Big|\varphi_k(\xi_0)\int_{t_1}^{t_2}e^{iL(\xi_0,s)}e^{is|\xi_0|^{3/2}}\widehat{\mathcal{R}_{\geq 4}}(\xi_0,s)\,ds\Big|\lesssim \eps_1^32^{-200p_0m}.
\end{equation}
\end{lem}

\begin{proof}[Proof of Lemma \ref{LemmaLMJ3}] Let
\begin{equation}\label{lmj21}
F(\xi):=\varphi_k(\xi)\int_{t_1}^{t_2}e^{iL(\xi_0,s)}e^{is|\xi|^{3/2}}\widehat{\mathcal{R}_{\geq 4}}(\xi,s)\,ds.
\end{equation}
In view of Lemma \ref{interpolation}, it suffices to prove that
\begin{equation*}
2^{-k}\|F\|_{L^2}\big[2^k\|\partial F\|_{L^2}+\|F\|_{L^2}\big]\lesssim \eps_1^62^{-400p_0m}.
\end{equation*}
Since $\|F\|_{L^2}\lesssim \eps_1^42^{-(1/8-10p_0)m}2^{-2k}$, see the first inequality in \eqref{lmj10}, it suffices to prove that
\begin{equation}\label{lmj22}
2^k\|\partial F\|_{L^2}\lesssim \eps_1^42^{3k}2^{(1/8-500p_0)m}.
\end{equation}

To prove \eqref{lmj22} we write
\begin{equation*}
|\xi\partial_\xi F(\xi)|\leq |F_1(\xi)|+|F_2(\xi)|+|F_3(\xi)|,
\end{equation*}
where
\begin{equation*}
\begin{split}
&F_1(\xi):=\xi(\partial_\xi\varphi_k)(\xi)\int_{t_1}^{t_2}e^{iL(\xi_0,s)}\big[e^{is|\xi|^{3/2}}\widehat{\mathcal{R}_{\geq 4}}(\xi,s)\big]\,ds,\\
&F_2(\xi):=\varphi_k(\xi)\int_{t_1}^{t_2}e^{iL(\xi_0,s)}\big[\xi\partial_\xi-(3/2)s\partial_s\big]\big[e^{is|\xi|^{3/2}}\widehat{\mathcal{R}_{\geq 4}}(\xi,s)\big]\,ds,\\
&F_3(\xi):=\frac{3}{2}\varphi_k(\xi)\int_{t_1}^{t_2}e^{iL(\xi_0,s)}s\partial_s\big[e^{is|\xi|^{3/2}}\widehat{\mathcal{R}_{\geq 4}}(\xi,s)\big]\,ds.
\end{split}
\end{equation*}
Using \eqref{lmj10} and the commutation identity $\big[\big[\xi\partial_\xi-(3/2)s\partial_s\big], e^{is|\xi|^{3/2}}\big]=0$, we have
\begin{equation*}
\|F_1\|_{L^2}+\|F_2\|_{L^2}\lesssim \eps_1^42^{3k}2^{30p_0m}.
\end{equation*}
Moreover, using integration by parts in $s$ and the bound $\big|\partial_s\big[e^{iL(\xi_0,s)}\big]\big|\lesssim 2^{-m}$, see the definition \eqref{nf50}, we can also estimate $\|F_3\|_{L^2}\lesssim \eps_1^42^{3k}2^{30p_0m}$. The desired bound \eqref{lmj22} follows.
\end{proof}

\end{document}